\documentclass{CUP-JNL-FMS}%

\usepackage{amsmath, amsthm, amssymb,bbm, stmaryrd, tikz,mathrsfs}
\usepackage[normalem]{ulem}
\usepackage[utf8]{inputenc} 
\usepackage{graphicx} 
\usepackage[small]{caption}
\usetikzlibrary{arrows,calc}
\numberwithin{equation}{section}
\usepackage{lmodern} 
\usepackage[T1]{fontenc} 
\usepackage{enumerate} 
\usepackage{verbatim} 
\usepackage{color}
\usepackage{tabularx}
\usepackage[shortlabels]{enumitem}
\usepackage[colorlinks,allcolors=weblmcolor]{hyperref}
\usepackage[capitalise]{cleveref}
\usepackage{xcolor}

\newcommand{\Prob}{\mathbb{P}}
\newcommand{\Q}{\mathbb{Q}}
\newcommand{\R}{\mathbb{R}}
\newcommand{\E}{\mathbb{E}}

\newcommand{\N}{\mathbb{N}}
\newcommand{\Z}{\mathbb{Z}}

\newcommand{\bP}{\mathbf{P}}
\newcommand{\bE}{\mathbf{E}}

\newcommand{\fR}{\mathfrak{R}}
\newcommand{\fM}{\mathfrak{M}}

\newcommand{\fA}{\mathcal{A}}
\newcommand{\fC}{\mathfrak{C}}
\newcommand{\cC}{\mathcal{C}}
\newcommand{\cI}{\mathcal{I}}

\newcommand{\cP}{\mathcal P}

\newcommand{\cE}{\mathcal{E}}
\newcommand{\cG}{\mathcal{G}}
\newcommand{\cS}{\mathcal{S}}

\newcommand{\calN}{\mathcal{N}}
\renewcommand{\d}{\mathrm{d}}
\renewcommand{\P}{\Prob}
\newcommand{\bbS}{\mathbb{S}}
\newcommand{\ind}[1]{\mathbbm{1}_{\{ #1\}}}
\newcommand{\indset}[1]{\mathbbm{1}_{#1}}
\newcommand{\Cpts}{\mathsf{Cpts}}

\newtheorem{theorem}{Theorem}[section]
\newtheorem*{theorem*}{Theorem}
\newtheorem{lemma}[theorem]{Lemma}
\newtheorem{claim}[theorem]{Claim}
\newtheorem{proposition}[theorem]{Proposition}
\newtheorem{observation}[theorem]{Observation}
\newtheorem{corollary}[theorem]{Corollary}

\theoremstyle{definition}{

\newtheorem{definition}[theorem]{Definition}
\newtheorem*{definition*}{Definition}

\newtheorem{question}[theorem]{Question}
\newtheorem*{question*}{Question}
\newtheorem*{example*}{Example}
\newtheorem*{examples*}{Examples}
\newtheorem{remark}[theorem]{Remark}
\newtheorem*{remark*}{Remark}

}
\numberwithin{equation}{section}

\newcommand{\norm}[1]{\left\Vert #1 \right\Vert}

\newcommand{\Cov}{\operatorname{Cov}}
\newcommand{\Var}{\operatorname{Var}}
\newcommand{\one}{\mathbbm{1}}
\newcommand{\given}{\;\big|\;}
\newcommand\abs[1]{\left|#1\right|}

\newcommand{\floor}[1]{\left\lfloor#1\right\rfloor}

\newcommand{\llb}{\llbracket}
\newcommand{\rrb}{\rrbracket}

\newcommand{\ep}{\epsilon}
\newcommand{\ostar}{\mathsf{o}^*}
\renewcommand{\e}{\mathsf{e}} 
\newcommand{\w}{\mathsf{w}}
\newcommand{\x}{\mathsf{x}} 
\newcommand{\y}{\mathsf{y}}
\newcommand{\z}{\mathsf{z}}
\renewcommand{\u}{\mathsf{u}}
\renewcommand{\v}{\mathsf{v}}
\newcommand{\X}{\mathsf{X}} 
\newcommand{\sS}{\mathsf{S}}

\newcommand{\h}{\mathsf{h}}
\newcommand{\fJ}{\mathfrak{J}}

\newcommand{\C}{\underline{\cC}} 
\newcommand{\A}{\mathsf{A}} 
\newcommand{\AL}{\mathsf{A_L}} 
\newcommand{\AR}{\mathsf{A_R}} 
\renewcommand{\H}{\mathbb H}
\newcommand{\Pa}{\P^{\h_{\x}}}

\newcommand{\vv}{\vec{v}}

\newcommand{\cL}{\mathcal L}
\newcommand{\Y}{\mathcal{Y}}
\newcommand{\V}{\mathsf{V}}
\newcommand{\Lstar}{\mathsf{L}*}
\newcommand{\Rstar}{\mathsf{R}*}

\newcommand{\legs}{\texttt{\textup{legs}}}
\newcommand{\full}{\texttt{\textup{full}}}
\newcommand{\tv}{\textsc{\textup{tv}}}
\newcommand{\pilegs}{\pi_{W}^{\legs}}
\newcommand{\hatpilegs}{\widehat\pi_{W}^{\legs}}
\newcommand{\n}{\vec{n}}
\newcommand{\SOS}{\textsc{sos}}
\newcommand{\Gb}{\cG}
\newcommand{\fh}{\mathfrak h}
\newcommand{\cp}{\mathrm{cp}}
\newcommand{\len}{\mathrm{len}}

\newcommand{\fcone}{\mathcal{Y}^\blacktriangleleft}
\newcommand{\bcone}{\mathcal{Y}^\blacktriangleright}

\newcommand{\Kb}{\mathcal W}
\newcommand{\wulffcone}{\Kb_{\textup{in}}^\blacktriangleleft}
\newcommand{\extwulffcone}{\Kb_{\textup{out}}^{\blacktriangleleft}}

\newcommand{\np}{\vec{n}^{\perp}}
\newcommand{\bQ}{\mathbf{Q}}
\newcommand{\bS}{\mathbf{S}}
\newcommand{\bX}{\mathbf{X}}

\makeatletter
\crefname{step}{Step}{Steps}
\crefname{property}{Property}{Properties}
\crefname{case}{Case}{Cases}
\crefname{question}{Question}{Questions}
\makeatother

\articletype{RESEARCH ARTICLE}

\begin{document}

\begin{Frontmatter}

\title[On level line fluctuations of SOS surfaces above a wall]{On level line fluctuations of SOS surfaces above a wall}

\author[1]{Patrizio Caddeo}\orcid{0009-0004-1018-6950}
\author[2]{Yujin H.\ Kim}\orcid{0000-0001-8457-2548}
\author[3]{Eyal Lubetzky}\orcid{0000-0002-2281-3542}

\authormark{P.\ Caddeo \textit{et al}.}

\address[1]{\orgname{Courant Institute of Mathematical Sciences, New York University}, \orgaddress{\street{251 Mercer Street, New York, NY}, \postcode{10012}, \country{USA}};\email{patrizio.caddeo@courant.nyu.edu}}
\address[1]{\orgname{Courant Institute of Mathematical Sciences, New York University}, \orgaddress{\street{251 Mercer Street, New York, NY}, \postcode{10012}, \country{USA}};\email{yujin.kim@courant.nyu.edu}}
\address[1]{\orgname{Courant Institute of Mathematical Sciences, New York University}, \orgaddress{\street{251 Mercer Street, New York, NY}, \postcode{10012}, \country{USA}};\email{eyal@courant.nyu.edu}}




\abstract{We study the low temperature $(2+1)$D Solid-On-Solid model on $\llbracket 1, L \rrbracket^2$ with zero boundary conditions and non-negative heights (a floor at height $0$). Caputo et al.\ (2016) established that this random surface typically admits either $\mathfrak h $ or $\mathfrak h+1$ many nested macroscopic level line loops $\{\mathcal L_i\}_{i\geq 0}$ for an explicit $\mathfrak h\asymp \log L$, and its top loop $\mathcal L_0$ has cube-root fluctuations: e.g., if $\rho(x)$ is the vertical displacement of $\mathcal L_0$ from the bottom boundary point $(x,0)$, then
$\max \rho(x) = L^{1/3+o(1)}$ over~$x\in I_0:=L/2+\llbracket -L^{2/3},L^{2/3}\rrbracket$. It is believed that rescaling $\rho$ by $L^{1/3}$ and $I_0$ by~$L^{2/3}$ would yield a limit law of a diffusion on $[-1,1]$. However, no nontrivial lower bound was known on $\rho(x)$ for a  fixed $x\in I_0$ (e.g., $x=\frac L2$), let alone on $\min\rho(x)$ in $I_0$, to complement the bound on~$\max\rho(x)$. Here we show a lower bound of the predicted order $L^{1/3}$: for every $\epsilon>0$ there exists $\delta>0$ such that $\min_{x\in I_0} \rho(x) \geq \delta L^{1/3}$ with probability at least $1-\epsilon$. The proof relies on  the Ornstein--Zernike machinery due to Campanino--Ioffe--Velenik, and a result of Ioffe, Shlosman and Toninelli (2015) that rules out pinning in Ising polymers with modified interactions along the boundary. En route, we refine the latter result into a Brownian excursion limit law, which may be of independent interest.
We further show that in a $ K L^{2/3}\times K L^{2/3}$ box with boundary conditions $\mathfrak h-1,\mathfrak h,\mathfrak h,\mathfrak h$ (i.e., $\mathfrak h-1$ on the bottom side and $\mathfrak h$ elsewhere), the limit of  $\rho(x)$ as $K,L\to\infty$ is a Ferrari--Spohn diffusion.}

\end{Frontmatter}

\localtableofcontents

\vspace*{14pt}

\section{Introduction}
We consider the Solid-On-Solid (SOS) model on $\Lambda_L =\llb1,L\rrb^2$, an $L\times L$ square in $\Z^2$, at large inverse-temperature $\beta>0$, with zero boundary conditions and a floor at height $0$: denoting by
$\x\sim \y$ a pair of adjacent sites $\x,\y\in\Z^2$, and setting $\varphi_\x = 0$ for all $\x\notin\Lambda_L$, the model assigns a height function $\varphi:\Lambda_L \to \Z_{\geq 0}$ (taking nonnegative integer heights) the probability 
\begin{equation}\label{eq:def-sos-measure}\pi^0_{\Lambda_L}(\varphi) \propto \exp\Big(-\beta\sum_{\x\sim \y} |\varphi_\x-\varphi_\y|\Big)\,.
\end{equation}
The model was introduced in the early 1950's (see~\cite{BCF51,Temperley52}) to approximate the formation of crystals and the interface separating the plus and minus phases in the low temperature 3D Ising model.

While of interest in any dimension $d$, the study of the model on $\Z^2$ has special importance, as it is the only dimension associated with the roughening phase transition. For the low temperature 3D Ising model, which the $(2+1)$D SOS model approximates for large $\beta$, rigorously establishing the roughening phase transition is 
a tantalizing open problem which has seen very little progress since being observed some 50 years ago (numerical experiments suggest it takes place at $\beta_{\textsc r}\approx 0.408$, compared to the critical 3D Ising temperature $\beta_c \approx 0.221$). The corresponding phase transition for the $(2+1)$D SOS \emph{with no floor} $\widehat\pi$ (where $\varphi$ can be negative) was rigorously confirmed as follows:
(i)~(\emph{localization}) for $\beta$ large enough, the surface is rigid, in that $\Var(\varphi_\x)=O(1)$ at $\x$ in the bulk, and furthermore $|\varphi_\x|$ has an exponential tail~\cite{BW82};
(ii)~(\emph{delocalization}) for $\beta$ small enough, 
Fr\"ohlich and Spencer~\cite{FrSp81a,FrSp81b} famously showed that 
$\Var(\varphi_\x) \asymp\log L$, just as in the case where $\varphi$ takes values in~$\R$.
(iii)~Very recently, Lammers~\cite{lammers22} showed the phase transition in $\Var(\varphi_\x)$ is sharp: there exists $\beta_{\textsc r}>0$ such that $\Var(\varphi_\x)\to\infty$ for all $\beta \leq \beta_{\textsc r}$ whereas it is $O(1)$ for all $\beta>\beta_{\textsc r}$; numerical experiments suggest that $\beta_{\textsc r}\approx 0.806$.
(See~\cite{lammers22} for additional details on the recent developments in the SOS model with no floor and related models of integer-valued height functions.)

Our setting is the low-temperature regime ($\beta$ large), yet with the restriction that the surface must lie above above a hard wall (the assumption $\varphi \geq 0$). Bricmont, El-Mellouki and Fr\"ohlich~\cite{BEF86} showed that this induces \emph{entropic repulsion}, regarded as a key feature of the physics of random surfaces: the restriction $\varphi \geq 0$ propels the surface (despite the energy cost) so as to gain entropy. Namely, it was shown in~\cite{BEF86} that $ \frac{c}\beta\log L \leq \E[\varphi_\x \mid \varphi \geq 0] \leq \frac{C}\beta \log L$ for absolute constants $c,C>0$.

The gap between these bounds was closed in~\cite{CLMST14}, where it was established that 
$\E[\varphi_\x \mid \varphi\geq 0]$ is $\frac1{4\beta}\log L + O(1)$, and moreover, $(1-\epsilon_\beta)L^2$ sites are at such a height with high probability (w.h.p.). The following intuition explains the height asymptotics: if the surface lies rigid about height $h$, then the cost of raising every site by $1$ is $4\beta L$ (incurred at the sites along the boundary); the benefit in doing so would be to gain the ability to feature spikes of depth $h+1$ (forbidden at level $h$ due to the restriction $\varphi\geq 0$), and as such a spike has an energetic cost of about $e^{-4\beta h}$, the entropy gain is about $\mathsf H(e^{-4\beta h})L^2$ where $\mathsf H(\cdot)$ is the Shannon entropy; the two terms 
are equated at $h\sim \frac1{4\beta}\log L$. 

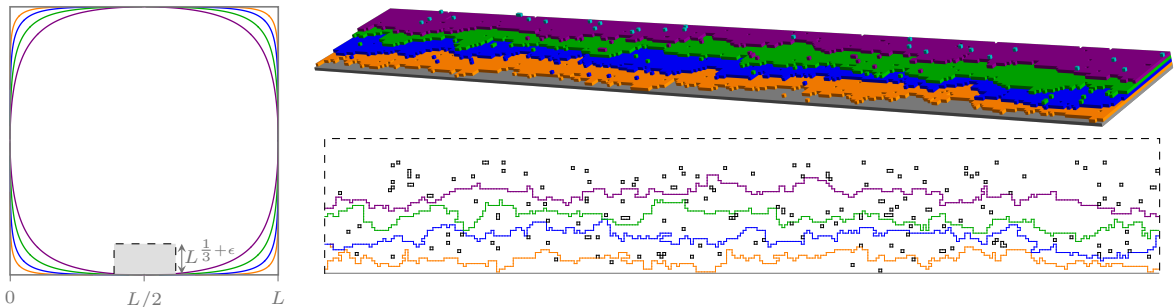
\begin{figure}
\vspace{-0.15in}
    \centering
    \begin{tikzpicture}
   \node (fig1) at (7.2,0.85) {    \includegraphics[width=.7\textwidth]{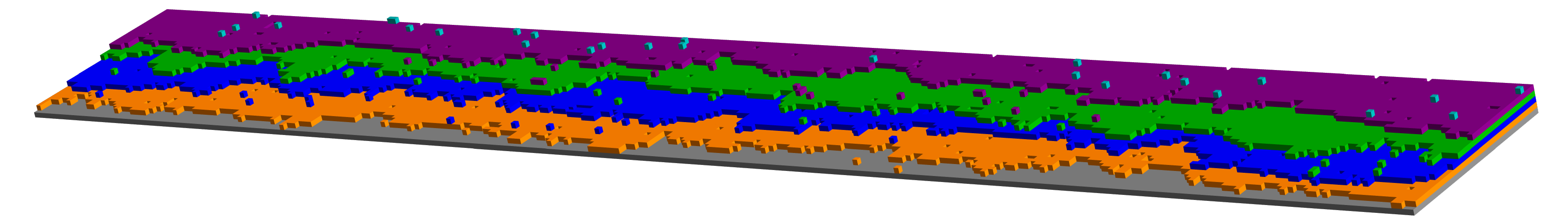}};

    \node (fig2) at (7.2,-1.15) {    \includegraphics[width=.68\textwidth]{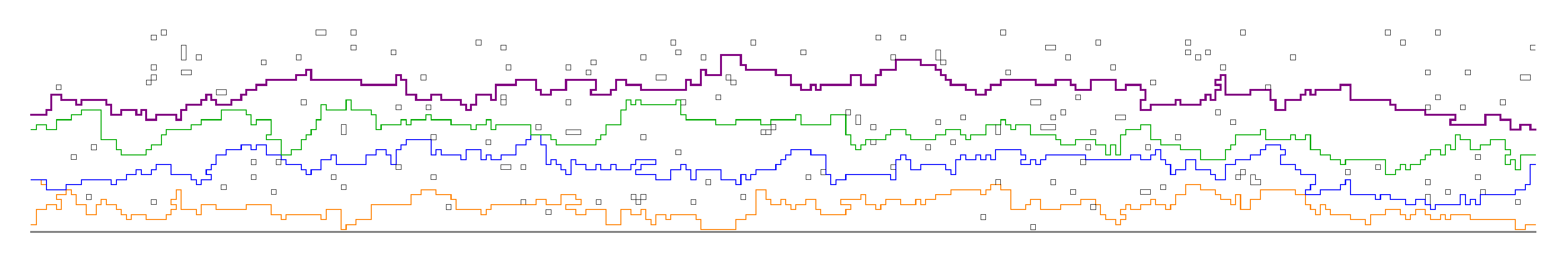}};

    \node (fig3) at (0,-0.13) {    \includegraphics[width=.274\textwidth]{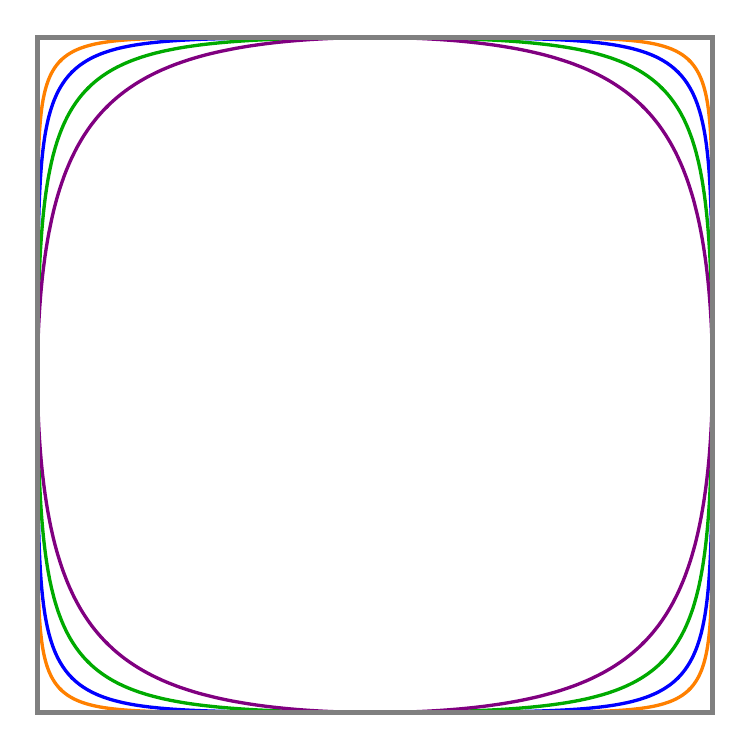}};

   \draw[black,dashed,fill=gray!25!white] (-.4,-1.9) -- (-.4,-1.5) -- (.42,-1.5) -- (.42,-1.9);

   \draw[stealth-stealth,gray!90!black,font=\tiny] (0.5,-1.9) -- node[pos=0.75,right,xshift=-3pt] {$L^{\frac13+\epsilon}$} (.5,-1.5);

    \draw[gray!90!black,font=\tiny] (-1.785,-1.85) -- node[pos=1,below] {$0$} (-1.785,-1.95);
   \draw[gray!90!black,font=\tiny] (0,-1.9) -- node[pos=1,below] {$L/2$} (0,-2.);
    \draw[gray!90!black,font=\tiny] (1.775,-1.85) -- node[pos=1,below] {$L$} (1.775,-1.95);

   \draw[black,dashed] (2.5,-1.7) -- (2.5,-.1) -- (11.9,-.1) -- (11.9,-1.7);

   \end{tikzpicture}
\vspace{-0.2in}
\caption{Fluctuations of the SOS level lines about the flat portions of their scaling limits. Maximal fluctuation is known to be at most $L^{1/3+\epsilon}$ w.h.p., and it is believed that the distance of the top level line from a given boundary point (e.g., the center-side) is of order~$ L^{1/3}$.}
\label{fig:level-lines}
\vspace{-0.1in}
\end{figure}

Significant progress in the understanding of the shape of the SOS surface above a hard wall was obtained in the sequel by the same authors~\cite{CLMST16}. The height-$h$ \emph{level lines} of the surface are the loops formed by placing dual-bonds between every pair $\x\sim \y$ such that $\varphi_\x<h$ and $\varphi_\y\geq h$. To account for local thermal fluctuations, call a loop \emph{macroscopic} if its length is at least $(\log L)^2$. With this notation, (a more detailed version of) the following theorem was given in~\cite{CLMST16} (see also~\cite{CLMST12}):
\begin{theorem*}[{\cite[Thms.~1,2,3 and Rem.~1.3]{CLMST16}}]
For $\beta$ large enough, the $(2+1)$D SOS model with zero boundary conditions on a square $\Lambda_L=\llb1,L\rrb^2$ above the wall $\varphi\geq 0$, satisfies the following w.h.p.:  \begin{enumerate}[(i)]
    \item{} \label{it:jems:loops} Shape:
    At least $(1-\epsilon_\beta)L^2$ of the sites $\x\in\Lambda_L$ have height $\varphi_\x = \fh^\star$, where the random $\fh^\star$ is either $ \lfloor \frac1{4\beta}\log L\rfloor$ or $\lfloor \frac1{4\beta}\log L\rfloor-1$. Moreover, there is a unique macroscopic loop at  each height $0,1,\ldots,\mathfrak h^\star$, and none above height $\fh^\star$.
Further, for a diverging sequence\footnote{Namely, for any sequence of $L$'s where $a_L$, the fractional part of $\frac1{4\beta}\log L$, does not converge to an explicit $\lambda_c(\beta)$.} of $L$'s,
the sequence of nested loops $\mathcal L_0 \subset \mathcal L_1 \subset\ldots$, when rescaled to $[0,1]^2$, converges in probability in Hausdorff distance to a deterministic limit defined by a Wulff shape $\mathcal W$, which is
    the convex body of area~$1$ minimizing the line integral of a surface tension  $\tau_\beta(\cdot)$ along its perimeter~$\partial \mathcal W$; the scaling limit of $\mathcal L_k$ (where $\cL_0$ is the top level line at height~$\fh^\star$) is given by the union of all possible translates of $\mathcal W$, rescaled by an explicit radius $r_k$ that is decreasing in $k$.

    \item{} \label{it:jems-fluct} Fluctuations:
For a diverging
sequence\footnote{Namely, for any sequence of $L$'s such that, for the above $a_L$'s and $\lambda_c(\beta)$, one has  $\liminf_{L\to\infty} a_L > \lambda_c$.} of $L$'s,
the
maximum displacement of the top level line~$\mathcal L_0$ from the boundary segment $I\times\{0\}$ for $I = \llb\epsilon_\beta L,(1-\epsilon_\beta) L\rrb$ is $L^{1/3+o(1)}$. That is, if
\[
\overline\rho(x)=\max\{ y \leq L/2\,:\; (x,y)\in\cL_0\}
\]
is the maximum $y$-coordinate of a point $(x,y)$ visited by $\cL_0$ in the bottom-half of $\Lambda_L$, then 
\begin{equation}\label{eq:max-rho-upper-bound}
\max_{x\in I} \overline\rho(x) \leq L^{1/3+\epsilon} \,,
\end{equation}
for any fixed $\epsilon>0$, whereas for every interval $I'\subset I$ of length $L^{2/3-\epsilon}$,
\begin{equation}\label{eq:max-rho-lower-bound}
\max_{x\in I'} \overline\rho(x) \geq L^{1/3-\epsilon} \,.
\end{equation}
\end{enumerate}
\end{theorem*}

Consider the distance of the top level line loop $\cL_0$ from a point $(x_0,0)$ on the bottom boundary of the box, where the scaling limit is flat---e.g., the center-side $x_0=L/2$ (see \cref{fig:level-lines} for a depiction). The above theorem shows that $\overline\rho(x_0) \leq L^{1/3+\epsilon} $ w.h.p., yet it gives no nontrivial lower bound on it. It is believed that $\overline\rho(x_0)$ should have order $L^{1/3}$ (with no poly-log corrections);  
more precisely, one expects $\overline\rho(x_0)\asymp_{\texttt P} L^{1/3}$, where we write $f\lesssim_{\texttt P} g$ if $f/g$ is uniformly tight, and $f\asymp_{\texttt P} g$ if $f\lesssim_{\texttt P} g \lesssim_{\texttt P} f$.

Moreover, one expects that if one were to rescale $\overline\rho(x)$ by $L^{1/3}$ along an interval of order $L^{2/3}$ positioned on bottom boundary (within the flat portion of the scaling limit)---take, e.g., \[ I_0 := \llb \tfrac L2 - L^{2/3}, \tfrac L 2 + L^{2/3}\rrb\] for concreteness---then, after rescaling said interval by $L^{2/3}$ (in the concrete example, to $[-1,1]$), one would arrive at a limit law of a nontrivial diffusion, a variant of a Ferrari--Spohn diffusion~\cite{FS05}.
(This prediction was stated here in terms of $\overline\rho(x)$, the maximal vertical displacement  of $\cL_0$, so as to be well-defined, as $\cL_0$ can have many points with a given $x$-coordinate; the same statement is expected to hold for $\underline\rho(x)$ measuring the \emph{minimal} displacement of $\cL_0$, as defined in \cref{mainthm:SOS lower bound}.)

To explain this prediction, note first that it is well-known that the law of a Brownian excursion on $[-N,N]$ tilted (penalized) by $\exp(-\lambda A)$, where $A$ is the area under it, tends as $N\to\infty$ to that of a Ferrari--Spohn diffusion. The entropic repulsion that propels the SOS level line loop $\cL$ to height $h$ acts much like an area tilt: if a loop has internal area $\mathsf S$, then its probability (roughly) gains a tilt of $\exp(\lambda \mathsf S)$ for $\lambda = \mathsf{H}(e^{-4\beta h})$ as described earlier, or equivalently, a tilt of $\exp(-\lambda  A)$ where $A = L^2-\mathsf S $ is the area exterior to it.
Consider $\cL_k$, which is at height $\mathfrak h^\star-k$: there $\lambda \approx e^{-4\beta(\mathfrak h^\star - k)} \approx L^{-1} e^{4\beta k}$ (recall $\mathfrak h^\star \approx\tfrac1{4\beta}\log L$), and we see that the rescaling of $\rho(x)$ by $L^{1/3}$ and $I_0$ by $L^{2/3}$  cancels the $L^{-1}$ factor in $\lambda$ and translates into a tilt of 
$\exp(-e^{4\beta k} \hat{A})$ 
where $\hat{A}$
is the rescaled area, as in the above continuous approximation (see also~\cite{ISV15,IoffeVelenik18}).
Related to this, the famous problem of establishing a Ferrari--Spohn law for the 2D Ising interface under critical prewetting (which may be seen as a version of the SOS problem only with a single contour as opposed to $c\log L$ many) was finally settled in a recent seminal work by Ioffe, Ott, Shlosman and Velenik~\cite{IOSV21} (prior to that, the $L^{1/3+o(1)}$ fluctuations were established by Velenik~\cite{Velenik04} and the tightness of the rescaled area was proved by Ganguly and Gheissari~\cite{GaGh21}). The~challenges in handling a diverging number of interacting (non-crossing) contours with  distinct area tilts (the $k$-th one is tilted by $\approx \exp(-e^{4\beta k}\hat{A})$) are such that the simplified problem that has Brownian excursions with area tilts is already nontrivial; see~\cite{CG23,CIW19a,CIW19b,DLZ23} for recent progress on~it.

In accordance with this prediction for the scaling limit of $L^{-1/3}\rho(x)$ along $I_0$, one expects that both $\max_{x\in I_0} \rho(x)$ and $\min_{x\in I_0} \rho(x)$ would be of the same order as our rescaling factor $L^{1/3}$; i.e., to be precise, that $\max \overline\rho(x) \lesssim_{\texttt P} L^{1/3}$ and that $\min\underline\rho(x) \gtrsim_{\texttt P} L^{1/3}$. Our main result is the latter part (readily implying $\rho(x_0)\gtrsim_{\texttt P} L^{1/3}$ at any given $\epsilon_\beta L \leq x_0 \leq (1-\epsilon_\beta)L$, e.g., the center-side $x_0=L/2$). 

\begin{theorem}\label{mainthm:SOS lower bound}
Fix $\beta$ large, and consider the $(2+1)$D SOS model with zero boundary conditions on $\Lambda_L=\llb1,L\rrb^2$ as per \cref{eq:def-sos-measure} above a wall $\varphi\geq 0$. Let $\cL_0$ be the (w.h.p.\ unique) top macroscopic level line, consider the interval $I_0 = \llb \tfrac{L}2-L^{2/3},\tfrac{L}2+L^{2/3}\rrb$ centered on the bottom boundary, and let
\[
\underline\rho(x)=\min\{ y \geq 0\,:\; (x,y)\in\cL_0\}
\]
denote the minimum vertical displacement of $\cL_0$ from the bottom boundary at the coordinate $x$. 
Then for every $\epsilon>0$ there exists $\delta>0$ such that for large enough $L$, with probability at least $1-\epsilon$,
\begin{equation}\label{eq:rho-lower-bound} \min_{x\in I_0} \underline\rho(x) \geq \delta \,L^{1/3}\,.\end{equation}
\end{theorem}

As we later explain, we obtain  \cref{eq:rho-lower-bound} by moving from $\rho(x)$, at a constant probability cost, to a  curve whose limit (after the same rescaling) is a Brownian excursion, yielding the $L^{1/3}$ bound.
This refines the lower bound in  \cref{eq:max-rho-lower-bound} 
into the estimate $\max_{x\in I_0}\overline\rho(x) \geq \min_{x\in I_0}\underline\rho(x) \gtrsim_{\texttt P} L^{1/3} $.
Note~though that one cannot replace the $L^{1/3+o(1)}$ in the upper bound of \cref{eq:max-rho-upper-bound} by $O(L^{1/3})$, 
as it addresses $\max\overline \rho(x)$ over all $ I=\llb\epsilon_\beta  L,(1-\epsilon_\beta)L\rrb$. 
Our comparison to a Brownian excursion implies that 
$\max_{x\in I}\underline\rho(x) \geq c L^{1/3}\sqrt{\log L}$ w.h.p.; as we later explain, \cref{mainthm:FS} will imply that  
$\max_{x\in I}\underline\rho(x) 
\gtrsim_{\texttt P} L^{1/3} (\log L)^{2/3}$, its predicted order (see \cref{rem:max-upper-tail} as well as \cref{q:upper-tail}).

To derive the Brownian excursion law, we rely on the powerful Ornstein--Zernike framework as developed by Campanino, Ioffe and Velenik~\cite{Ioffe98,CampaninoIoffe02,CIV03,CIV08}, that allows one to couple the interface in hand to a directed random walk. This machinery was the key to several recent advances in the understanding 2D Ising interfaces (e.g., \cite{IOSV21} mentioned earlier) and Potts interfaces (e.g., \cite{OttVelenik18,IOVW20}). In fact, the  work~\cite{IOVW20}, due to Ioffe, Ott, Velenik and Wachtel, is of particular interest in our setting: there it was shown that the interface of the 2D Potts model in a box with Dobrushin's boundary conditions has the scaling limit of a Brownian excursion for all $\beta>\beta_c$. 
As in our case, one of the main obstacles is the interaction of the interface with the boundary, and in particular, ruling out the scenario whereby the interface is pinned to the wall. This was achieved for the Potts interface in~\cite{IOVW20} (and later used as an ingredient in~\cite{IOSV21}) via a direct analysis of its random cluster counterpart, and then combined with a version of Ornstein--Zernike theory tailored to that model.

Here we instead appeal to the framework of Ioffe, Shlosman and Toninelli~\cite{IST15} to rule out pinning. That approach, while valid only for large enough $\beta$ (whereas the analysis in~\cite{IOSV21} holds for all $\beta>\beta_c$), is fairly generic, and applicable to SOS contours as part of the following family of \emph{Ising polymers}
(to~aid the exposition, we describe it briefly here, deferring its full definition to \cref{subsec:free-ising-polymers,subsec:modified-ising-polymers}).
Call a path $\gamma$ of distinct adjacent edges
 in $(\Z^2)^* = \Z^2+(\frac12,\frac12)$ (vertices may repeat according to a splitting rule)
a \emph{polymer}, or contour, if it connects the origin $\ostar=(\frac12,\frac12)$ to a marked~$\x_N$ at distance~$N$ from~$\ostar$ while staying in a half-plane $\H_{\n}$.
The model gives $\gamma$ a probability proportional~to 
\[ q(\gamma) := \exp\bigg(-\beta|\gamma| + \sum_{\cC}
\Phi(\cC;\gamma)\bigg)\,,\]
where the sum goes over every finite connected subset $\cC$ in $\Z^2$ that intersects $\Delta_\gamma$, the vertex boundary of $\gamma$, and the potential function $\Phi$ satisfies the following properties:
({\bf P1}) $\Phi(\cC;\cdot)$ is \emph{local}, in the sense that it only depends on $\gamma$ through $\cC\cap\Delta_\gamma$; ({\bf P2}) $\sup_\gamma|\Phi(\cC,\gamma)|$ \emph{decays exponentially} in the size of $\cC$ (more precisely, in the minimum size of a graph connecting its boundary edges); and ({\bf P3}) $\Phi$ is invariant under translations of the form $(\cC,\gamma)\mapsto(\cC+\v,\gamma+\v)$. The final requirement in \cite{IST15} is to have that ({\bf P4}) the \emph{surface tension is symmetric}: if one defines the surface tension as \[ \tau_\beta(\n):=-\lim_{N \to\infty}\frac1{ N}\log\Big(\sum_\gamma q(\gamma)\Big)\,,\] 
then 
the function $\n\mapsto \tau_{\beta}(\n)$ should have all discrete symmetries of~$\Z^2$. Under these conditions, the main result of \cite{IST15} was that modifying the potential function $\Phi$ into $\Phi'$ along $\partial \H_{\n}$ does not affect the surface tension. That is, if we let
$\Phi'(\cC,\gamma)=\Phi(\cC,\gamma)$ whenever $\cC$ is fully contained in~$\H_{\n}$, and the modified $\Phi'$ still obeys the decay condition in \cref{p:P2}, then the modified $\tau'_\beta$ agrees with the original~$\tau_\beta$. Moreover, the corresponding partition functions are comparable (see \cref{the:ist-main}).

The main ingredient in our proof of \cref{mainthm:SOS lower bound} is the following result, which establishes a Brownian excursion limit law for (a) Ising polymers as defined by \cite{IST15} in the positive half-plane~$\H$,
and (b) Ising polymers in a box of side length $N$. Our proof of Part (a) hinges  on the ``no pinning'' main result of \cite{IST15} (mentioned above) en route to refining its conclusion and deriving the limit law. Part (b), proved similarly, may be viewed as an analogue of  \cite{IOVW20} for any Ising polymer at large $\beta$. 

\begin{theorem}\label{mainthm:BM-for-Ising-polymers}
Fix $\beta$ large,  and consider the family of Ising polymers $\gamma$ (see \cref{def:free-ising-polymer,def:mod-ising-polymer}) in a domain $D$, where the potential function $\Phi'$ is modified along its boundary $\partial D$, and $D$ is either
\begin{enumerate}[(a)]
\item the positive half-plane $\H$ with the marked end-points $\ostar=(\frac12,\frac12)$ and $\x_N = (\frac12,N-\frac12)$; or
\item a box of side length $N$ whose bottom corners are the same marked end-points $\ostar$ and $\x_N$.
\end{enumerate}
There exists $\sigma>0$ such that,
if $\overline\gamma(x)=\max\{y : (x,y)\in\gamma\}$, then
$\overline \gamma(\lfloor x N\rfloor)/(\sigma\sqrt{N})$ converges weakly to a standard Brownian excursion on $[0,1]$, and the same holds for $\underline\gamma(x)=\min\{y : (x,y)\in\gamma\}$.

In particular, Part (b) applies to the SOS model $\widehat\pi_{\Lambda}^{0,1,1,1}$ with \emph{no floor} on a box $\Lambda$ of side length~$N$, for~$\beta>\beta_0$ and  boundary conditions $0$ on the bottom side and $1$ elsewhere: namely, the height-$1$ level line that connects the bottom corners of the box $\Lambda$ has a scaling limit of a Brownian excursion.
\end{theorem}

While \cref{mainthm:BM-for-Ising-polymers} addressed level lines in the SOS model (and more generally, Ising polymers) with \emph{no floor}---whereby the scaling limit is a Brownian excursion---its application for \cref{mainthm:SOS lower bound} (addressing SOS above a floor) used the fact that in that setting the effect of the floor is uniformly bounded. Indeed, in an $L^{2/3}\times L^{2/3}$ box centered on the bottom boundary, the tilting effect of the floor on the top level line (as a Radon--Nikodym derivative) amounts to a factor of $\exp[ c A / L ]$, where $A$ is the area under the non-tilted curve (note $A \lesssim_{\texttt P} L$ for a Brownian excursion on an interval of length $L^{2/3}$).
Since, as mentioned above, a Brownian excursion tilted by an area term is known to converge to a Ferrari--Spohn diffusion, one expects that the top level line of SOS in that box will actually dominate a Ferrari--Spohn diffusion. This is the content of \cref{mainthm:FS} below.

We first define the limiting object formally. 
Let $\mathsf{Ai}(x)$ denote the Airy function (of the first kind), i.e., the solution to $y''(x) = x y$ with the initial condition $y=0$ at $x=\infty$. For $\lambda,\sigma>0$, define $f_{\lambda,\sigma}(x) := \mathsf{Ai}((2\lambda\sigma)^{1/3}x +\omega_1)$, where $\omega_1$ is the ``first'' zero of $\mathsf{Ai}$ ($\omega_1<0$ and closest to~$0$).\footnote{The function $f_{\lambda,\sigma}$ is the first eigenfunction of the operator $\mathsf{L}$, see \cite[Eq.~(1.18)]{ISV15}}
The stationary Ferrari--Spohn diffusion we consider is the diffusion on $(0,\infty)$ with generator
\begin{align}\label{def:fs-generator}
    \mathsf{L}\psi = \frac{1}2 \psi'' + \frac{f_{\lambda,\sigma}'}{f_{\lambda,\sigma}}\psi'
\end{align}
and Dirichlet boundary condition at $0$. 

The following result establishes that if we consider the SOS model on a $KL^{2/3}\times KL^{2/3}$ box with boundary conditions $H-1,H,H,H$, where $H$ is the typical height of the top level line (up to $1$ integer), then the $H$-level line will converge weakly  to a Ferrari--Spohn diffusion in $(C[-T,T], \|\cdot\|_{\infty})$ for any $T>0$. 
A direct consequence (via the monotonicity argument in \cref{sec:sos}) is a refinement of \cref{mainthm:SOS lower bound}, showing that $\rho(x)$ from that theorem essentially dominates a Ferrari--Spohn diffusion (thus $\rho(x) \gtrsim_{\texttt P} L^{1/3}$; see \cref{rem:dominating-RW-area-tilt}).

\begin{theorem}
    \label{mainthm:FS}
    Fix $\beta$ large and consider the SOS model on a $K L^{2/3}\times K L^{2/3}$ box with a floor~at~$0$
    and boundary conditions $H = \lfloor \frac1{4\beta}\log L\rfloor$ everywhere except the bottom side, where they are $H-1$. Suppose that $a_L$, the fractional part of $\frac1{4\beta}\log L$, converges to a limit, and let $\overline\rho(x)$ denote the maximum vertical distance of the $H$-level line (connecting the bottom corners of the box) from the bottom side at horizontal location $x\in\R$. 
    Let $\sigma>0$ be the constant from \cref{mainthm:BM-for-Ising-polymers}, Part~(b).
    Then there exists $\lambda>0$ such that $ \overline\rho(\lfloor x L^{2/3} \rfloor ) / (\sigma L^{1/3})$  converges weakly as $L\to\infty$ followed by $K\to\infty$ to the stationary Ferrari--Spohn diffusion on $(0,\infty)$ with generator $\mathsf{L}$ and Dirichlet boundary condition at~0.
     The same holds for $\underline{\rho}(x)$, the minimum height fluctuation of the level line at $x\in\R$.
\end{theorem}
\begin{remark}\label{rem:FS-K-L-1/20}
Using the same methods, a Ferrari--Spohn diffusion limit may also be derived for Ising polymers with the appropriate area tilt. Furthermore, it is possible to take any diverging sequence of $K:= K(L) \in (0,L^{1/20})$. See \cref{thm:fs-detailed} for a more detailed version of \cref{mainthm:FS}. 
\end{remark}

\begin{remark}\label{rem:dominating-RW-area-tilt}
As mentioned above,  \cref{mainthm:BM-for-Ising-polymers} is proved via coupling the Ising polymer to a 2D directed random walk excursion. The proof of \cref{mainthm:FS}, taking into account the floor in the SOS model, proceeds by using the same machinery to couple the polymer to a random walk excursion, yet this time with an area-tilt. We then appeal to the approach of \cite[Section~6]{IOSV21} for handling the convergence of such 2D random walks to the Ferrari--Spohn diffusion.
As a byproduct of this argument, one can read off quantitative results on the model before taking $K,L\to\infty$; namely, the top level line of the SOS model with zero boundary conditions on $\Lambda_L$ above a wall dominates a random walk excursion with endpoints $0$ and an area tilt on any interval of length $L^{2/3}$ at distance at least $\epsilon_\beta L$ from the box corners (a stronger result than \cref{mainthm:SOS lower bound}).
\end{remark}

\begin{remark}\label{rem:max-upper-tail}
Consider the aforementioned stronger version of \cref{mainthm:SOS lower bound}, whereby the vertical displacement $\underline \rho(x)$ of the top level line $\cL_0$ of the SOS model, along any interval of length $L^{2/3}$ bounded away from the corners, stochastically dominates a random walk with area tilt $e^{-cA / L}$. Standard tools will then imply that $\max_{x\in I}\underline\rho(x)$ is $\Omega( L^{1/3}(\log L)^{2/3})$ for $I = \llb \epsilon_\beta L, (1-\epsilon_\beta)L\rrb$. 
Namely, one could show, for some absolute constant $c>0$, a lower bound on the upper tail of $\underline\rho(x_0)$, valid for all $x_0\in I$, \`{a} la Tracy--Widom distribution:
\[ \pi_{\Lambda_L}^0(\underline\rho(x_0) > a L^{1/3}) \geq e^{- c a^{3/2}}\,.\]
Considering about $L^{1/3}$ such $x_0$ taken in disjoint boxes of length $L^{2/3}$ each, along with monotonicity arguments similar to those employed in \cref{sec:sos}
will then imply $\max_{x\in I} \underline\rho(x) \gtrsim_{\texttt P} L^{1/3}(\log L)^{2/3}$.
\end{remark}

It is plausible that this gives the correct order of the upper tail large deviation rate function, and that consequently,  $\max_x\underline\rho(x) \asymp_{\texttt P} L^{1/3}(\log L)^{2/3}$ (see, e.g.,  \cite{GaGh21}, where an estimate of this type was obtained for the 2D Ising interface in critical prewetting, as well the work of Alexander~\cite{Alexander01} on  local roughness of droplet boundaries in the random cluster model).

\begin{question}\label{q:upper-tail} 
Let $x_0\in \llb\epsilon_\beta L, (1-\epsilon_\beta)L\rrb$. What is the rate function $a\mapsto -\log\pi_{\Lambda_L}^0(\underline\rho(x_0)>a L^{1/3})$?
\end{question}

As for lower tails, to our knowledge these are still open for the 2D Ising under critical prewetting, where one expects the Ising interface to reach the bottom in $O_{\texttt P}(1)$ locations along $\llb\epsilon_\beta L,(1-\epsilon_\beta) L\rrb$. One should stress though that it is unclear that the SOS large deviations would take after the behavior of the 2D Ising interface under prewetting---particularly for the lower tails, where in SOS there are $\Theta(\log L)$ level lines below $\cL_0$, all of which must cooperate with a downward deviation.
\begin{question}
    \label{q:lower-tail}
    Let $x_0\in \llb\epsilon_\beta L, (1-\epsilon_\beta)L\rrb$. What is the order of $\pi_{\Lambda_L}^0(\underline\rho(x_0)=0)$?
\end{question}

The paper is organized as follows.
In \cref{sec:polymers-oz}, we formalize the setting of Ising polymers, as well as the inputs we need from Ornstein--Zernike theory. We also 
 establish that the SOS model satisfies the required hypotheses of Ising polymers, and thus 
\cref{mainthm:BM-for-Ising-polymers} is applicable to it.
\Cref{sec:sos} proves \cref{mainthm:SOS lower bound},  addressing the SOS measure $\pi$,
using a monotonicity argument and the conclusion of \cref{mainthm:BM-for-Ising-polymers} on $\widehat\pi$, the SOS measure without a floor. In \cref{sec:rw-1}, we introduce a random walk in $\H$ that is closely related to Ising polymers, and state the key limit theorems for this random walk. In turn, \cref{sec:pf-BM-ist} provides the proof of \cref{mainthm:BM-for-Ising-polymers} modulo these random walks results that are deferred to \cref{sec:rw-half-space}.
In \cref{sec:FS}, we prove \cref{mainthm:FS}.

\section{Cluster expansion, Ising polymers, and Ornstein--Zernike theory}\label{sec:polymers-oz}

In this section, we review the tools needed for our proofs---notably, cluster expansion, prior work on Ising polymers, and Ornstein--Zernike theory.
In several cases, we will need variants of existing results, which are not covered by the results proved in the literature. In those cases, we provide proofs of these analogues (either in the main text or in the appendix).

Throughout the paper, we say that an event holds with high probability (w.h.p.) if its probability tends to $1$ as the system size (typically, $L$ or $N$) tends to $\infty$.
For two functions $f: \N \to (0,\infty)$ and $g: \N \to (0,\infty)$, write
$ f \sim g $
to denote that $\lim_{N\to \infty} f(N)/g(N) = 1$;  write 
$ f \lesssim g$
when there exists a constant $K>0$ such that $f(N) \leq K g(N)$ for all $N \in \N$; and write 
$f \asymp g $ 
when $f\lesssim g$ and $g\lesssim f$.

\subsection{Contours and cluster expansion}
A contour $\gamma$ is a collection of bonds $(e_i)_{i=1}^m$ in the dual lattice $(\Z^2)^*$, where all bonds are distinct except possibly $e_1$ and $e_m$ may coincide, every two consecutive edges share a vertex, and the path formed is simple except in accordance with a splitting rule: if the pair $e_{i}, e_{i+1}$ and $e_j, e_{j+1}$ all intersect at a vertex $\x$, then
the two other end-points of $e_i,e_{i+1}$ are on the same side of the line through $\x$ with slope~$1$ (from southwest to northeast), and similarly for $e_j,e_{j+1}$ (this is the \emph{northeast} splitting rule). 
We call $\gamma$ an \emph{open contour} if $e_1\neq e_m$.

In the context of the SOS model in a finite, connected $\Lambda \Subset \Z^2$ under $0$ boundary conditions, for any $h$, the $h$-level lines (recall that for any configuration $\varphi$ and integer $h$, these are the bonds dual to $\x\sim \y$ with $\varphi_\x<h$ and $\varphi_\y\geq h$) give rise to a collection of disjoint loops after applying the global splitting rule. 
In the presence of a boundary condition $\xi \in \{0,1\}^{\partial \Lambda}$ consisting of a connected stretch of $0$'s (and $1$'s elsewhere), this gives rise to a unique open contour among the height-$1$ level lines (accompanied by a collection of closed contours). We refer to this path as the open $1$-contour.
Let $\Delta_{\gamma}^+$  and $\Delta_{\gamma}^-$ denote the set of sites (of $\Z^2$) immediately above and below $\gamma$, respectively, and define 
\[
\Delta_{\gamma}:= \Delta_{\gamma}^+ \cup \Delta_{\gamma}^- \,.
\]
Note that the sites in $\Delta_{\gamma}^+$ have height $\geq 1$, while the sites in $\Delta_{\gamma}^-$ have height $\leq 0$. Each $\gamma$ divides $\Lambda$ into two regions, $\Lambda_{\gamma}^+$ and $\Lambda_{\gamma}^-$, where we write $\Lambda_{\gamma}^+$ to denote the region that contains $\Delta_{\gamma}^+$ as part of its inner boundary. 
The next proposition addresses the law of this unique SOS open contour $\gamma$. 

\begin{proposition}[{\cite[Lem.~A.2]{CLMST14}}]
\label{prop:soscontour-law} Consider the SOS model $\widehat\pi_\Lambda^\xi$ on  any finite, connected $\Lambda \subset \Z^2$ and under any boundary condition $\xi \in \{0, 1\}^{\partial \Lambda}$ that induces a unique open $1$-contour $\gamma$.
Then there exists a constant $\beta_0>0$ such that for all $\beta \geq \beta_0$,
\begin{align}
\widehat\pi_{\Lambda}^{\xi}(\gamma) \propto \exp\bigg(-\beta|\gamma| + \sum_{\cC \cap \Delta_{\gamma} \neq \emptyset} \phi(\cC;\gamma)\ind{\cC \subset \Lambda}\bigg) \,, \label{eqn:sos-contour-law}
\end{align}
for some ``decoration functions'' $\{\phi(\cC, \gamma)\}_{\cC \subset \Z^2}$ satisfying the following properties:
\begin{enumerate}[(i)]
    \item \label[property]{it:ce-conn} If $\cC$ is not connected, then $\phi(\cC; \gamma) = 0$.
    \item \label[property]{it:ce-local} The decoration function $\phi(\cC;\cdot)$ depends on $\gamma$ only through $\cC \cap \Delta_{\gamma}$. 
    \item \label[property]{it:ce-translation} For all $\v \in \Z^2$, 
    $\phi(\cC; \gamma) = \phi(\cC+\v; \gamma+\v)$\,.
    \item \label[property]{it:ce-decay} Letting $d(\cC)$ denote the cardinality of the smallest connected set of bonds of $\Z^2$ containing all boundary bonds of $\cC$ (i.e., bonds connecting $\cC$ to $\cC^c$), we have the decay bound
    \begin{align}
        \sup_{\gamma} |\phi(\cC; \gamma)| \leq \exp\Big(- (\beta-\beta_0) d(\cC)\Big)\,. \label{eqn:sos-decay-decorations}
    \end{align}
\end{enumerate}
Furthermore, if $\widehat Z_{\Lambda}^\xi$ is the partition function of the SOS model $\widehat \pi_{\Lambda}^\xi$ and 
\[
    \Gb^{\xi}_{\Lambda} 
    := \sum_{\gamma} e^{-\beta|\gamma| + \sum_{\cC \subset \Lambda} \phi(\cC;\gamma)}
\]
is the interface partition function corresponding to the distribution \cref{eqn:sos-contour-law}, then
\begin{align}
    \Gb^{\xi}_{\Lambda} =  \widehat{Z}_{\Lambda}^{\xi}/\widehat{Z}_{\Lambda}^{0}\,. \label{eqn:sos-contour-partitionfunction}
\end{align}
 \end{proposition}

The decoration functions come from 
\textit{cluster expansion} applied to the partition functions in $\Lambda_{\gamma}^+$ and $\Lambda_{\gamma}^-$.
In \cref{appendix:cluster-expansion},
we recall cluster expansion for the SOS model and provide the proof of \cref{prop:soscontour-law},
as the expression for the decoration function $\phi$ is needed to verify that it meets the criteria of modified Ising polymers (\cref{def:mod-ising-polymer}).
In light of \cref{it:ce-conn}, from now on we will write $\cC$ to denote a connected subset (or \emph{cluster}) of~$\Z^2$.

Next, we recall the notion of surface tension for the SOS model.  

\begin{definition}[Dobrushin boundary conditions, surface tension] \label{def:dob}
Fix $\vec{u} \in \bbS^{1}$ in the first quadrant, i.e., with $\theta_{\vec{u}} \in [0,\pi/2)$, where $\theta_{\vec{u}}$ is the angle $\vec{u}$ makes with the positive horizontal axis. Set $\Lambda_{N,M} := \llb 1,N\rrb\times\llb -M,M\rrb$, and let $\xi(\vec{u})$ denote the boundary condition defined by $\xi(\vec{u})_v = 0$ for all $v\in \partial \Lambda$ lying on or below $\mathrm{span}(\vec{u})$, and $\xi(\vec{u})_v= 1$ otherwise. Our main focus is on the boundary condition $\xi(\e_1)$, which we denote by $0,1,1,1$ (after the values induced by $\xi(\e_1)$ on the four sides of the box $\llb 1,N\rrb\times\llb 0,M\rrb$).
Define $d_{N,\vec{u}} := N/\cos(\theta_{\vec{u}})$. The \textit{surface tension} $\tau_{\beta}^{\SOS}\textbf{}(\vec{u})$ is defined by 
\begin{align}
\tau_{\beta}^{\SOS}\textbf{}(\vec{u}) := \lim_{N \to\infty}\lim_{M\to\infty} -\frac{1}{d_{N,\vec{u}}} \log \Big(\cG_{\Lambda_{N,M}}^{\xi(\vec{u})}\Big) \,.
    \label{def:sos-surface-tension}
\end{align}
The value of $\tau_\beta^{\SOS}$ for the other quadrants is defined symmetrically, so that $\tau_\beta^{\SOS}(\vec u) = 
\tau_\beta^{\SOS}(\vec v)$ when $\theta_{\vec u} = -\theta_{\vec v}$ or $\theta_{\vec u} = \pi -\theta_{\vec v}$.

Finally, $\tau_{\beta}^{\SOS}$ extends to an even function on all of $\R^2$ via homogeneity: 
\[
    \tau_{\beta}^{\SOS}(\x) := \norm{\x} \tau_{\beta}^{\SOS}\big(\tfrac{\x}{\norm{\x}}\big) \,, \text{ for all } \x \in \R^2\,.
\]
\end{definition}
The proof of the existence of $\tau_{\beta}^{\SOS}(\x)$ as well as many of its properties can be found in \cite[{\S1--2}]{DKS92}. 

\subsection{The free Ising polymer model}
\label{subsec:free-ising-polymers}
In this section, we define the class of Ising polymer models, as given by \cite{IST15}.\footnote{The Ising polymer model in \cite{IST15} had a weaker decay condition, taking $d(\cC)$ in \cref{p:P2} to be the $L^\infty$-diameter of~$\cC$, whereas our arguments  require $d(\cC)$ to be the minimum size of a connected set containing its boundary.} The reader is also referred to~\cite{DKS92} for many useful results on such polymer models.
It will be shown that the SOS open contour from \cref{eqn:sos-contour-law} falls in this class.

Recall that we  write $\cC$ to denote a finite, connected subset of $\Z^2$. For every contour $\gamma$, consider any \textit{decoration function} $\Phi(\cC; \gamma)$ satisfying the following four properties:
\begin{enumerate}[label=(P\arabic*)]
    \item  \label[property]{p:P1} \emph{Locality}: $\Phi(\cC; \cdot)$ depends on $\gamma$ only through $\cC\cap \Delta_{\gamma}$.
    
    \item \label[property]{p:P2} \emph{Decay}: There exists some $\chi >1/2$ such that, for all $\beta >0$ sufficiently large, 
    \begin{align}
        \sup_{\gamma} |\Phi(\cC,\gamma)| \leq \exp(-\chi \beta (d(\cC)+1))\,,
        \label{eqn:decay-decorations}
    \end{align}
    where $d(\cC)$ is defined as in \cref{it:ce-decay} of \cref{prop:soscontour-law}.
    
    \item \label[property]{p:P3} \emph{Translational symmetry}: for all $\v\in \Z^2$, $\Phi(\cC; \gamma) = \Phi(\cC+\v;\gamma+\v)$.
    
    \item \label[property]{p:P4} \emph{Symmetry of the surface tension}: the surface tension $\tau_{\beta}(\x)$ defined below in~\cref{def:tau-beta} possesses all discrete symmetries of $\Z^2$ (rotations by $\pi/4$ and reflections w.r.t.\ axes and the diagonals $y=\pm x$).
\end{enumerate}
Towards specifying the probability that the model assigns to each polymer---as well as the surface tension $\tau_\beta$ mentioned in \cref{p:P4}---define the \textit{free (polymer) weight} via 
\begin{equation}\label{eq:q-def}
    q(\gamma) = \exp\Big(-\beta|\gamma| + \sum_{\cC \cap \Delta_{\gamma} \neq \emptyset} \Phi(\cC;\gamma)\Big)\,,
\end{equation}
where, here and throughout the article, sums over $\cC$ are assumed to only go over \emph{connected} subsets $\cC \subset \Z^2$. Next, for any $\x \in (\Z^2)^*$, consider the partition function going over all contours $\gamma$ with start-point given by the dual origin $\ostar := (1/2,1/2)$ and end-point $\x$:
\begin{align*}
    \Gb(\x) := \sum_{\gamma: \ostar \to \x} q(\gamma)\,.
\end{align*}
For any set of contours $E$, consider also the partition function going over all contours in $E$ with end-points $0$ and $\x$:
\begin{align}
    \Gb(\x \given E) := \sum_{\gamma: \ostar \to \x} q(\gamma) \ind{\gamma \in E}\,.
\end{align}
Lastly, define the \textit{Ising polymer surface tension} $\tau_{\beta}(\cdot)$ via\footnote{It is common to define $\tau_\beta$ with the $1/\beta$ pre-factor (such was the case in~\cite{DKS92} as well as \cite{CLMST16}).
We do not include the pre-factor here since related Ornstein--Zernike works (e.g.,~\cite{IV08,IOSV21}) do not include this pre-factor, and this keeps various definitions (e.g., $\Kb$, $W^{\h}$, $\Pa$ defined below) consistent with those works. In~\cite{IST15}, $\tau_\beta$ does have a $1/\beta$ pre-factor, though it seems to be a typo, as their inputs from the Ornstein--Zernike theory come from the aforementioned~\cite{IV08} (and as such their calculations are consistent with the above definition of $\tau_\beta$). We will prove the Ornstein--Zernike facts we require here, to make the proof more self-contained and avoid potential consistency issues. \label{footnote:surface-tension-beta}
}
\begin{align}
    \tau_{\beta}(\n) := -\lim_{N\to\infty} \frac{1}{ N} \log \Gb(N\n) \qquad \text{ for } \n \in \bbS^1,
    \label{def:tau-beta}
\end{align}
where the limit is taken over $N$ such that $N \n$ is in $(\Z^2)^*$. By homogeneity, $\tau_{\beta}$ extends to 
all of $\R^2$.
\begin{definition}[Free Ising polymer]
\label{def:free-ising-polymer}
The \textit{free Ising polymer model} in a subset $D\subset \R^2$ is given by the probability measure over contours $\gamma:\ostar \to \x$ contained in $D$:
\begin{align}
    \bP^{\x}( \cdot \given \gamma \subset D) := \frac{\Gb(\x \given \gamma \subset D, \gamma \in \cdot )}{\Gb(\x \given \gamma \subset D)}\,,\label{eq:free-ising-polymer}
\end{align}
for a partition function $\Gb$ as above with a decoration function $\Phi$ satisfying \cref{p:P1,p:P2,p:P3,p:P4}.
\end{definition}

Below, we list some needed properties of $\tau_{\beta}$, proven in \cite{DKS92}.
\begin{proposition}[{Surface tension properties, \cite{DKS92}}]\label{prop:surface-tension-facts}
    There exists  $\beta_0 >0$ such that for all $\beta \geq \beta_0$:
    \begin{enumerate}[(i)]
        \item\label[property]{p:tau-unif-converge}  the formula in \cref{def:tau-beta} converges uniformly,
        \item \label[property]{p:tau-analytic} the surface tension $\tau_{\beta}: \R^2 \to \R^2$ is analytic, 
        \item \label[property]{p:tau-str-tri-ineq} (Strong triangle inequality) For any two non-collinear vectors $\u$ and $\v$ in $\Z^2$, we have
        \[
            \tau_{\beta}(\u) + \tau_{\beta}(\v) > \tau_{\beta}(\u+\v)\,.
        \]
    \end{enumerate}
\end{proposition}
\begin{proof}
\Cref{p:tau-unif-converge,p:tau-analytic} follow from the combination of Theorem~4.8 and Proposition~4.12 of~\cite{DKS92}.\footnote{These results are stated in~\cite{DKS92} for the Ising polymer model $\Gb(\x)$. The analogue of $\Gb$ is defined in \cite[Eq.~4.3.11]{DKS92} (where for us, $h=0$). Though that formula involves decoration functions $\Phi$ coming from the cluster expansion of the Ising model, the results in \cite[Section~4]{DKS92} that follow this formula only use \cref{it:ce-translation,it:ce-decay} of our \cref{prop:soscontour-law}. This is stated in the final paragraph of \cite[Section~4.3]{DKS92}.}
Lastly, \cref{p:tau-str-tri-ineq}  appeared in \cite[Proposition 1, Section 4.21]{DKS92}.
\end{proof}

For our main application of \cref{mainthm:BM-for-Ising-polymers}, the SOS model, the following result relates the surface tension $\tau_\beta^{\SOS}$ defined in \cref{def:sos-surface-tension} to the Ising polymer surface tension $\tau_\beta$ defined above. 
We suspect it is known, though could not find an exact reference for it, and as it follows from our other arguments in \cref{sec:pf-BM-ist}, we include its proof in
\cref{subsection:proof-surface-tension-equiv} for completeness. 

\begin{proposition}\label{prop:surface-tension-equiv}
Fix $\beta\geq \beta_0$. Let $q(\gamma)$ denote the free polymer weight as per \cref{eq:q-def} with $\Phi(\cC;\gamma)$ taken to be $\phi(\cC;\gamma)$ from
\cref{prop:soscontour-law}.
Consider the Ising polymer surface tension $\tau_{\beta}$ corresponding to the free polymer weights $q(\gamma)$ as per \cref{def:tau-beta}.
For all $\mathsf{u} \in \R^2$,
\[
    \tau_{\beta}^{\emph{\SOS}}(\mathsf{u}) = \tau_{\beta}(\mathsf{u})\,.
\]
\end{proposition}
Observe that the decoration function $\phi(\cC;\gamma)$ appearing in the law of the unique open $1$-contour
under $\widehat{\pi}_{\Lambda}^{\xi}$ satisfies \cref{p:P1,p:P2,p:P3,p:P4}.
Indeed, \cref{prop:soscontour-law} states that the decoration function $\phi(\cC;\gamma)$  satisfies \cref{p:P1,p:P2,p:P3}. \cref{p:P4} holds for $\tau_{\beta}^{\SOS}$ because of the symmetry of the SOS model, and thus for $\tau_{\beta}$ thanks to \cref{prop:surface-tension-equiv}.

However, the open $1$-contour in the SOS model is not a free Ising polymer, due to the $\ind{\cC\subset \Lambda}$ term appearing in~\eqref{eqn:sos-contour-law}, which introduces an interaction of $\gamma$ with the boundary via the decoration function $\phi$. This served as one of the motivations of \cite{IST15} to study modified Ising polymers---the generalization of the free Ising polymers described above to allow domain-induced modifications on the ``free'' weights (the SOS open contour does belong to that family of models: see \cref{obs:sos-is-ist}).

\subsection{The modified Ising polymer model} \label{subsec:modified-ising-polymers}
For any $D \subset \R^2$ and for any decoration function $\Phi(\cC;\gamma)$ satisfying \cref{p:P1,p:P2,p:P3,p:P4}, consider any function $\Phi_D(\cC;\gamma)$ satisfying
\begin{enumerate}[label=(M\arabic*)]
    \item \label[property]{p:M1} $\Phi_D(\cC; \gamma) = \Phi(\cC; \gamma)$ for any $\cC \subset D$, and 
    \item \label[property]{p:M2} $\Phi_D(\cC;\gamma)$ satisfies the same decay bound in \cref{eqn:decay-decorations} for all $\cC$.
\end{enumerate}
Call $\Phi_D(\cC;\gamma)$ a \textit{modified decoration function} (with modifications outside of $D$), and define  the \textit{modified polymer weight}
\[
    q_{D}(\gamma) = \exp\Big(-\beta|\gamma| + \sum_{\cC \cap \Delta_{\gamma} \neq \emptyset} \Phi_D(\cC;\gamma)\Big)
\]
as well as the partition function
\begin{align*}
    \Gb_D(\x) := \!\!\! \sum_{\gamma : \ostar \to \x , \gamma \subset D} q_D(\gamma) \ \  \text{ and } \ \ 
    \Gb_D(\x \given E ) := \!\! \! \sum_{\gamma : \ostar \to \x , \gamma \subset D} q_D(\gamma) \ind{\gamma \in E} \text{ for a set of contours $E$.}
\end{align*}
\begin{definition}[Modified Ising polymer]
\label{def:mod-ising-polymer}
The \textit{modified Ising polymer} in a subset $D \subset \R^2$ is given by the probability measure on contours $\gamma:\ostar\to\x$ contained in $D$:
\begin{align}
    \bP_{D}^{\x}(\cdot) := \frac{\Gb_D(\x \given \gamma \in \cdot)}{\Gb_D(\x)} \,,
    \label{eqn:modified-ising-polymer}
\end{align}
where the partition function $\Gb_D$ is defined for a decoration function $\Phi$ satisfying \cref{p:P1,p:P2,p:P3,p:P4}, and a modified decoration function $\Phi_D$ satisfying \cref{p:M1,p:M2}.
\end{definition}

Note that the SOS open contour with law $\widehat{\pi}_{\Lambda}^{\xi}$ given by \cref{eqn:sos-contour-law} is of the form above, with  $\Phi(\cC;\gamma)=\phi(\cC;\gamma)$, $D=\Lambda$, and $\Phi_{\Lambda}(\cC;\gamma)=\phi(\cC;\gamma)\ind{\cC \subset \Lambda}$ for $\phi(\cC;\gamma)$ from 
\cref{prop:soscontour-law} (in which this choice of $\Phi_D(\cC;\gamma)$ clearly satisfies \cref{p:M1,p:M2}). Namely, the following holds:

\begin{observation}[SOS is a modified Ising polymer] \label{obs:sos-is-ist}
Fix $\beta\geq \beta_0$, and consider the SOS model $\widehat\pi_{\Lambda}^{\xi}$ in a finite, connected subset $\Lambda \subset \Z^2$ and boundary condition $\xi \in \{0,1\}^{\partial \Lambda}$ that induces a unique open $1$-contour $\gamma$. Let $\bar{\Lambda}$ denote the region in $\R^2$ enclosed by $\partial \Lambda$ (i.e., the region enclosed by the $\Z^2$-edges connecting the boundary vertices of $\Lambda$). Assume for convenience that the start-point of $\gamma$ is $\ostar$, and denote its end-point by $\x$. Then $\gamma$ 
has the Ising polymer law $\bP^{\x}_{\bar\Lambda}$ defined with decoration weights $\Phi(\cC;\gamma) := \phi(\cC;\gamma)$ and $\Phi_{\bar\Lambda}(\cC;\gamma):= \phi(\cC;\gamma)\ind{\cC \subset \Lambda}$.
\end{observation}

It  will be very convenient to view Ising polymers as connected paths in the lattice $\Z^2$ rather than the dual lattice $(\Z^2)^*$ via the translation map $\iota: \ostar \mapsto 0$, and this is the convention we will follow for the remainder of the article (for clarity, note also that, through this mapping, clusters $\cC$ will henceforth denote finite, connected subsets of $(\Z^2)^*$ instead of $\Z^2$), excluding \cref{sec:sos}.

Let us now re-state \cref{mainthm:BM-for-Ising-polymers} in the above language,  which is the form in which we will prove it (\cref{subsec:pf-BM-for-Ising-polymers}). For $N \in \N$, define $Q:= [0,N]^2$.
\begin{theorem}\label{thm:BM-for-Ising-polymers-restated}
Fix $\beta>0$ large,  take $D$ to be either $\H$ or $Q$, and set $\x := (N,0)$. 
Consider an Ising polymer $\gamma \sim \bP_D^{\x}(\cdot)$. There exists $\sigma>0$ such that, if $\overline\gamma(x)=\max\{y : (x,y)\in\gamma\}$, then
$\overline \gamma(\lfloor xN \rfloor)/(\sigma\sqrt{N})$ converges weakly to a standard Brownian excursion in the Skorokhod space $(D[0,1], \|\cdot\|_{\infty})$ and the same holds for $\underline\gamma(x)=\min\{y : (x,y)\in\gamma\}$.
\end{theorem}

\begin{remark}\label{rk:curvature}
The variance $\sigma^2$ with $\sigma$ as in \cref{thm:BM-for-Ising-polymers-restated} (and by extension, the one in \cref{mainthm:BM-for-Ising-polymers}) is given above \cref{thm:rw-invariance}. It is related to the curvature of the Wulff shape (defined in \cref{subsec:surface-tension-wulff}) associated to the Ising polymer. See  \cite[Appendix~B]{IOSV21} for details.
\end{remark}

\subsection{Non-negative decoration functions and a product structure}
\label{subsec:pos-decorations-def-animals}
Following \cite[Section~3.1]{IST15}, we employ a construction going back to \cite{DS99} that allows us to consider decoration functions which are non-negative. In this discussion, we allow for the case $D =\R^2$, in which case $\Phi_D(\cC;\gamma) = \Phi(\cC;\gamma)$ and $q_D(\gamma) =q(\gamma)$. 

For a contour $\gamma$, we consider the set of (not necessarily distinct) bonds in $\Z^2$ (recall
we have applied the translation map $\iota: \ostar \mapsto 0$):
\[
    \nabla_{\gamma} := \bigcup_{b = (\y, \y+\e_i) \in \gamma} \{ b, b+\e_i,b-\e_i\} \,.
\]
To be clear, $b$ is a bond in $\gamma$, $\e_i$ denotes a standard basis vector, and $b\pm\e_i$ denotes the bond obtained by translating $b$ by $\pm \e_i$.
Thus, $\nabla_{\gamma}$ contains three bonds for each bond of $\gamma$. Define
\begin{align*}
    \Phi_{D}'(\cC;\gamma) &:= \abs{\cC \cap \nabla_{\gamma}}e^{-\chi \beta (d(\cC)+1)} + \Phi_D(\cC;\gamma) 
\end{align*}
where $|\cC \cap \nabla_{\gamma}|$ is equal to the number of bonds, counted with multiplicity, in $\nabla_{\gamma}$ that $\cC$ intersects:
\[
    \abs{\cC \cap \nabla_{\gamma}} = \sum_{b = (\y, \y+\e_i) \in \gamma} (\ind{b \cap \cC \neq \emptyset} + \ind{b+\e_i \cap \cC \neq \emptyset} + \ind{b-\e_i\cap \cC \neq \emptyset}) \,.
\]
Observe that $\Phi'_D \geq 0$ by \cref{eqn:decay-decorations} and \cref{p:M2}.
For any fixed bond $b \in \Z^2$, let $c(\beta)$ denote the value of
\[
    c(\beta) := \sum_{\cC \subset (\Z^2)^*  \,,\, \cC \cap b \neq \emptyset} e^{-\chi\beta( d(\cC)+1)}\,.
\]
Note that
\begin{align}
    c(\beta) = \sum_{m \geq 1} e^{-\chi\beta(m+1)} \abs{\{\cC \subset (\Z^2)^*: \cC \cap b \neq \emptyset ,d(\cC) = m \}} \leq e^{-\chi \beta}\,,
    \label{eqn:c-beta-bound}
 \end{align}
where we used that $\abs{\{\cC \subset (\Z^2)^*: \cC \cap b \neq \emptyset ,d(\cC) = m \}} \leq e^{c m}$ for some $c>0$. (To see this, replace the marked $e\in\cC$ intersecting $b$ by a marked boundary edge $e_0\in\partial\cC$ with the same $y$-coordinate (say) at the cost of a factor of $m$; then, when enumerating the smallest connected set of edges of $\Z^2$ containing the marked $e_0$ and specified edges $\partial\cC$, regard $\partial \cC$ as the vertices of a 6-regular graph (whose vertices are the $\Z^2$ bonds and two are adjacent if they share an end-point; that is, the line graph of $\Z^2$), and recall that in a graph whose maximum degree is $\Delta$, the number of $m$-vertex connected subgraphs rooted at a given vertex is at most $(e \Delta)^m$.)
Using the fact that $\cC\cap \Delta_{\gamma} \neq \emptyset$ implies $\cC\cap \nabla_{\gamma} \neq \emptyset$, we have
\begin{align*}
    \sum_{\cC \cap \Delta_{\gamma} \neq \emptyset} \Phi_D(\cC;\gamma) 
    = -3c(\beta)|\gamma| + \sum_{\cC \cap \nabla_{\gamma} \neq \emptyset} \Phi_{D}'(\cC;\gamma)\,,
\end{align*}
Thus, we have
\begin{align*}
    q_{D}(\gamma) = \exp \Big( -(\beta+3c(\beta))|\gamma| + \sum_{\cC \cap \nabla_{\gamma} \neq \emptyset} \Phi_{D}'(\cC;\gamma)\Big)\,. 
\end{align*}
Since $f(\beta):= \beta+3c(\beta)$ is  strictly increasing for all $\beta$ large enough, we will henceforth redefine $\beta = f(\beta)$ so that we may drop the laborious $3c(\beta)$ from our weights:
\begin{align}
    q_{D}(\gamma) = \exp \Big( -\beta|\gamma| + \sum_{\cC \cap \nabla_{\gamma} \neq \emptyset} \Phi_{D}'(\cC;\gamma)\Big) 
    \label{def:q-positive-decorations}\,.
\end{align}

A useful comparison to record at this stage is that, for any $D \subset \R^2$ , 
\begin{align}
    \abs{ \log \frac{q_D(\gamma)}{q(\gamma)}} \leq 6 e^{-\chi \beta}|\gamma|\,.
    \label{eqn:modified-weight-comparison}
\end{align}
\begin{remark}\label{rk:positive-decorations}
    For any domain $D \subset \R^2$, the non-negative decoration functions $\Phi'(\cC;\gamma)$ still satisfy \cref{p:P1,p:P2,p:P3,p:P4} above, and the modified non-negative decoration functions $\Phi_D'(\cC;\gamma)$ still satisfy \cref{p:M1,p:M2}.
\end{remark}

We next uncover the product structure of $q_D$. Defining $\Psi_{D}(\cC,\gamma) := \big(\exp(\Phi'_{D}(\cC; \gamma))-1\big)\one_{\{\cC \cap \nabla_{\gamma} \neq \emptyset\}}$, we may write
\[
    \exp\Big(\sum_{\cC \cap \nabla_{\gamma} \neq \emptyset} \Phi_{D}'(\cC;\gamma)\Big)
    =
    \prod_{\cC \cap \nabla_{\gamma} \neq \emptyset} \Big(\big(e^{ \Phi_{D}'(\cC;\gamma)} -1\big) +1\Big) = \sum_{\C = \{\cC_i\}} \prod_i \Psi_{D}(\cC_i;\gamma)\,,
\]
where the sum goes over all possible finite collections $\C$ of clusters.
Given this, for any contour $\gamma$ and any collection $\C$ of clusters, we define the \textit{animal weight} $q_{D}(\Gamma)$ of the \textit{animal} $\Gamma = [\gamma, \C]$ by
\begin{align}\label{def:animal-weight}
    q_{D}(\Gamma) = q_{D}([\gamma,\C]) := e^{-\beta|\gamma|}\prod_{\cC \in \C} \Psi_{D}(\cC; \gamma) \,,
\end{align}
and observe
\[
    q_{D}(\gamma) = \sum_{\Gamma = [\gamma,\C]}q_{D}(\Gamma)\,.
\]
The above allows us to consider the free and modified Ising polymer measures as probability measures on animals:
\begin{align}
    \bP^{\x}(\cdot \given \gamma \subset D) = \frac{\Gb(\x \given \gamma \subset D, \Gamma \in \cdot)}{\Gb(\x\given \gamma \subset D)} \qquad \text{and}\qquad
    \bP_{D}^{\x}(\cdot) := \frac{\Gb_D(\x \given \Gamma \in \cdot)}{\Gb_D(\x)} \,.
    \label{eqn:ising-polymer-animals}
\end{align}
Due to the product structure of $q_{D}(\Gamma)$, it is often convenient to consider animals rather than contours; indeed, we will see in \cref{subsec:rw-model} that the product structure begets a connection with a random walk, which is crucial to our analysis.

When $D$ is taken to be $\R^2$, we will omit $D$ from the notation laid out above.

\subsection{Notation for Ising polymers and animals}\label{subsec:notation}
Below, we set notation that will be used throughout the article in the context of Ising polymers and animals.
Recall that we regard Ising polymers as connected paths in the lattice $\Z^2$ rather than the dual lattice $(\Z^2)^*$ via the translation map $\iota: \ostar \mapsto 0$ (the convention hereafter, excluding \cref{sec:sos}).
\begin{itemize}[leftmargin=2em]
    \item We  write $Q:= [0,N]^2$. 
    \item For a point $\mathsf{u} \in \Z^2$, we will write $\mathsf{u}_{1}$ and $\mathsf{u}_{2}$ to denote the $x$-coordinate and $y$-coordinate of $\mathsf{u}$ respectively.

    \item For a subset $E \subset \Z^2$, we'll write $\mathsf{u}+ E$ to denote the translation of $E$ by the  vector defined by~$\mathsf{u}$.

    \item For a contour $\gamma$, we write $|\gamma|$ to denote the number of bonds in  $\gamma$. We write $(\gamma(0), \dots, \gamma(|\gamma|))$ to denote the ordered vertices of $\gamma$.

    \item For an animal $\Gamma := [\gamma, \C]$, we write $|\Gamma| := |\gamma|$, and  $\X(\Gamma) := \X(\gamma)$ to denote the displacement of $\gamma$; that is, the vector in $\Z^2$ given by the end-point of $\gamma$ minus the start-point of $\gamma$. 
    
    \item For a pair of contours $\gamma:= (\gamma(0), \dots, \gamma(|\gamma|))$ and $\gamma':= (\gamma'(0), \dots, \gamma'(|\gamma'|))$, let $\Delta := \gamma(|\gamma|) - \gamma'(0)$. 
    We define their concatenation to be
    \[
    \gamma \circ \gamma' := \big( \gamma(0), \dots, \gamma(|\gamma|),\Delta+ \gamma'(1), \dots, \Delta+ \gamma'(|\gamma'|) \big ) \,.
    \]
    Similarly, for a pair of animals $\Gamma := [\gamma, \C]$ and $\Gamma':= [\gamma', \C']$, we define their concatenation to be 
    \[
    \Gamma \circ \Gamma' := [ \gamma \circ \gamma', ~\C \cup \C']\,.
    \]
    \item For $\u, \v \in \Z^2$, we  write $\gamma : \u \to \v$ to indicate that $\gamma(0) = \u$ and $\gamma(|\gamma|) = \v$. We  write $\Gamma: \u \to \v$ to indicate that $\Gamma = [\gamma, \C]$ for some contour $\gamma : \u \to \v$.
    \item For a set $D\subset \R^2$, we say $\Gamma = [\gamma, \C] \subset D$ if $\gamma$ as well as all clusters $\cC \in \C$ are contained in $D$.
    \item For any $D \subset \R^2$, we define the set of contours 
    \[
        \cP_D(\u,\v):= \{\gamma: \gamma \subset D \,,\, \gamma : \u \to \v \}\,.
    \]
    When $\u = 0$, we will simply write $\cP_D(\v)$. We'll write $\Gamma \in \cP_D(\u,\v)$ for animals $\Gamma$ to mean $\Gamma = [\gamma, \C]$ for some contour $\gamma \in \cP_D(\u,\v)$.
\end{itemize}

\subsection{The Wulff Shape}
\label{subsec:surface-tension-wulff}
Now, define the \textit{Wulff shape}
\[
\Kb := \bigcap_{\y \in \R^2} \{\h \in \R^2: \h \cdot \y \leq \tau_\beta(\y) \}\,,
\]
which is clearly closed and convex (as it is the intersection of half-spaces). 
Observe that
\[
\Kb = \overline{\left\{\h \in \R^2: \sum_{\y \in \Z^2} e^{\h\cdot \y} \Gb(\y) < \infty\right\}}\,.
\]
Indeed, from \cref{def:tau-beta}, we have
\begin{align}
\log \Gb(\y) = - \tau_\beta(\y) \big(1+o_{\norm{\y}_1}(1)\big) \,.
\label{eqn:partition-fn-tau}
\end{align}
It follows that the sum in the second expression for $\Kb$ converges if and only if $\h \cdot \y <  \tau_\beta(\y)$ for all $\y\in\Z^2$ large,
  which by homogeneity and continuity of $\tau_{\beta}$ is equivalent to the same holding for all $\y \in \R^2$. 
Including the equality case $\h\cdot \y = \tau_{\beta}(\y)$  in the first expression for $\Kb$ is equivalent to taking the closure of the set in the second expression.

\subsection{Cone-points, the irreducible decomposition of animals, and weight factorization}
Let us define the \textit{forward cone}
$\fcone := \{(x,y) \in \Z^2 : |y|\leq x \}$ and the \textit{backward cone} $\bcone := -\fcone$. We will also need $\fcone_{\delta}:= \{ (x,y) \in \Z^2 : |y|\leq \delta x\}$ for $\delta >0$. Given a contour $\gamma$ and an animal $\Gamma=[\gamma,\C]$, we say that $\u\in\gamma$ is a \textit{cone-point} for $\gamma$ if 
\[
\gamma  \subset \u+\bcone \cup \u+\fcone\,,
\]
and we say $\u\in\gamma$ is a \textit{cone-point} for $\Gamma$ if
\[
\Gamma \subset \u+\bcone \cup \u+\fcone\,.
\]
Recall from \cref{subsec:notation} that the previous display means that the forward and backward cones emanating from $\u$ fully contain $\gamma$ as well as all clusters $\cC$ in $\C$. Of course if $\u \in \gamma$ is a cone-point for $\Gamma$, then $\u$ is a cone-point for $\gamma$ as well.
Note that for any $\u, \v \in \Z^2$, $\u \in \v + \fcone$ if and only if $\v \in \u +\bcone$.

If $\Gamma$ has two cone-points $\u$ and $\v$, the contour part of $\Gamma$  between $\u$ and $\v$ as well as all associated clusters are entirely contained  within the ``diamond''  
$\u+\fcone \cap \v+\bcone$.
An animal is called  \textit{left-irreducible} if it contains no cone-points and if it is entirely contained in the backwards cone emanating from its end-point. Similarly, an animal is called \textit{right-irreducible} if it contains no cone-points and it is contained in the forward cone emanating from its start-point. An animal is called \textit{irreducible} if it is both left- and right-irreducible.  We let $\AL, \AR$ and $\A$ be the sets of left-irreducible, right-irreducible and irreducible animals, respectively, with start-point at the origin.

Consider now an animal $\Gamma$ with at least two cone-points, and decompose it into (left-/right-) irreducible animals as follows:
\begin{align}
    \Gamma = \Gamma^{(L)} \circ \Gamma^{(1)} \circ \cdots \circ \Gamma^{(n)} \circ \Gamma^{(R)}\,, \label{eqn:irreducible-decomposition}
\end{align}
for some $n\geq 1$, $\Gamma^{(L)} \in \AL$, $\Gamma^{(R)} \in \AR$, and $\Gamma^{(i)} \in \A$ for $i =1,\dots,n$. For $i\in \{L, 1,\dots,n, R\}$, we will write 
\[
    \Gamma^{(i)} =: [\gamma^{(i)} , \C_i]\,.
\]
From the definition of a  cone-point, any cluster $\cC \in \C_i$ will be such that $\cC \cap \nabla_{\gamma^{(j)}} = \emptyset$ for $j \neq i$. Informally, this means that any cluster of $\Gamma^{(i)}$ will \textit{not} be a cluster of $\Gamma^{(j)}$. Thus, the weights $q_D(\Gamma)$ ``factorize'' into a product of the weights of the irreducible pieces:
    \begin{align}
        q_D(\Gamma) = \prod_{i \in \{L, 1,\dots, n,R\}} e^{- \beta |\gamma^{(i)}|}
        \prod_{\cC \in \C_i} \Psi_D(\cC;\gamma) = q_D(\Gamma^{(L)}) q_D(\Gamma^{(R)})\prod_{i=1}^n q_D(\Gamma^{(i)})
        \label{eqn:q-factorization}
    \end{align}
When $\gamma \subset Q$, note that the shape of the forward and backward cones necessitates 
\[
    \Gamma^{(1)} \circ \cdots \Gamma^{(n)} \subset [1,N-1] \times (-\infty, N]\,,
\]
so that \cref{eqn:q-factorization} becomes 
\begin{align}
    q_Q(\Gamma) = q_Q(\Gamma^{(L)}) q_Q(\Gamma^{(R)}) \prod_{i=1}^n q_{\H}(\Gamma^{(i)})\,. 
    \label{eqn:qQ-factorization}
\end{align}

We  write $\Cpts(\gamma)$ and $\Cpts(\Gamma)$  to denote the set of  cone-points of $\gamma$ and the set of  cone-points of $\Gamma$, respectively. 
Let $n:= n(\Gamma) = |\Cpts(\Gamma)|-1$ as above, so that $n$ is equal to the number of irreducible pieces of $\Gamma$. We label the cone-points of $\Gamma$ by $\zeta^{(1)}, \dots, \zeta^{(n+1)}$, where the ordering is  strictly increasing in the $x$-coordinates, i.e.,
    \[
    \zeta_{1}^{(i)} < \zeta_1^{(i+1)} \,.
    \]
 The set $\Cpts(\Gamma)$ therefore takes on a natural meaning as an ordered $n(\Gamma)-$tuple, and so we will abuse notation and write
    \begin{align}
        \Cpts(\Gamma) = \big(\zeta^{(i)}\big)_{i=1}^{n+1} 
        = \big( \X(\Gamma^{(L)}), \X(\Gamma^{(L)}) + \X(\Gamma^{(1)}), \dots, \X(\Gamma^{(L)})+ \cdots + \X(\Gamma^{(R)}) \big ) \,. 
        \label{eqn:cone-points-vector}
    \end{align}

The following result of \cite{IST15} states that the free weight of 
contours in $\Z^2$ that have large length and  contain a sub-linear number of cone-points compared to the distance of the end-point is exponentially small compared to $\Gb$. 
\begin{lemma}[{\cite[Eq.~(4.5)]{IST15}}\footnote{In \cite{IST15}, $1+\ep$ is replaced by an unspecified constant $r_0$; however, it is trivial to see that $r_0$ may be taken to be arbitrarily close to $1$ in \cite[Lemma~4]{IST15},  which is the same $r_0$ as in \cite[Eq.~4.5]{IST15}. Additionally, our condition $\y \in \fcone_{\delta}\setminus \{0\}$ appears as $\x \in \mathcal{Q}_+$ in \cite{IST15}.}] \label{lem:ist-length-cpts}
Fix any $\ep,\delta \in (0,1)$. 
There exist $\beta_0 \in (0, \infty)$, $\nu_0>0$, $\delta_0 >0$, and $c>0$ such that the following bounds hold uniformly over $\beta \geq \beta_0$, $\y \in \fcone_{\delta}\setminus \{0\}$, and $r \geq 1+\ep$:
    \begin{align}
        \Gb\big(\y \given |\gamma| \geq r \norm{\y}_1\big) 
        &\leq ce^{-\nu_0 \beta r \norm{\y}_1} \Gb(\y)
        \label{eqn:ist-length-bound}
        \\
        \Gb\big(\y \given |\Cpts(\gamma)| < 2\delta_0 \norm{\y}_1 \big)
        &\leq c e^{- \nu_0 \beta \norm{\y}_1} \Gb(\y)
        \label{eqn:ist-many-cone-points}
    \end{align}
\end{lemma}

In \cite{IST15}\footnote{See the paragraph containing equation~(4.9) in~\cite{IST15}}, it is claimed without proof that the an analogue of \cref{lem:ist-length-cpts} also holds for animals. For completeness, we supply the result as well as  a proof below.

\begin{proposition}\label{prop:many-cone-points}
For any $\delta \in (0,1)$, 
there exist  $\nu>0$ and $c>0$ such that the following bounds hold uniformly over $\beta \geq \beta_0$, and $\y \in \fcone_{\delta}\setminus \{0\}$:
\[
\Gb(\y \given |\Cpts(\Gamma)| < \delta_0\|\y\|_1) 
\leq
ce^{-\nu\beta \|\y\|_1}\Gb(\y) \,,
\]
where $\delta_0$ is as in \cref{lem:ist-length-cpts}.
\begin{proof}
The idea is to show that, for a typical $\Gamma= [\gamma,\C]$, many of the cone-points for $\gamma$ are also cone-points for $\Gamma$. We do this by showing that typically, $\cC$ does not contain many clusters (more precisely, the total $d(\cdot)$-size is not too big). The result then follows from \cref{eqn:ist-many-cone-points}.

Our starting point is the following, which comes from  \cref{p:P2} and \cref{rk:positive-decorations} and states that for some constant $c_1>0$, for all $\beta>0$ sufficiently large, and for any animal $[\gamma, \C]$,
\begin{equation}\label{eqn:clusterWeightUB}
q([\gamma, \underline{\mathcal{C}}]) \leq e^{-\beta|\gamma|}\exp\bigg({-c_1\beta\sum_{\cC \in \underline{\mathcal{C}}} d(\cC)}\bigg)\,.
\end{equation}
Now, define the events
\begin{align*}
A &= \{\Gamma = [\gamma,\underline{\mathcal{C}}] \given |\Cpts(\Gamma)| < \delta_0\|\y\|_1\} \\
B &= \{\Gamma = [\gamma,\underline{\mathcal{C}}] \given |\Cpts(\gamma)| \geq 2\delta_0\|\y\|_1\,, \,|\gamma| \leq 1.1 \|\y\|_1 \}
\end{align*}
Then, from \cref{eqn:ist-length-bound,eqn:ist-many-cone-points}, we have
\[
\frac{\Gb(\y \given A)}{\Gb(\y)} = \frac{\Gb(\y \given A, B)}{\Gb(\y)} + O\big( e^{-\nu_0 \beta \norm{\y}_1}\big)\,.
\]
Now, consider a contour $\gamma$ with at least $2\delta_0\|\y\|_1$ cone-points 
and an animal $\Gamma$ such that $\Gamma = [\gamma, \C]$ for some set of clusters $\C$. Note that if $\Gamma$ has less than $\delta_0\|\y\|_1$ cone-points, then the clusters of $\C$ must intersect $\gamma$ in such a way that at least $\delta_0\|\y\|_1$ cone-points of $\gamma$ are not cone-points of $\Gamma$.
This necessitates that the sum of $d(\cC)$ over $\cC \in \C$ exceeds $\delta_0 \|\y\|_1$.
Hence, we can write
\begin{align*}
&\Gb(\y \given A, B) \\ \leq &\sum_{\substack{\gamma:0\to \y \\ |\gamma|\leq 1.1 \|\y\|_1}}\sum_{m \geq \delta_0 \norm{\y}_1} \sum_{n=1}^{m}\sum_{k=1}^{n}\binom{3|\gamma|}{k} \sum_{\substack{\underline{n} = (n_1,\dots,n_k) \\ n_1+\dots+n_k=n}} \sum_{\substack{\underline{m} = (m_1,\dots,m_n) \\ m_1+\dots+m_n=m}} \bigg(\prod_{i=1}^k|G_{i, \underline{n}, \underline{m}}|\bigg) e^{-\beta|\gamma|-c_1\beta m}\,,
\end{align*}
where:
the second sum accounts for the possible values of $m =\sum_{\cC\in \C} d(\cC)$; 
the third sum accounts for the possible values of $n=|\C|$ given $m$; 
the fourth sum and the binominal coefficient account for the number of possible bonds of $\nabla_{\gamma}$ that the clusters of $\C$ intersect;
the fifth sum accounts for the number of clusters that intersect a particular bond of $\nabla_{\gamma}$ (given $k$ and an arbitrary labeling of the bonds $b_1, \dots, b_k$);
the sixth sum accounts for the possible $d(\cdot)$-value of each of the $n$ clusters given that their total $d(\cdot)$-sum  is $m$, and where $\underline{m}=(m_1,\dots,m_n)$ 
is ordered such that
$m_1,\dots,m_{n_1}$ represent the $d(\cdot)$-values of the $n_1$ clusters intersecting $b_1$, $m_{n_1+1},\dots,m_{n_1+n_2}$ represent the $d(\cdot)$-values of the $n_2$ clusters intersecting $b_2$ and so on;
each $G_{i, \underline{n}, \underline{m}}$ denotes the set of all possible collections of the $n_i$ clusters intersecting $b_i$ 
given the aforementioned $\underline{n}$ and $\underline{m}$;
and lastly, the exponential is the cumulative upper bound on $q([\gamma,\underline{\mathcal{C}}])$ given by \cref{eqn:clusterWeightUB} (given that the cumulative $d(\cdot)$ is $m$). 

Now, similar to \cref{eqn:c-beta-bound}, the number of clusters $\cC$ with $d(\cC)= r$ and  intersecting some fixed bond $b$ is bounded from above by $e^{c_2 r}$, and therefore
\[
\prod_{i=1}^k|G_{i, \underline{n}, \underline{m}}| \leq e^{c_2 (m_1+\dots+m_n)} = e^{c_2 m}\,.
\]
Hence, $\Gb(\y \given A, B)$ is upper-bounded by
\begin{align*}
&\sum_{\substack{\gamma:0\to \y \\ |\gamma|\leq 1.1 \|\y\|_1}}e^{-\beta|\gamma|}\sum_{m \geq \delta_0 \norm{\y}_1} \sum_{n=1}^{m}\sum_{k=1}^{n}\binom{3|\gamma|}{k} \sum_{\substack{\underline{n} = (n_1,\dots,n_k) \\ n_1+\dots+n_k=n}} \sum_{\substack{\underline{m} = (m_1,\dots,m_n) \\ m_1+\dots+m_n=m}} e^{-c_1\beta m}e^{c_2 m} \\
\leq&\sum_{\substack{\gamma:0\to \y \\ |\gamma|\leq 1.1 \|\y\|_1}}e^{-\beta|\gamma|}\sum_{m \geq \delta_0 \norm{\y}_1}e^{-c_3\beta m} \sum_{n=1}^{m}\sum_{k=1}^{n}\binom{3|\gamma|}{k} \binom{n-1} {k-1} \binom{m-1}{n-1} \\
\leq&\sum_{\substack{\gamma:0 \to \y \\ |\gamma|\leq 1.1 \|\y\|_1}}e^{-\beta|\gamma|} \sum_{m \geq \delta_0 \norm{\y}_1}e^{-c_3\beta m} 2^m (2^m-1) 2^{3|\gamma|} 
\leq \ e^{-c_4\beta \norm{\y}_1}\sum_{\gamma:0\to\y}e^{-\beta|\gamma|}\,.
\end{align*}
Finally, the non-negativity of the decorations $\Phi'$ implies
\[
\Gb(\y) = \sum_{\substack{\Gamma = [\gamma, \underline{\mathcal{C}}] \\ \gamma:0 \to\y}}q([\gamma, \underline{\mathcal{C}}]) = \sum_{\gamma:0 \to\y}\exp\bigg({-\beta|\gamma|+\sum_{\substack{\cC \subset \Lambda \\ \cC \cap \Delta_{\gamma} \neq \emptyset}} \Phi'(\cC;\gamma)}\bigg) \geq \sum_{\gamma:0 \to\y}e^{-\beta|\gamma|}\,,
\]
and thus
\[
\frac{\Gb(\y \given A, B)}{\Gb(\y)} \leq e^{-c_5\beta \norm{\y}_1}\,,
\]
concluding the proof.
\end{proof}
\end{proposition}

\begin{remark}\label{rk:cone-points-general}
Note that the proof of \cref{prop:many-cone-points} is such that, for any Ising polymer model $\Gb_D$ or $\Gb(\cdot \given \gamma \subset D)$ satisfying the analogous bounds in \cref{lem:ist-length-cpts},    \cref{prop:many-cone-points}  follows with $\Gb$ replaced by $\Gb_D$ or $\Gb(\cdot \given \gamma \subset D)$. This is used when we prove the existence of cone-points for $\Gamma$ in these models in \cref{lem:generalized-length-cpts}.
\end{remark}

\subsection{Ornstein--Zernike theory and its applications} \label{subsec:oz-theory}
For $\h \in \R^2$, we define $W^{\h}(\cdot)$ by
\begin{align*}
W^{\h}(\Gamma) := e^{\h \cdot \X(\Gamma)}q(\Gamma)
\end{align*}
for any animal $\Gamma$ (not necessarily irreducible). For $\y\neq0$, define the \textit{dual} parameter $\h_\y$ by
\[
\h_\y = \nabla \tau_\beta(\y)\,.
\]
The homogeneity of $\tau_\beta(\y)$ as per \cref{def:tau-beta} implies that $\h_\y \in \partial \Kb$ (since $\h_\y \cdot \y = \nabla \tau_\beta(\y)\cdot\y = \tau_\beta(\y)$ and recalling the first definition of $\Kb$). 

\cref{prop:oz-weights} below, which forms the main result of the Ornstein--Zernike theory for the Ising polymer models, was stated in \cite[Thm.~5]{IST15} in a slightly different setting (namely, the cones are defined differently here), pointing to \cite{IV08} as its relevant input. Indeed, one may infer this result via an Ornstein--Zernike analysis as in \cite[Sec~3.3 and~3.4]{IV08}, but since (a) the models considered in that work do not include Ising polymers, and (b) our setting differs somewhat from that of \cite{IST15}, we include a full proof of this proposition in \cref{appendix:ornstein--Zernike}.

\begin{proposition}\label{prop:oz-weights}
For any $\delta \in (0,1)$, there exists $\beta_0>0$ such that for all $\beta>\beta_0$ and for any $\y \in \fcone_{\delta}\setminus\{0\}$, the collection of weights $W^{\h_{\y}}$ 
defines a probability distribution on the set $\A$ of irreducible animals. 
To emphasize that $W^{\h_{\y}}$ defines a probability distribution (on irreducible animals), and for consistency with \cite{IST15}, we use the notation
\begin{align}
\P^{\h_{\y}}(\Gamma) := e^{\h_{\y}\cdot \X(\Gamma)} q(\Gamma) = W^{\h_{\y}}(\Gamma) \,.
\label{def:pa}
\end{align}
Let $\E^{\h_{\y}}$ denote expectation under $\P^{\h_{\y}}$. Then 
\begin{align}
    \E^{\h_{\y}}[\X(\Gamma)] = \alpha \y \,,
    \label{eqn:colinear-expectation}
\end{align}
for some constant $\alpha := \alpha(\beta, \y) >0$--- in particular, the expectation of $\X(\Gamma)$ under $\P^{\h_{\y}}$ is collinear to $\y$.
Finally, there exists a ``mass-gap'' constant $\nu_g>0$ 
such that for all $\beta>0$ large, $\y \in \fcone_{\delta} \setminus\{0\}$, and $k \geq 1$,
\begin{align}
    \sum_{\Gamma \in \AL \cup \AR} \P^{\h_{\y}}(\Gamma)\ind{|\Gamma| \geq k} \leq C e^{-\nu_g \beta k}\,,
    \label{eqn:oz-exp-decay}
\end{align}
where $C:= C(\beta)>0$ is a positive constant. 
\end{proposition}

Note that, since $\h_{\y} \cdot \y = \tau_{\beta}(\y)$ (by the homogeneity of $\tau_\beta$ as mentioned above), we have the following from \cref{eqn:q-factorization} 
\begin{align}
    q(\Gamma) = e^{-\tau_{\beta}(\y)} \P^{\h_{\y}}(\Gamma^{(L)}) \P^{\h_{\y}}(\Gamma^{(R)}) \prod_{i=1}^n \P^{\h_{\y}}(\Gamma^{(i)})
    \,,
        \label{eqn:decomposed-weight}
\end{align}
for any $\Gamma$ with at least two cone-points. 
For $\u, \v \in \Z^2$ and a set of animals $E$,  introduce
\begin{equation}
\fA(\u,\v ; E)  = \sum_{n\geq 1} \sum_{\Gamma^{(1)}, \dots, \Gamma^{(n)} \in \A} \prod_{i=1}^n \P^{\h_{\y}}(\Gamma^{(i)}) \ind{\u+ \gamma^{(1)} \circ \dots \circ \gamma^{(n)} \in \cP_{\H}(\u,\v)} \ind{\u+\Gamma^{(1)} \circ \dots \circ \Gamma^{(n)} \in E} \,. \label{def:A-partition-fn}
\end{equation}
When $E$ is taken to be the set of all possible animals, we'll simply write $\fA(\u,\v ; E) = \fA(\u,\v)$.
Then, for $D = Q$ or $\H$, 
\cref{eqn:decomposed-weight} allows us to write 
\begin{multline}
    \Gb(\y \given \gamma \subset D, |\Cpts(\Gamma)|\geq 2)
    \\
    = \ e^{-\tau_{\beta}(\y)} \sum_{\substack{\Gamma^{(L)} \in \AL \\ \gamma^{(L)} \subset D}} \P^{\h_{\y}}(\Gamma^{(L)})
    \ \sum_{\substack{\Gamma^{(R)} \in \AR \\ \gamma^{(R)} \subset D}} 
    \P^{\h_{\y}}(\Gamma^{(R)})  \ \fA( \X(\Gamma^{(L)}), \y- \X(\Gamma^{(R)}) )\,.
    \label{eqn:good-decomposed-2}
\end{multline}
Note that the sum over $n$ in~\eqref{def:A-partition-fn} above closely resembles the probability of a random walk event.
This product structure of probability weights $\P^{\h_{\y}}$ naturally leads to considering the random walk model related to the law of $\Cpts(\Gamma)$ described in \cref{subsec:rw-model}.

\section{Proof of 
\texorpdfstring{\cref{mainthm:SOS lower bound}}{Theorem 1.1}
modulo \texorpdfstring{\cref{mainthm:BM-for-Ising-polymers}}{Theorem 1.2}}\label{sec:sos}

In this section, we infer \cref{mainthm:SOS lower bound} from  \cref{mainthm:BM-for-Ising-polymers}(a). Throughout the section, let $H_L=\lfloor\frac1{4\beta}\log L\rfloor$.
Recall that, as per \cref{eq:rho-lower-bound}, we wish to give a lower bound on the event
\[
E := \Big\{\min_{x\in I_0} \underline\rho(x) \geq \delta \,L^{1/3}\Big\}\,.
\]
The fact that the definition of $\underline\rho(x)$ pertains to the height-$\fh^*$ (top) level line, where $\fh^*$ is random in $\{H_{L-1}, H_L\}$ (as mentioned in the Introduction, the results of \cite{CLMST16} determine it w.h.p.\ for side lengths $L$ outside of a critical subsequence) make this event delicate to work with. We will derive the required lower bound on $E$ via more tractable events, tailored to the analysis of a local rectangle of side-length $\asymp L^{2/3}$ located at the center of the bottom side of $\Lambda_L$.

Let
\[ R_0 := \llb \tfrac L2 - 3 L^{2/3}, \tfrac L2 + 3 L^{2/3}\rrb \times \llb 0, 2 L^{2/3}\rrb\,,\]
and let $R_1\subset \Lambda_L$ be the set of $\x\in \Lambda_L$ at distance at most $(\log L)^2+1$ from $R_0$. For fixed $\delta>0$, define the marked region $\mathfrak{M} := I_0 \times \llb 0, \delta L^{1/3}\rrb$. \cref{mainthm:SOS lower bound} will follow from the next proposition, where here and in what follows, $x$ and $y$ are $*$-connected if their Euclidean distance is at most $\sqrt2$.
\begin{proposition}
\label{prop:sos-h-and-h-1-final}
Fix $\beta\geq \beta_0$.
Let $B_n$ be the event that there is a $*$-connected path in $R_0$, intersecting $\mathfrak{M}$, whose  length is $\geq (\log L)^2$ and whose sites have height $\geq H_L-n$. Let $C_n$ be the event that there is no simple path in $R_1$ of length $\geq (\log L)^2$ whose sites have height $\geq H_L-n+1$. Then, for every fixed $n\geq 0$ and $\epsilon>0$ there exists $\delta>0$ such that $\pi_{\Lambda_L}^0(B_n\cap C_n) < \epsilon$ for all $L$ large.
\end{proposition}
\begin{proof}[Proof of \cref{mainthm:SOS lower bound}]

Observe that 
\[\pi_{\Lambda_L}^0( E^c ,\, \fh^*= H_L) \leq \pi_{\Lambda_L}^0(B_0 \cap C_0)+o(1)\,.\]
Indeed, \cite[Theorem~2a]{CLMST16} implies that there are w.h.p.\ no level lines of length at least $\log^2 L$ nested in the height-$\fh^{*}$ level line loop $\cL^{*}$ (implying $C_0$), whereas 
$E^c $ implies that $\cL^*$ intersects $\mathfrak M$; thus,
following its interior boundary for length at least $(\log L)^2$ will produce the path as per $B_0$.

Similarly,
\[\pi_{\Lambda_L}^0( E^c ,\, \fh^*=H_L-1) \leq \pi_{\Lambda_L}^0(B_1 \cap C_1)+o(1)\,,\]
again using that the height-$\fh^{*}$ level line loop $\cL^{*}$ must encompass a $*$-chain of sites crossing the marked region $\mathfrak{M}$ as in $B_1$ and that $C_1$ holds w.h.p. when $\fh^{*} = H_L-1$, by \cite[Theorem~2a]{CLMST16}.  The conclusion now follows from \cref{prop:sos-h-and-h-1-final}.
\end{proof}

\begin{proof}[Proof of \cref{prop:sos-h-and-h-1-final}]
Fix $n\geq 0$ and $\ep >0$. 
The following procedure will reveal the ``outermost'' chain $\mathfrak C$ of $\Z^2$-connected  sites of height at most $H_L-n$ that encloses, together with the southern boundary of $\Lambda_L$, the box $R_0$.
 For each site $\x$ in the north, east, and west boundaries of $R_1$, reveal its $\Z^2$-connected component of sites $\y \in R_1$ for which $\varphi_{\y} > H_L-n$. 
 Let $U$ denote the collection of all revealed vertices, and let $\fC$ be the exterior boundary of $U$ contained in $R_1$. By definition,  $\varphi_\y \leq H_L-n$ for all $\y \in \fC$. 
 Now, let $W$ denote the collection of sites whose exterior boundary is formed by $\fC$ and the southern boundary of $\Lambda_L$.
On the event $C_{n}$, each connected component comprising $U$ has diameter strictly less than $(\log L)^2$; in particular, $R_0 \subset W$.

 Now, condition on $\fC$, and let $\xi(\fC)$ denote the boundary condition of $W$ given by the heights on $\fC$ and the sites of height $0$ on $\partial \Lambda$. From the domain Markov property,  we may deduce our desired bound $\pi_{\Lambda_L}^0(B_n \cap C_n) <\ep$ if we can show
 \[
    \pi_{W}^{\xi(\fC)}(B_n) < \ep\,,
 \]
 uniformly over $\fC$.
Define $\x_{\ell} := \lfloor \tfrac{L}2 -2L^{2/3} \rfloor$ and  $\x_r:=\x_{\ell} + \lfloor 4L^{2/3} \rfloor$.
Using that all heights on $\fC$ are bounded by $H_L-n$, that all heights on $\partial \Lambda_L$ are $0$ (in particular, bounded by $H_L-n-1$), and that $B_n$ is an increasing event, we have
 \begin{align}
    \pi_{W}^{\xi(\fC)}(B_n)\leq \pi_{W}^{\legs}(B_n) \,,
    \label{eqn:Bn-monotonicity}
 \end{align}
where the boundary conditions $\legs$ are $H_L-n-1$ on the horizontal line of sites $\llb \x_{\ell}, \x_r \rrb \times \{-1\}$ and $H_L-n$ elsewhere (see \cref{fig:wiggle}, noting the ``legs'' of sites of height $H_L-n$ protruding inwards from the bottom boundary of $W$). Thus, it suffices to show the right-hand is bounded by~$\ep$.
\begin{figure}
    \centering
    \begin{tikzpicture}

    \draw[black,dashed] (-7.195,-.7) -- (-4.195,-.7);
    \draw[black,dashed] (4.195,-.7) -- (7.195,-.7);
    
    \draw [gray!75,fill=gray!20] (-4.195,-.7) rectangle ++(8.39,2.9);

    \node (fig) at (0,0.668) {    \includegraphics[width=8.21cm]{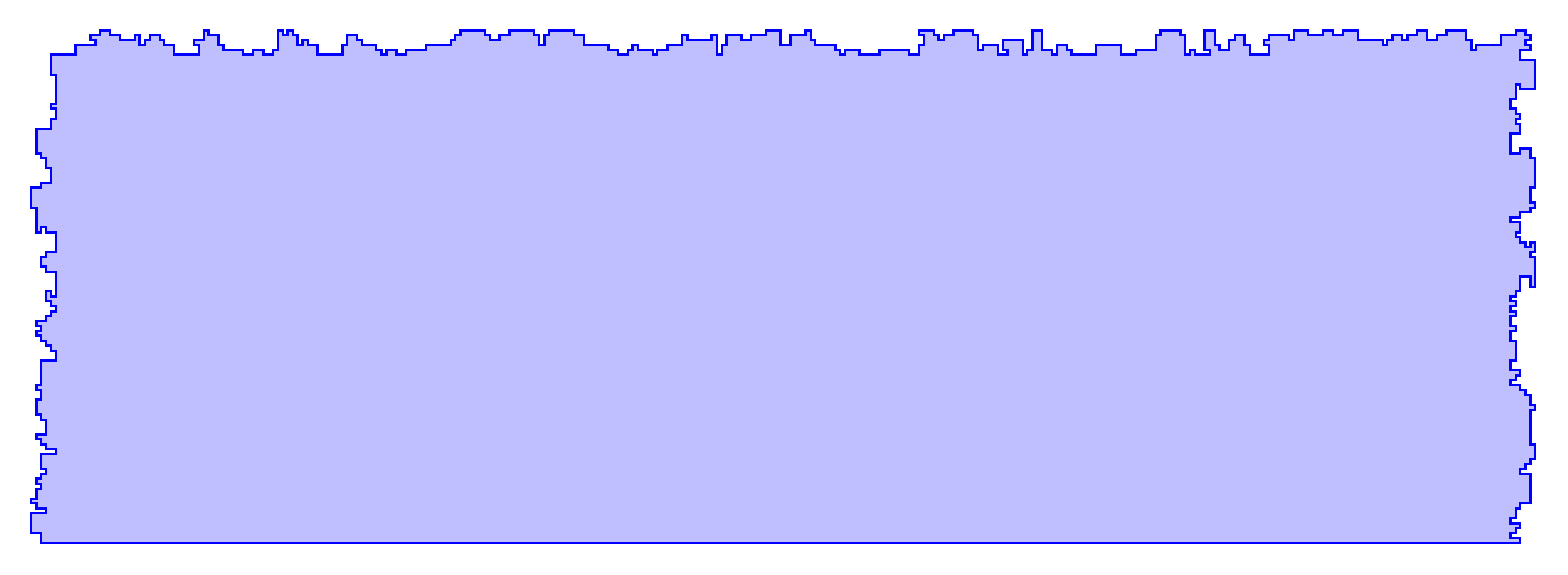}};

    \node[blue,font=\small] at (0.5,2.14) {$W$};
  
   \draw [purple!75!blue] (-3.7,-.69) -- (-3.7,1.75) -- (3.7,1.75) -- (3.7,-.69);
  
   \node[purple!75!blue,font=\small] at (0,1.46) {$R_0$};
   \node[gray!75,font=\small] at (0,2.42) {$R_1$};
  \draw[black,font=\tiny] (0,-.7) -- node[pos=1,below,yshift=3pt] {$(\frac{L}2,0)$} (0,-.8);
    
   \draw[black,font=\tiny] (-2.6,-.7) -- node[pos=1,below,yshift=3pt] {$\x_{\ell}$} (-2.6,-.8);

   \draw[black,font=\tiny] (2.6,-.7) -- node[pos=1,below,yshift=3pt] {$\x_r$} (2.6,-.8);
   
  \draw[black,font=\tiny] (0,-.7) -- node[pos=1,below,yshift=3pt] {$(\frac{L}2,0)$} (0,-.8);

    \draw [blue, ultra thick] (-3.9,-.7) --(-2.6,-.7);

    \draw [blue, ultra thick] (3.9,-.7) --(2.6,-.7);

    \draw[black,dotted] (-3.7,-0.7)--(-3.7,-1.25);
    \draw[black,dotted] (-2.6,-1.05)--(-2.6,-1.25); 
    \draw[black,dotted] (2.6,-1.05)--(2.6,-1.25); 
    \draw[black,dotted] (3.7,-0.7)--(3.7,-1.25);
    
    \draw[stealth-stealth,gray!75!black,font=\tiny] (-3.7,-1.25) -- node[pos=0.5,below,yshift=2pt] {$L^{2/3}$} (-2.6,-1.25);
    \draw[stealth-stealth,gray!75!black,font=\tiny] (-2.6,-1.25) -- node[pos=0.5,below,yshift=2pt] {$4L^{2/3}$} (2.6,-1.25);
\draw[stealth-stealth,gray!75!black,font=\tiny]     (3.7,-1.25) -- node[pos=0.5,below,yshift=2pt] {$L^{2/3}$} (2.6,-1.25);

    \draw[stealth-stealth,gray!75!black,font=\tiny] (-3.5,-.67) -- node[pos=0.5,right,xshift=-2pt] {$2L^{2/3}$} (-3.5,1.72);

   \end{tikzpicture}
\caption{The rectangles $R_0$ (purple) and $R_1$ (gray), and in between the region $W$ (blue) whose boundary conditions are raised via monotonicity to $H_L-n$ except on the segment between $\x_{\ell}$ and $\x_r$, where they are set to $H_L-n-1$.}
\label{fig:wiggle}
\end{figure}
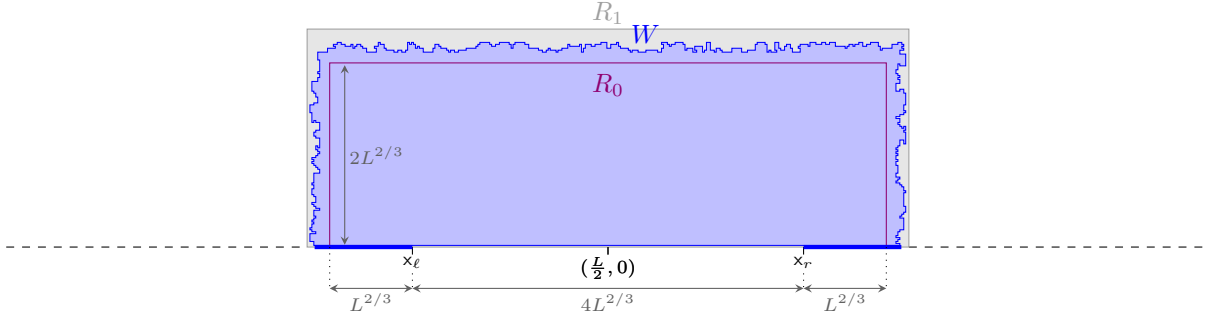

Moving to $\pi_{W}^{\legs}$ has two main advantages. The first is that there exists a \emph{unique} open $H_L-n$-contour $\gamma$ under this measure, and in particular, the sites of height $\geq H_L-n$ along $\gamma$ form a $*$-chain of sites of length larger than $(\log L)^2$ (typically of order $L^{2/3}$). Further, a standard Peierls argument shows that, w.h.p.\ under $\pi_{W}^{\legs}$, there are no such $*$-chains of length larger than $\log L$.
In particular, defining
\[
    \underline{\rho}^{\gamma}(x) := \min \{ y \geq 0 : (x,y) \in \gamma \} \text{ for $x \in \R$,}
\]
we have 
\begin{align}
    \pi_{W}^{\legs}(B_n) = \pi_{W}^{\legs}\big( \min_{x \in I_0} \underline{\rho}^{\gamma}(x) < \delta L^{1/3}\big) +o(1)\,.
    \label{eqn:Bn-opencontour}
\end{align}

The second advantage of moving to $\pi_{W}^{\legs}$ is that we may couple $\gamma$ under the no-floor measure $\widehat\pi_{W}^{\legs}$ with a modified Ising polymer law  in the half-space $\bP_{\H}^\x$ (see \cref{def:mod-ising-polymer}), for which we have \cref{mainthm:BM-for-Ising-polymers}.
After exhibiting this in \cref{claim:sos-legs-coupling}, we compare $\hatpilegs$ and $\pilegs$ via \cref{eqn:contourLawAreaTerm}.

We begin with the argument relating $\hat\pi_W^\legs$ to $\bP_{\H}^\x$.
From now on, we shift $W$ to the left so that $\x_{\ell} = 0$ and $\x_r = \lfloor 4L^{2/3} \rfloor$. Note that $\gamma$ is a contour in $W$ with start-point $\ostar \in (\Z^2)^*$ and end-point $\x:= \x_r - (1/2,1/2)$. We will couple $\hatpilegs$ with the Ising polymer law
$\bP_{\H}^{\x}$, defined using the SOS decoration weights $\Phi(\cC; \gamma) := \phi(\cC;\gamma)$ and $\Phi_{\H}(\cC;\gamma):= \phi(\cC;\gamma)\ind{\cC\subset \H}$ from \cref{prop:soscontour-law}.
This coupling is natural in light of the very far distance of $\fC$ from both $\ostar$ and $\x$.
Define $N:= \lfloor 4L^{2/3} \rfloor$. Recall the set $\cP_{D}(\u,\v)$ from \cref{subsec:notation}. 
Let $\bar{W}$ denote the region in $\R^2$ enclosed by $W$, i.e., the region enclosed by the $\Z^2$-edges connecting the boundary vertices of $W$.  Then
$\gamma \in \cP_{\bar{W}}(\ostar, \x)$.
Lower the boundary conditions of $\hatpilegs$ to $0,1,1,1$ using the shift-invariance of the no-floor model. Then \cref{obs:sos-is-ist} states that $\gamma \sim \hatpilegs$ has the Ising polymer law $\bP_{\bar{W}}^{\x}$, 
and the decoration weights are defined via  
$
\Phi(\cC; \gamma) := \phi(\cC;\gamma)$ and 
$\Phi_{\bar{W}}(\cC;\gamma):= \phi(\cC;\gamma)\ind{\cC\subset \bar{W}} 
$.

\begin{claim}\label{claim:sos-legs-coupling} 
Extend  $\hatpilegs$ to a probability measure on all contours in $\cP_{\H}(\ostar, \x)$ via $\hatpilegs(\gamma) = 0$ for all $\gamma \not \in \cP_{\bar{W}}(\ostar, \x)$.  For some constant $C>0$ and for all $\beta>0$ large enough, we have 
\[
 \big\| \bP_{\H}^{\x} - \hatpilegs \big\|_{\tv}
\leq C e^{-\beta C^{-1} N}\,.
\]
\begin{proof}[Proof of \cref{claim:sos-legs-coupling}]
We use the identification of the law of the contour induced by $\hatpilegs$ with $\bP_{\bar{W}}^{\x}$ throughout this proof. 

The main input for the proof is the following: for any $\delta \in (0,1)$ and for all $\beta,N>0$  large, 
\begin{align}
    \max\Big(\hatpilegs\big(|\gamma| \geq (1+\delta)N\big)\,,\, \bP_{\H}^{\x}\big(|\gamma| \geq (1+\delta)N\big) \Big) \leq e^{-\beta \delta N/2}\,.
    \label{eqn:legscoupling-length-bound}
\end{align}
We prove \cref{eqn:legscoupling-length-bound} now.
Using the decay of $\phi(\cC;\gamma)$ (\cref{prop:soscontour-law}(iv)) as in~\cref{eqn:c-beta-bound}, we have the following for any $\gamma : \ostar \to \x$:
\begin{align}
    \sum_{\cC \cap \Delta_{\gamma} \neq \emptyset } \abs{\phi(\cC;\gamma)} \leq \sum_{\y \in \Delta_{\gamma}} \, \sum_{\cC  \ni \y} \sup_{\gamma} |\phi(\cC;\gamma)| \leq 3|\gamma| e^{-(\beta-\beta_0)}\,.
    \label{eqn:legs-length-bound-1}
\end{align}
Since $|\gamma|$ is always larger than $N$, we have 
\begin{align}
    \min\Big( \Gb_{\H}(\x)\,,\,\Gb_{\bar{W}}(\x)\Big) \geq \exp\Big(-\big(\beta + 3e^{-(\beta -\beta_0)}\big) N\Big)\,. \label{eqn:legs-partition-fn-compare}
\end{align}
On the other hand, \cref{eqn:legs-length-bound-1} implies
\begin{multline*}
    \max\Big(\Gb_{\H}\big(\x \given |\gamma| \geq (1+\delta)N\big)\,,\,\Gb_{\bar{W}}\big(\x \given |\gamma| \geq (1+\delta)N\big) \Big) \\
    \leq \sum_{\substack{\gamma \subset \H \\ |\gamma| \geq (1+\delta)N}} \exp\Big( -(\beta - 3e^{-(\beta -\beta_0)})|\gamma| \Big) 
    \leq e^{-(\beta -c)(1+\delta)N}\,,
\end{multline*}
where $c>0$ is a constant independent of $\beta$ and $N$. The above two bounds yield \cref{eqn:legscoupling-length-bound}.  
Note that a consequence of~\eqref{eqn:legscoupling-length-bound} is
\begin{align}
    \hatpilegs\big(d_{L^{\infty}}(\gamma, \fC) \leq (1-2\delta)N \big) \leq e^{-\beta \delta N/2}\,.
    \label{eqn:legscoupling-distance-bound}
\end{align}

Now, fix any $\delta \in (0,1/2)$.  Observe that 
\[
    q_{\H}(\gamma)/q_{\bar{W}}(\gamma) = \exp \Big(\sum_{\cC \cap \Delta_{\gamma} \neq \emptyset} \phi(\cC;\gamma) \ind{\cC \subset \H \setminus W}\Big)
\]
For any contour $\gamma$ such that $d_{L^{\infty}}(\gamma,\fC) > (1-2\delta)N$, we have the following for all $\beta>0$ large:
\[
    \sum_{\cC \cap \Delta_{\gamma} \neq \emptyset} |\phi(\cC;\gamma)| \ind{\cC \subset \H \setminus W} \leq \sum_{\y \in\Delta_{\gamma}} \sum_{\cC \ni \y} \sup_{\gamma} |\phi(\cC;\gamma)| \leq |\gamma| e^{-c \beta N}\,,
\]
where in the last inequality, $c>0$ is a constant independent of $N$ and $\beta$ large, and the bound follows similarly to \cref{eqn:legs-length-bound-1}.
For any $|\gamma|<(1+\delta)N$, \cref{eqn:legs-length-bound-1} yields an upper bound on $\Gb_\H(\x)$ of $\exp(-(\beta - 3(1+\delta)e^{-(\beta -\beta_0)}) N)$. This upper bound, the lower-bound on $\Gb_{\bar{W}}(\x)$ in~\eqref{eqn:legs-partition-fn-compare}, and the previous display imply the following for any $|\gamma|<(1+\delta)N$:
\[
    \abs{\widehat{\pi}_{W}^{\legs}(\gamma) - \bP_{\H}^{\x}(\gamma)} \leq e^{-c' \beta N}\,,
\]
where $c'>0$  is a constant independent of $\beta$ and $N$. Since the number of connected paths in $(\Z^2)^*$ rooted at $\ostar$ of length at most $(1+\delta)N$ is trivially bounded $4^{(1+\delta)N}$, the previous display along with \cref{eqn:legscoupling-length-bound} imply that the total variation distance between $\widehat{\pi}_{W}^{\legs}$ viewed as a probability measure on $\cP_{\H}(\ostar, \x)$ and $\bP_{\H}^{\x}$ is bounded by $e^{-c''\beta N}$.
\end{proof}
\end{claim}

Now, from~\cref{eqn:contourLawAreaTerm}, we have 
\begin{equation}
    \pi_{\tilde R_0}^{\legs}(\gamma) \propto \widehat\pi_{\tilde R_0}^{\legs}(\gamma) \exp\left(- \frac{\lambda^{(n)}}{L}A(\gamma)+o(1)\right) \,,
    \label{eqn:floor-nofloor}
\end{equation}
where $A(\gamma)$ denotes the area under $\gamma$ in $W$.
From \cref{eqn:Bn-monotonicity} and \cref{eqn:Bn-opencontour},
the claim will be proven if we can show that for any $\ep>0$, there exists $\delta >0$ such that 
\begin{align}
    \pilegs(\min_{x \in I_0} \underline{\rho}^{\gamma}(x) < \delta L^{1/3}) <\epsilon 
    \label{eqn:prop3-goal}
\end{align}
for $L$ sufficiently large.
\cref{mainthm:BM-for-Ising-polymers} and  \cref{claim:sos-legs-coupling} imply that,
under horizontal rescaling by $|\x|_1 = \lfloor 4L^{2/3} \rfloor$ and vertical rescaling by $\lfloor 4L^{2/3} \rfloor^{1/2}$, $\gamma$ under $\hatpilegs$ converges weakly to a Brownian excursion as $L\to\infty$.
In particular, $\hatpilegs(\min_{x \in I_0} \underline{\rho}^{\gamma}(x) < \delta L^{1/3})$ can be made arbitrarily close to $0$ by taking $\delta$ small enough.  Letting $\tilde \E$ denote expectation with respect to  $\hatpilegs$, we have
\[
    \pilegs\big(\min_{x \in I_0} \underline{\rho}^{\gamma}(x) < \delta L^{1/3}\big) = \frac{\tilde \E\Big[\ind{\min_{x \in I_0} \underline{\rho}^{\gamma}(x) < \delta L^{1/3}} e^{-\frac{\lambda^{(n)}}{L}A(\gamma)+o(1)} \Big]}{\tilde \E\Big[e^{-\frac{\lambda^{(n)}}{L}A(\gamma)+o(1)}  \Big]}\,.
\]
The denominator is bounded away from $0$ because $\lambda^{(n)}$ is bounded for fixed $\beta$, and the Brownian excursion limit of $\gamma$ implies that, for any $c>0$, $\tilde{\E}[e^{-\frac{c}{L}A(\gamma)}]$ converges to $\E[e^{-cA(\xi)}]$ where $\xi$ has the law of a Brownian excursion on $[0,4]$
(see also  \cref{claim:area-tilt} below, where this is explained in more detail in the more delicate setting of \cref{thm:fs-detailed}).

$A(\gamma)$ is unlikely to be much bigger than $L^{2/3}\cdot L^{1/3} = L$ (making the expectation strictly positive). Hence, since the exponential in the numerator is bounded above by $1+o(1)$, we get
\[
\pilegs\big(\min_{x \in I_0} \underline{\rho}^{\gamma}(x) < \delta L^{1/3}\big) \leq C \hatpilegs\big(\min_{x \in I_0} \underline{\rho}^{\gamma}(x) < \delta L^{1/3}\big)\,,
\]
for some $C$ depending only on $\beta$ and $L$ large enough. Since the right-hand side can be made arbitrarily small by taking $\delta$ small enough, as discussed above, this implies \cref{eqn:prop3-goal}, thereby concluding the proof.
\end{proof}

\section{The effective random walk model, free Ising polymers, and limit theorems}
\label{sec:rw-1}

The factorization of $q(\Gamma)$  from~\eqref{eqn:decomposed-weight} and the input from Ornstein-Zernike theory (\cref{prop:oz-weights}) naturally leads one to consider the two-dimensional effective random walk defined in~\cref{subsec:rw-model}. In that subsection, we further expose the connection between this effective random walk and the cone-points of the \emph{free} Ising polymer model in $\H$. 
A key result there is \cref{prop:ist-repulsion}, which enables a quantitative comparison between the partition function $\Gb(\x \given \gamma \subset \H)$ and the probability that the random walk stays in $\H$.
In particular, we will be able to rule out bad events for the cone-points of the free Ising polymer in $\H$ using random walk estimates. 
In the next Section, we develop comparison results between  free and modified Ising polymers such that bad events for the cone-points of modified Ising polymers can also be ruled out using random walk estimates.
Eventually, this will lead to a coupling between the cone-points of the modified Ising polymer and our  random walk.

In the rest of this section, we develop various limit theorems for our random walk that will be used in the analysis of (modified) Ising polymers in \cref{sec:pf-BM-ist}, culminating in the proof of \cref{thm:BM-for-Ising-polymers-restated} (and thus \cref{mainthm:BM-for-Ising-polymers}).

\subsection{The random walk model and free Ising polymers}
\label{subsec:rw-model}
Consider the random walk $(\sS(i))_{i \in \Z_{\geq 0}}$, whose law will be denoted by $\P$, with i.i.d.\ increments $\{X(i) = (X_1(i), X_2(i))\}_{i \geq 1}$ of step distribution $X$, where for any $\vv \in \Z^2$, 
\begin{align}
    \P(X = \vv) = \sum_{ \Gamma \in \A } \Pa(\Gamma) \ind{\X(\Gamma) = \vv}\,.  
    \label{eqn:step-dist}
\end{align}
We will denote by $(\sS_1(\cdot),\sS_2(\cdot))$ the two coordinates of $\sS(\cdot)$. 
For $\u \in \Z^2$, we will write $\P_{\u}$ to denote the law of $\sS(\cdot)$ started from $\sS(0) = \u$.
From \cref{eqn:colinear-expectation}, we have $\E[X] = \alpha \x$, and from \cref{eqn:oz-exp-decay}, we inherit exponential tail decay: for all $\beta$ sufficiently large, there exists a constant $C':= C'(\beta)>0$ such that for all $k \geq 1$, we have
\begin{align}
    \sum_{\vec{v} \in  \Z^2: \norm{\vec{v}}_1 \geq k} \P(X = \vv) \leq C' e^{-\nu_g \beta k}\,.
    \label{eqn:rw-exp-tail}
\end{align}
We remark that the step-distribution of $X$ need not be symmetric in the $y$-coordinate $X_2$ due to the northeast splitting rule (this is the key difference between our random walk and the random walk considered in \cite{IOVW20}, and why we cannot use their random walk results here).
Also note that $\P(X \in \fcone) = 1$, and thus the $x$-coordinate $X_1 >0$ a.s.

We will write $H_{E}^A$ to denote the first hitting time of a set $E \subset \R^2$ by a stochastic process $A_.$, omitting the $A$ from the notation when the stochastic process is clear. For  singleton sets, we will simply write $H_{\u}^A$ rather than $H_{\{\u\}}^A$. Let $\H_{-} := \R \times \R_{<0}$.
In light of the factorization~\eqref{eqn:decomposed-weight}, for fixed $\Gamma^{(L)}$ and $\Gamma^{(R)}$, we will be interested  in our random walk $\sS(\cdot)$ started from $\Gamma^{(L)}$ and conditioned on the event
\[
    \{ H_{\x - \X(\Gamma^{(R)})} < H_{\H_{-}} \} \,.
\]
It will therefore be helpful to obtain an expression for $\P_{\u}(H_{\v} < H_{\H_-})$ in terms of animal weights, for any $\u, \v \in \Z^2$. Towards this end, define the sets 
\begin{align}
    \V^+_{\u, \v} &:= \{ (\vv_1, \dots, \vv_n) \in (\Z^2)^n : n \geq 1, \u+\vv_1+ \dots+ \vv_m = \v, \u+\vv_1+\dots+\vv_i \in \H \ ~\forall 1 \leq i \leq n \} \nonumber \\
    \A^+_{\u,\v} &:= \{ (\Gamma^{(1)}, \dots, \Gamma^{(n)}) \in \A^n : n \geq 1 , \big(\X(\Gamma^{(1)}), \dots, \X(\Gamma^{(n)})\big) \in \V^+_{\u, \v} \}\,,
    \label{def:V+,A+}
\end{align}
and, for a set of animals $E$, introduce 
\[
\fA^+(\u,\v ; E)  = \sum_{n\geq 1} \sum_{(\Gamma^{(1)}, \dots, \Gamma^{(n)}) \in \A_{\u,\v}^+} \prod_{i=1}^n \Pa(\Gamma^{(i)})  \ind{\Gamma^{(1)} \circ \dots \circ \Gamma^{(n)} \in E} \,.
\]
When $E$ is taken to be the set of all possible animals, we'll simply write $\fA^+(\u,\v ; E) = \fA^+(\u,\v)$. 
Observe that
\begin{align}
    \P_{\u}\big(H_{\v} < H_{\H_{-}} \big) 
    = \sum_{n \geq 1} \sum_{(\vv_1, \dots, \vv_n) \in \V^+_{\u, \v}} \prod_{i=1}^n \P(X=\vv_i)
    = 
    \fA^+(\u, \v)\,. \label{eqn:hitting-time-equation}
\end{align}
In words, the above equation states that the probability of the random walk hitting $\v$ before leaving $\H_{-}$ is equal to a sum over all possible products of weights of irreducible animals whose concatenation has all cone-points in $\H$.

Recall $\mathcal{A}(\cdot,\cdot;\cdot)$ from~\eqref{def:A-partition-fn}, and note that for any $\u, \v \in \Z^2$ and any set of animals $E$,
\begin{align}
    \fA(\u,\v;E) \leq \fA^+(\u,\v;E) \,, \label{eqn:U-U+}
\end{align}
since the (only) difference between $\fA^+(\u, \v;E)$ and $\fA(\u, \v;E)$ is that the sum defining $\fA(\u, \v;E)$ restricts to tuples of irreducible animals such that the \emph{entire concatenation} $\u+ \Gamma^{(1)} \circ \cdots \circ \Gamma^{(n)}$ is contained in $\H$, while the sum defining $\fA^+(\u, \v;E)$  requires only the \emph{cone-points} to stay in $\H$. In particular, $\Cpts(\Gamma)$ (viewed as an ordered tuple, see \cref{eqn:cone-points-vector}) under the conditioned law $\bP^{\x}(\cdot \given \Gamma^{(L)}, \Gamma^{(R)}, \gamma \subset \H)$ and the trajectory of the random walk $\P_{\X(\Gamma^{(L)})}(\cdot \given H_{\x- \X(\Gamma^{(R)})} < H_{\H_-})$ only differ due to these requirements.

Ultimately, we will achieve a coupling between $\Cpts(\Gamma)$ with the trajectory of the random walk  because of an entropic repulsion result (\cref{prop:animal-entropic-repulsion}): the cone-points of $\Gamma$ in the ``bulk'' of the strip $[0,N] \times [0,\infty)$ stay far away from the boundary of $\H$ with high probability, so that the aforementioned difference between $\mathcal{A}^+$ and $\mathcal{A}$ becomes negligible. A major step towards showing entropic repulsion is is the following crucial result from \cite{IST15}.\footnote{The notation in \cite{IST15} differs from ours: note that their quantity $\P_{\beta,+}^{\h_\x}(\u, \v)$ (defined in \cite[Eq.~5.10]{IST15}) is exactly our $\fA^+(\u,\v)$, which by \cref{eqn:hitting-time-equation} is nothing but $\P_\u(H_{\v}<H_{\H_-})$.}
\begin{proposition}[{\cite[Theorem~7]{IST15}}]\label{prop:ist-repulsion}
Recall $\nu_g$ from~\cref{prop:oz-weights}. There exists $\overline{\delta} \in (0,\nu_g/4)$ and $\beta_0>0$  such that for all $\beta > \beta_0$, there exists a constant $C:=C(\beta) > 0$ such that 
\[
    \sup_{\substack{\u, \x - \v \in \Y \\ \u, \v \in \H}} e^{-\overline{\delta} \beta ( \norm{\u}_1 + \norm{\x-\v}_1 )} \P_{\u}(H_{\v} < H_{\H_{-}}) \leq C e^{\tau_{\beta}(\x)} \Gb(\x \given \gamma \subset \H)\,.
\]
\end{proposition}
\cref{prop:ist-repulsion} combined with the exponential tail decay of $\Gamma^{(L)}$ and $\Gamma^{(R)}$~\eqref{eqn:oz-exp-decay} (note the significance of $\bar{\delta} < \nu_g/4$) will allow us to eliminate bad events  for $\Cpts(\Gamma)$ under the free Ising polymer law via estimates for random walks in a half-space, which are significantly easier to obtain compared to directly analyzing the polymer law.   

\subsection{Limit theorems for our random walk, confined to the half-space}
\label{subsec:rw-limit-theorem-statements}
We now state two important limit theorems,  \cref{thm:our-rw-inputs,thm:rw-invariance}, for our random walk $\sS(\cdot)$ confined to positive half-space, though our results apply to a more general class of two-dimensional random walks. 
Our  random walk $\sS(\cdot)$ is on $\Z^2$ and has i.i.d.\ increments $X$ with mean $(\mu,0)$ for some $\mu>0$, exponential tails, and satisfies $X_1 >0$ a.s. 
Let us also write $\Var(X) = (\sigma_1^2, \sigma_2^2)$.

Let $V_1$ denote the unique positive harmonic function for the one-dimensional random walk $\sS_2$ killed upon leaving $(0,\infty)$ satisfying 
\[
\lim_{a \to \infty} \frac{V_1(a)}{a}  = 1 \,.
\]
The uniqueness of $V_1$ was established by Doney in~\cite{Doney98}.
Similarly, let $V_1'$ denote the analogous harmonic function for the random walk $-\sS_2$.

\begin{theorem}\label{thm:our-rw-inputs}
    Fix any $A>0$ and any $\delta \in (0,1/2)$.
    Uniformly over $k \in [N/\mu - A\sqrt{N}, N/\mu + A \sqrt{N}]$ and $u, v \in (0, N^{1/2-\delta}] \cap \N$, 
    we have the following results. 
    \begin{enumerate}
        \item There exist constants $\mathbf{C}>0$ and $\kappa := \kappa(X)>0$ such that
    \begin{align}
        \P_{(0,u)}\Big(\sS(k) = (N,v) \,,\, H_{\H_{-}} > k \Big) \sim \mathbf{C} \kappa \frac{V_1(u) V_1'(v)}{k^2} \exp\Big(-\frac{(N-k\mu)^2}{2k \sigma_1^2 }\Big)\,.
        \label{eqn:our-rw-ballot}
    \end{align}
Furthermore, 
if
\[ p_{N,A} := \P_{(0,u)}\Big( H_{(N,v)}  \in [\tfrac{N}{\mu} - A\sqrt{N}, \tfrac{N}{\mu}+ A\sqrt{N}] \given H_{(N,v)} < H_{\H_{-}} \Big)\]
then 
\begin{align}
    \lim_{A \to \infty} \liminf_{N\to\infty} p_{N,A}= 1\,. \label{eqn:our-rw-goodk}
\end{align}
 
    \item  
    Let $\widehat{\sS}_2(\cdot)$ denote the denote the linear interpolation of the points $(n, \sS_2(n))_{n\in\N}$, viewed as an element of $C[0,\infty)$.
    The family of conditional laws 
    \[
        \Q_{u,v}^k(\cdot) := \P_{(0,u)}\bigg( \Big(\frac{\widehat{\sS}_2(tk)}{\sigma_2\sqrt{k}}\Big)_{t\in [0,1]} \in \cdot \given \sS(k) = (N,v) \,,\, H_{\H_{-}}>k \bigg)
    \]
    converges as $k \to \infty$ to the law of the standard Brownian excursion in $C[0,1]$.
\end{enumerate}
\end{theorem}

Next, we describe the diffusive limit of our random walk.
For $\ell, n\in \N$ and any ordered $\ell$-tuple of points $\mathcal S$ in $\R^2$, enumerated as 
\[
    \mathcal S = \big( (s_1(1), s_2(1)), (s_1(2), s_2(2)) , \dots, (s_1(\ell),s_2(\ell))\big) 
\]
and satisfying $s_1(i) < s_1(i+1)$ for each $i \in [1,\ell-1]$, we 
define $\fJ_n(\mathcal S)$ to be the linear interpolation through the  diffusively-rescaled points 
\begin{align}
    \fJ_n(\mathcal S) := \Big(\frac{1}{n}s_1(i), \frac{1}{\sigma\sqrt{n}}s_2(i) \Big)_{i=1}^k
    \label{def:J} 
\end{align}
where $\sigma^2 := \sigma_2^2/\mu$ (c.f. \cref{rk:curvature}).
Consider now the linear interpolation corresponding to our random walk $\sS(\cdot)$ up to time $H_{(N,v)}$:
\[
    \mathfrak{e}_N^{\sS,v} := \fJ_N\Big( (\sS(i))_{i =0}^{H_{(N,v)}} \Big)\,,
\]
Viewing $\mathfrak{e}_N^{\sS,v}$ as a random element of $C[0,1]$,
\Cref{thm:rw-invariance} below gives 
a convergence result to the Brownian excursion on $[0,1]$. 

\begin{theorem}\label{thm:rw-invariance}
Fix any $\delta \in (0,1/2)$. Uniformly over $u, v \in (0,N^{1/2-\delta}] \cap \N$, we have the following. 
The  family of conditional laws
\begin{align*}
    \mathbf{Q}_{u,v}^N (\cdot) := \P_{(0,u)}\Big(\big(\mathfrak{e}_N^{\sS,v}(t))_{t\in [0,1]} \in \cdot \given  H_{(N,v)} < H_{\H_{-}} \Big) \,, 
\end{align*}
converges weakly as $N \to \infty$  to the law of the standard Brownian excursion in $(C[0,1], \|\cdot\|_{\infty})$.
\end{theorem}

\cref{thm:our-rw-inputs,thm:rw-invariance} are proved in \cref{subsec:pf-our-rw-inputs}. These results
are modifications of a ballot-type theorem from~\cite{DenisovWachtel15} and an invariance principle from~\cite{DurajWachtel20}, respectively. Those important results were proved for a much broader class of random walks in very general cones; however, they were only proved for fixed start-points and end-points. The two results above have been modified to hold uniformly in appropriate ranges of start-points and end-points, which is crucial for our application. 
We remark that \cite[Section~5]{IOVW20} states and proves results that similarly modify the results of~\cite{DenisovWachtel15} and~\cite{DurajWachtel20}. However, first, the random walk they consider is symmetric in the $y$-coordinate, which simplifies their analysis; and second, their proofs do not always make explicit the aforementioned uniformity, in particular for the $x$-coordinate of the end-point of their random walk.
In \cref{sec:rw-half-space}, where our random walk estimates are proved, we take care to  describe explicitly how we modify the proofs in \cite{DenisovWachtel15} and \cite{DurajWachtel20}  to handle  uniformity in a broad range of start- and end-points.

\section{Proof of \texorpdfstring{\cref{mainthm:BM-for-Ising-polymers}}{Theorem 1.2} via random walk coupling}\label{sec:pf-BM-ist}

Throughout this Section, set $\x := \x(N) = (N,0) \in \Z^2$. Recall the notation set forth in \cref{subsec:notation}. 

In this section we prove \cref{thm:BM-for-Ising-polymers-restated}, thereby proving \cref{mainthm:BM-for-Ising-polymers}, via a certain coupling between $\Cpts(\Gamma)$ under the modified Ising polymer law $\bP_D^{\x}$ and the effective random walk, where $D = Q$ or $\H$. This is accomplished as follows.

In \cref{subsec:rw-model}, we described the particular role played by the free Ising polymer model in the half-space due to its direct connection with an effective random walk. 
It is then crucial for our analysis to be able to  compare the modified Ising polymers we are interested in with free Ising polymers.
In \cref{subsec:comparison-result}, we state our needed comparison results, \cref{the:ist-main} (the main theorem of \cite{IST15}) and \cref{prop:partition-fn-comparison}, the latter being an extension of the former that allows us to handle the $D = Q$ case. The key consequence of these comparisons is that contour events of small probability in one model are still small in the other model (\cref{cor:small-contour-probs}).

In \cref{subsection:proof-partition-fn-comparison}, we prove \cref{prop:partition-fn-comparison}, showing along the way that $\Gamma$ under the modified Ising polymer law in $D$ typically has many cone-points (\cref{lem:generalized-length-cpts}). In \cref{subsection:proof-surface-tension-equiv}, we prove \cref{prop:surface-tension-equiv} in a similar way to \cref{prop:partition-fn-comparison}.

In \cref{subsec:bounded-animals}, we show that each (left/right-)irreducible piece $\Gamma^{(i)}$ has size bounded by $(\log N)^2$, $i \in \{L, 1, \dots, n, R\}$, w.h.p.\ under $\bP_D^{\x}$.

In \cref{subsec:entropic-repulsion}, we show the key result enabling a coupling between $\Cpts(\Gamma)$ and the effective random walk, \cref{prop:animal-entropic-repulsion}. This result, which is an expression of entropic repulsion, states that, with high $\bP_D^{\x}$-probability, each cone-point of $\Gamma$ stays above height $N^{\delta}$ in the large interval $[N^{4\delta}, N-N^{4\delta}]$, for any $\delta \in (0,1/4)$. This is achieved via  \cref{cor:small-contour-probs}, which reduces the result to the result in the free polymer model; and \cref{prop:ist-repulsion}, which will allow us to reduce the result to an entropic repulsion result for the random walk conditioned to stay in $\H$.

The entropic repulsion, the shape of the cones $\fcone$ and $\bcone$, and the bound $|\Gamma^{(i)}| \leq (\log N)^2$ implies that the entire animal stays bounded away from $\partial\H$, the boundary of the half-plane,
 in the strip $[N^{4\delta}, N-N^{4\delta}] \times [0,\infty)$. The first consequence of this is that the weight modifications play no role in this strip; that is, for each cluster $\cC$ of $\Gamma$ in this strip, we have $\Phi_D(\cC;\gamma) = \Phi(\cC;\gamma)$. Thus, the portion of $\Gamma$ contained in this strip behaves like a free Ising polymer, which can be related to the random walk. 
The second consequence is that the requirement that $\gamma$ stays in $D$ in this strip becomes trivial. In particular, following the discussion after \cref{eqn:U-U+}, we will be able to couple the cone-points of $\Gamma$ in this strip with an effective random walk. This is achieved in \cref{subsec:random-walk-coupling}.

\cref{thm:BM-for-Ising-polymers-restated} is finally proved in \cref{subsec:pf-BM-for-Ising-polymers}.

\subsection{Comparing Ising polymers with modified and unmodified weights}
\label{subsec:comparison-result}
We begin with the main theorem of \cite{IST15}, which states that the partition functions of the Ising polymers in $\H$, with and without modifications, are equivalent up to a $\beta$-dependent constant.

\begin{theorem}[{\cite[Theorem~1]{IST15}}]
\label{the:ist-main}
For all $\beta>0$ large enough, we have
\begin{align}
    \tau_{\beta}(\x) = - \lim_{N\to\infty} \frac{1}{\norm{\x}} \log \Gb_{\H}(\x)\,. \label{eqn:ist-main-tension}
\end{align}
In fact, 
there exist constants $C_1:=C_1(\beta), C_2:=C_2(\beta)>0$ such that for all $\beta$ large enough, 
\begin{align}
    C_1\, \Gb(\x \given \gamma \subset \H) \leq \Gb_{\H}(\x) \leq C_2\, \Gb(\x \given \gamma \subset \H)\,. \label{eqn:ist-main-partitionfn}
\end{align}
\end{theorem}

Though we only consider $\H$ in this article, \cite[Theorem~1]{IST15} is stated more generally for half-spaces whose interior normal has argument lying in $[-\pi/4, 3\pi/4]$.

The next proposition can be seen as an extension of \cref{the:ist-main} to address  polymers in $Q$.
\begin{proposition}\label{prop:partition-fn-comparison}
    For all $\beta>0$ sufficiently large, there exists a constant $C:= C(\beta)>0$ such that for any  $N\in \N$,
    \begin{align}
        C^{-1}\,\Gb(\x \given \gamma \subset \H ) \leq \Gb(\x \given \gamma \subset Q) \leq C\,\Gb(\x \given \gamma \subset \H ) \,,\label{eqn:partition-fn-comparison-freeQH}
    \end{align}
    and 
    \begin{align}
        C^{-1}\,\Gb(\x \given \gamma \subset Q ) \leq \Gb_Q(\x) \leq C\Gb(\x \given \gamma \subset Q ) \,.\label{eqn:partition-fn-comparison}
    \end{align}
\end{proposition}
The proof of \cref{prop:partition-fn-comparison} is given in \cref{subsection:proof-partition-fn-comparison}.

One can readily imply from \cref{eqn:ist-main-partitionfn,eqn:partition-fn-comparison} the following, relating $\bP^{\x}(\cdot \mid \gamma\subset D)$ and $\bP_D^{\x}$, exactly as \cite{IST15} derived its analogue from \cref{eqn:ist-main-partitionfn} (see the discussion below Eq.~(2.7) in that paper): 
\begin{corollary}\label{cor:small-contour-probs}
Let $D = \H$ or $Q$. There exists a constant $C:= C(\beta)>0$ such that for any set of contours $A$ contained in $D$, 
\[
\bP_D^{\x}(A) \leq C \sqrt{\bP^{\x}(A \given \gamma \subset D)} \quad \text{ and } \quad \bP^{\x}(A \given \gamma \subset D) \leq C\sqrt{\bP_D^{\x}(A)} \,.
\]
\end{corollary}

For instance, if $Y_\gamma := \sum_{\cC \cap \nabla_{\gamma} \neq \emptyset} \Phi_D'(\cC;\gamma) - \Phi'(\cC;\gamma)$ and $\mathbf{E}[\cdot]$ denotes expectation under $\bP^{\x}(\cdot\mid\gamma\subset D)$, then 
$\mathbf{E}[e^{Y_{\gamma}}]$ is just $\Gb_D(\x)/\Gb(\x \mid  \gamma \subset D)$, and therefore $\bP_D^\x(A) = \mathbf{E}[\one_{A} e^{Y_\gamma} ] / \mathbf{E}[e^{Y_\gamma}]$.
Notice $C^{-1} \leq \mathbf{E}[e^{Y_{\gamma}}] \leq C$ by \cref{eqn:ist-main-partitionfn} (for $D=\H$) and \cref{eqn:partition-fn-comparison} (for $D=Q$). 
By Cauchy--Schwarz,  $\bP_D^\x(A) \leq C \sqrt{\bP^\x(A\mid \gamma\subset D )}\sqrt{\mathbf{E}[e^{2Y_\gamma}]}$, and the last expectation is  uniformly bounded (as $\mathbf{E}[e^{2Y_{\gamma}}] =\tilde{\Gb}_D(\x)/\Gb(\x \mid \gamma \subset D)$, where $\tilde{\Gb}_D(\x)$ is defined w.r.t.\ $\tilde{\Phi}'_D(\cC;\gamma) := 2\Phi_D'(\cC;\gamma)-\Phi'(\cC;\gamma)$). 

A particular consequence of \cref{cor:small-contour-probs} is that contour events that hold with probability tending to $1$ in the unmodified models (which are much easier to study) still hold with probability tending to $1$ in the modified models. The most significant application of this fact is in the proof of entropic repulsion under $\bP_Q^{\x}(\cdot)$ (\cref{prop:animal-entropic-repulsion}), which, as explained at the start of this Section, is a crucial step towards the random walk coupling achieved in \cref{subsec:random-walk-coupling}.

\begin{remark}
In \cref{cor:small-contour-probs}, we specified that the event $A$ should be a set of contours. The reason this was emphasized is that in \cref{eqn:ising-polymer-animals} we extended the measure $\bP_D^{\x}$ from contours to animals, but \cref{cor:small-contour-probs} does not hold for general animal events $A$.
Indeed, it is \emph{not} true that small animal events in one model stay small in another model, as the ratio $|q(\Gamma)/q_D(\Gamma)|$ is not bounded away from $0$ nor $\infty$ (as there are no such bounds on $|\Phi'_Q(\cC;\gamma)|/|\Phi'(\cC;\gamma)|$ in the generality of weight modifications that we must consider).
\end{remark}

\subsection{Proof of \cref{prop:partition-fn-comparison}}
\label{subsection:proof-partition-fn-comparison}

We will need two auxiliary lemmas. 
The  first, 
\cref{lem:G-H-smallGammaL}, states that in any model where at least two cone-points exist, the partition function is dominated by animals whose 
first (left) and last (right) irreducible pieces have size of order $1$. 

\begin{lemma}\label{lem:G-H-smallGammaL}
Let $D$ be either $Q$ or $\H$. For any $\ep \in (0,1)$, there exists $K_{\ep} := K_{\ep,\beta} >0$ such that
\begin{align}
    \Gb_{D}(\x \given |\Cpts(\Gamma)|\geq 2 \,,\, \max(|\Gamma^{(L)}|, |\Gamma^{(R)}|) > K_{\ep}) \leq \ep \Gb_{\H}(\x) \label{eqn:smallGammaL-D}
\end{align}
and 
\begin{align}
    \Gb(\x \given \gamma \subset D  \,,\,  |\Cpts(\Gamma)|\geq 2 \,,\,  \max(|\Gamma^{(L)}|, |\Gamma^{(R)}|) > K_{\ep}) \leq \ep \Gb(\x \given \gamma \subset \H)\,. \label{eqn:smallGammaL-unmodified}
\end{align}
\end{lemma}

\begin{remark}\label{rk:upgrade-H-to-D}
    Observe that once \cref{prop:partition-fn-comparison} is proved, the domain $\H$ in the right-hand side of \cref{eqn:smallGammaL-D,eqn:smallGammaL-unmodified} may be replaced by $D$, for $D$ either $Q$ or $\H$.
\end{remark}

The second auxiliary result, \cref{lem:generalized-length-cpts}, extends \cref{lem:ist-length-cpts}, and thus \cref{prop:many-cone-points}  thanks to \cref{rk:cone-points-general}, to the modified models on $Q$ and $\H$. 

\begin{lemma}\label{lem:generalized-length-cpts}
Let $D$ be either $Q$ or $\H$. Fix any $\ep,\delta \in (0,1)$.
There exist positive constants $\beta_0, \delta_0$, and $\nu_0>0$ such that  for all $\beta \geq \beta_0$, there exists $C:= C(\beta)>0$ such that the following hold uniformly over all $N \in \N$ and $r \geq 1+\ep$:
    \begin{align}
        \Gb_D\big(\x \given |\gamma| \geq r \norm{\x}_1\big) 
        &\leq Ce^{-\nu_0 \beta r \norm{\x}_1} \Gb_{D}(\x)
        \label{eqn:generalized-length-bound}
        \\
        \Gb_D\big(\x \given |\Cpts(\gamma)|< 2\delta_0 \norm{\x}_1\big)
        &\leq C e^{- \nu_0 \beta\norm{\x}_1} \Gb_D(\x)
        \label{eqn:generalized-many-cone-points}
    \end{align}
In particular, there exists $\nu>0$ such that for all $\beta \geq \beta_0$, there exists $C:= C(\beta)>0$ such that the following holds uniformly over all $N\in\N$:
\begin{equation}
\Gb_D(\x \given |\Cpts(\Gamma)| < \delta_0\|\x\|_1) 
\leq
Ce^{-\nu \beta \|\x\|_1}\Gb_D(\x) \,.
\label{eqn:generalized-many-animal-cone-points}
\end{equation}
\end{lemma}

We postpone the proof of \cref{lem:G-H-smallGammaL} to the end of the subsection. We prove \cref{prop:partition-fn-comparison} and \cref{lem:generalized-length-cpts} in tandem, proceeding as follows:
\begin{enumerate}[(a)]
\item \label[part]{it:partition-fn-comparison-cpts-part-a} We begin by proving \cref{lem:generalized-length-cpts} for $D=\H$.
\item \label[part]{it:partition-fn-comparison-cpts-part-b} We then prove the first part of \cref{prop:partition-fn-comparison} (\cref{eqn:partition-fn-comparison-freeQH}), as well as its analog for $\Gb_{\H}$, i.e.,
\begin{align}
    C^{-1}\Gb_{\H}(\x)\leq \Gb_{\H}(\x \given \gamma \subset Q) \leq C\Gb_{\H}(\x)\,. \label{eqn:partition-fn-comparison-3}
\end{align}
\item \label[part]{it:partition-fn-comparison-cpts-part-c} These results are then used to prove \cref{lem:generalized-length-cpts} for $D=Q$.
\item \label[part]{it:partition-fn-comparison-cpts-part-d} Finally, we prove \cref{eqn:partition-fn-comparison}, thereby completing the proof of \cref{prop:partition-fn-comparison}.
\end{enumerate}
\begin{proof}[Proof of \cref{it:partition-fn-comparison-cpts-part-a}]
For each bound, we will first apply \cref{lem:ist-length-cpts}, then \cref{the:ist-main}.
Recall that \cref{lem:ist-length-cpts} states that the analogous bounds 
to \cref{eqn:generalized-length-bound} and \cref{eqn:generalized-many-cone-points} hold for $\Gb(\x)$.
\cref{eqn:ist-main-tension} implies that $\log \Gb(\x)$ and $\log \Gb(\x\given \gamma \subset \H)$ have the same leading order; thus, for some $\nu_0' \in (0, \nu_0]$, we have 
\begin{align*}
        \Gb\big(\x \given \gamma \subset \H \,,\, |\gamma| \geq r \norm{\x}_1\big) 
        &\leq 
        \Gb\big(\x \given |\gamma| \geq r \norm{\x}_1\big)
        \leq Ce^{-\nu_0' \beta r \norm{\x}_1} \Gb(\x \given \gamma \subset \H)\,, \text{ and}
        \\
        \Gb\big(\x \given \gamma \subset \H \,,\,  \Cpts(\gamma) < 2\delta_0 \norm{\x}_1 \big) 
        &\leq 
        \Gb\big(\x \given |\Cpts(\gamma)| < 2\delta_0 \norm{\x}_1 \big)
        \leq C e^{-\nu_0' \beta \norm{\x}_1} \Gb(\x \given \gamma \subset \H)\,. 
\end{align*}
\cref{eqn:generalized-length-bound} and \cref{eqn:generalized-many-cone-points} for $D =\H$ now follow from \cref{cor:small-contour-probs},
and \eqref{eqn:generalized-many-animal-cone-points}   from \cref{rk:cone-points-general}. 
\end{proof}

\begin{proof}[Proof of \cref{it:partition-fn-comparison-cpts-part-b}]
The upper-bounds in both \cref{eqn:partition-fn-comparison-freeQH,eqn:partition-fn-comparison-3} hold with $C = 1$, since $Q \subset \H$.  

We now  show the lower bound for \cref{eqn:partition-fn-comparison-3}. The lower bound in \cref{eqn:partition-fn-comparison-freeQH}  is obtained simply by replacing $q_{\H}$ with $q$ in what follows.
By \cref{lem:G-H-smallGammaL,lem:generalized-length-cpts} (for $D=\H$), we have the following: for any $\ep \in(0,1)$, there exists $K:=K(\beta,\ep)>0$ such that for $\|\x\|_1$ large,
\begin{align}
    (1-\ep)\Gb_{\H}(\x) \leq \Gb_{\H}\big(\x \given |\Cpts(\Gamma)| \geq \delta_0 \norm{\x}_1 \,,\,  \max(|\Gamma^{(L)}|, |\Gamma^{(R)}|)\leq K \big) \,, \label{eqn:partition-fn-comparison-1-1}
\end{align}
Now, using the weight factorization of $q_{\H}(\Gamma)$ in \cref{eqn:q-factorization}, we may express the right-hand-side of the above (for $N=\norm{\x}_1 \geq 2/\delta_0$) as
\begin{multline*}
    \!\!\!\!\!
    \sum_{\substack{\Gamma^{(L)} \in \AL\\ \gamma^{(L)} \subset \H}}\sum_{\substack{\Gamma^{(R)} \in \AR\\ \gamma^{(R)} \subset \H}} q_{\H}(\Gamma^{(L)})\ q_{\H}(\Gamma^{(R)})\ind{|\Gamma^{(L)}|, |\Gamma^{(R)}|\leq K} \\
    \times \sum_{n\geq \delta_0 \|\x\|_1} \sum_{\Gamma^{(1)}, \dots, \Gamma^{(n)} \in \A} \prod_{i=1}^n q_{\H}(\Gamma^{(i)}) 
    \ind{\substack{\Gamma^{(1)} \circ \cdots \circ \Gamma^{(n)} \in \cP_{\H}(\X(\Gamma^{(L)}), \x- \X(\Gamma^{(R)}))}}
    \,.
\end{multline*}
For any $\u \in \fcone$, let $\gamma^{(L)}(\u)$ denote an arbitrary up-right path from $0$ to $\u$, and let $\gamma^{(R)}(\u)$ denote an arbitrary down-right path from $\x-\u$ to $\x$ (if one exists).
Since the backwards cone of $\Gamma^{(L)}$ must contain the origin and the forwards cone of $\x-\Gamma^{(R)}$ must contain $\x$, it follows that 
\[
    \gamma^{(L)}(\X(\Gamma^{(L)})) \circ \gamma_1 \circ \cdots \circ \gamma_m \circ \gamma^{(R)}(\X(\Gamma^{(R)})) \subset Q
\]
for any collection $\{\Gamma^{(L)}, \Gamma^{(1)} ,\dots, \Gamma^{(m)}, \Gamma^{(R)}\}$ contributing to a nonzero term in the second-to-last display. In particular, for any $\Gamma^{(L)} \in \AL$ and $\Gamma^{(R)} \in \AR$ such that $\{|\Gamma^{(L)}|, |\Gamma^{(R)}|\leq K\}$, we have 
\begin{multline}
q_{\H}(\gamma^{(L)}(\X(\Gamma^{(L)})))  q_{\H}(\gamma^{(R)}(\X(\Gamma^{(R)}))) \\
\times \sum_{m\geq 1} \sum_{\Gamma^{(1)}, \dots, \Gamma^{(m)} \in \A} \prod_{i=1}^m q_{\H}(\Gamma^{(i)}) \ind{\substack{\Gamma^{(1)} \circ \cdots \circ \Gamma^{(m)} \in \cP_{\H}(\X(\Gamma^{(L)}), \x- \X(\Gamma^{(R)}))}}  
\leq \Gb_{\H}(\x \given \gamma \subset Q) \,. 
\label{eqn:clusterless-bound}
\end{multline}
Since 
$
    q_{\H}(\Gamma^{(L)}) \leq q_{\H}(\gamma^{(L)}(\X(\Gamma^{(L)})))
$
and
$q_{\H}(\Gamma^{(R)}) \leq q_{\H}(\gamma^{(R)}(\X(\Gamma^{(R)}))) 
$,
it follows from the last three displays that the right-hand side of \cref{eqn:partition-fn-comparison-1-1} is bounded above by
\begin{align*}
    \sum_{\substack{\Gamma^{(L)} \in \AL\\ \gamma^{(L)} \subset \H}}
    \frac{q_{\H}(\Gamma^{(L)})}{q_{\H}(\gamma^{(L)}(\X(\Gamma^{(L)})))} \ind{ |\Gamma^{(L)}|\leq K} 
    \sum_{\substack{\Gamma^{(R)} \in \AR\\ \gamma^{(R)} \subset \H}} \frac{q_{\H}(\Gamma^{(R)})}{q_{\H}(\gamma^{(R)}(\X(\Gamma^{(R)})))}\ind{ |\Gamma^{(R)}|\leq K} 
    \Gb_{\H}(\x \given \gamma \subset Q) \\
    \leq 
    \Big(\sum_{\u \in \Z^2 :\|\u\|_1 \leq K} e^{\beta \|\u\|_1}\Gb_{\H}(\u) \Big)^2 \Gb_{\H}(\x \given \gamma \subset Q)\,.
\end{align*}
The pre-factor in the last display is a constant depending only $K:=K(\beta,\ep)$, finishing the proof.
\end{proof}

\begin{proof}[Proof of \cref{it:partition-fn-comparison-cpts-part-c}]
In \cref{it:partition-fn-comparison-cpts-part-a}, we proved \cref{eqn:generalized-length-bound,eqn:generalized-many-cone-points} for $D=\H$. Bounding below the left-hand sides of these equations by adding the restriction $\gamma\subset Q$, and giving an upper bound on the right-hand sides by using \cref{eqn:partition-fn-comparison-3} to replace $\Gb_{\H}(\x)$ by $C\Gb_{\H}(\x \given \gamma \subset Q)$ yields the same results for $\Gb_{\H}(\x \given \gamma\subset Q, \cdot)$:
    \begin{align}
       \Gb_{\H}\big(\x \given \gamma \subset Q\,,\, |\gamma| \geq 1.1 \norm{\x}_1\big) 
        &\leq Ce^{-\nu_0 \beta \norm{\x}_1} \Gb_{\H}(\x \given \gamma \subset Q)\,, \text{ and}
        \label{eqn:ist-length-bound-Q}
        \\
        \Gb_{\H}\big(\x \given \gamma \subset Q\,,\, |\Cpts(\gamma)|  < 2\delta_0 \norm{\x}_1\big)
        &\leq C e^{-\nu_0 \beta \norm{\x}_1} \Gb_{\H}(\x \given \gamma \subset Q)
        \label{eqn:ist-cpts-Q}
    \end{align}
    uniformly in $\beta$ large enough and $\x$. 
    Similar to \cref{eqn:modified-weight-comparison}, we have
    \begin{align}
        \abs{ \log \frac{q_{\H}(\gamma)}{q_{Q}(\gamma)}} \leq 6e^{-\chi\beta} |\gamma|\,. \label{eqn:compare-weights-Q}
    \end{align}
    \cref{eqn:generalized-length-bound} and \cref{eqn:generalized-many-cone-points} for $D=Q$ are then simple consequences of \cref{eqn:ist-length-bound-Q,eqn:ist-cpts-Q,eqn:compare-weights-Q}. Once again, \cref{eqn:generalized-many-animal-cone-points} for $D=Q$ follows by \cref{rk:cone-points-general}. 
\end{proof}

\begin{proof}[Proof of \cref{it:partition-fn-comparison-cpts-part-d}]
Let us start by proving the upper bound in \cref{eqn:partition-fn-comparison}. As in \cref{eqn:partition-fn-comparison-1-1}, \cref{lem:G-H-smallGammaL,lem:generalized-length-cpts} (now for $D=Q$) yield the following for any $\ep \in (0,1)$, some $K:= K(\beta,\ep)>0$, and all $\|\x\|_1$ large enough: 
\begin{align}
    \Gb_{Q}(\x) \leq \ep \Gb_{\H}(\x) + \Gb_{Q}\big(\x \given |\Cpts(\Gamma)| \geq \delta_0 \|\x\|_1 \,,\,  \max(|\Gamma^{(L)}|, |\Gamma^{(R)}|)\leq K \big) \,. \label{eqn:partition-fn-comparison-1-2}
\end{align}
The first term on the right-hand side, $\ep\Gb_{\H}(\x)$, may be bounded  by $C_0(\beta)\Gb(\x \given \gamma \subset Q)$ via \cref{eqn:ist-main-partitionfn,eqn:partition-fn-comparison-freeQH}.
We now follow the proof of \cref{it:partition-fn-comparison-cpts-part-b} to bound the second term: using the factorization \cref{eqn:qQ-factorization} of $q_Q(\Gamma)$, expanding the right-hand side as in \cref{it:partition-fn-comparison-cpts-part-b}, and replacing $\Gamma^{(L)}$ and $\Gamma^{(R)}$ by cluster-less up-right paths $\gamma^{(L)}(\X(\Gamma^{(L)}))$ and $\gamma^{(R)}(\X(\Gamma^{(R)}))$, we obtain an upper bound by $C_1\Gb_{\H}(\x \given \gamma \subset Q)$, which by \cref{eqn:ist-main-partitionfn,eqn:partition-fn-comparison-freeQH} is bounded above by $C_1'\Gb(\x \given \gamma \subset Q)$.  

For the lower bound in \cref{eqn:partition-fn-comparison}, start instead from $\Gb(\x \given \gamma \subset Q)$. Bound it from above  by $C_2\Gb(\x \given \gamma \subset \H)$ using \cref{eqn:partition-fn-comparison-freeQH}, and bound the latter by $C_2'\Gb_{\H}(\x)$ using \cref{eqn:ist-main-partitionfn}. 
Now, use the inequality in \cref{eqn:partition-fn-comparison-1-1}, expand the right-hand side exactly as in the proceeding display, and again replace $\Gamma^{(L)}$ and $\Gamma^{(R)}$ by cluster-less, up-right paths.
Observe that the bound in \cref{eqn:clusterless-bound} holds with the right-hand side replaced with $\Gb_Q(\x)$, since all animals on the left-hand side are contained in $Q$ (and therefore the $q_{\H}$-weights of these animals are equal to the $q_Q$-weights). Following the remainder of the proof of \cref{it:partition-fn-comparison-cpts-part-b}, we obtain the lower bound in \cref{eqn:partition-fn-comparison}.
\end{proof}

\begin{proof}[Proof of~\cref{lem:G-H-smallGammaL}]
For any $\u,\v \in \Z^2$ and for any set of animals $E$, define the partition functions
    \begin{align}\label{def:general-partition-fn}
        \Gb_D(\u \to \v) := \sum_{ \substack{\gamma : \u \to \v\\ \gamma \subset D}} q_D(\gamma)\quad \text{and} \quad \Gb_D(\u \to \v \given E) := \sum_{ \Gamma \in \cP_D( \u , \v) \cap E} q_D(\Gamma) \,.
    \end{align}
We begin with three estimates pertaining to $\Gb_D(\u \to \v)$.

First, we have the analogue of \cref{eqn:oz-exp-decay} for the modified animal weights: for all $\beta>0$ sufficiently large, there exists a constant $c':=c'(\beta)>0$ such that for all $k, N\geq 1$,
\begin{align}
    \sum_{\Gamma \in \AL \cup \AR} e^{\h_{\x} \cdot \X(\Gamma)} q_D(\Gamma) \ind{|\Gamma|\geq k} \leq c' e^{-\nu_g \beta k} \,.
    \label{eqn:gen-oz-exp-decay}
\end{align}
This is proved at the end of \cref{appendix:ornstein--Zernike}. 

Next, below Theorem 3 of \cite{IST15}, it is noted that \cref{eqn:ist-main-partitionfn} in the current article
actually holds if the end-points of the contour are not on the line $\partial \H$; that is, for any $\u,\v \in \H$, we have 
\begin{align}
 C_1(\beta) \Gb(\u \to \v\given \gamma \subset \H) \leq \Gb_{\H}(\u\to \v) \leq C_2(\beta) \Gb(\u \to \v\given \gamma \subset \H)\,. \label{eqn:generalized-IST}
\end{align}
Similarly, in Sections~4.1 and~4.2 of \cite{IST15}, it is shown that contours with linear size and many cone-points dominate $\Gb(\x \given \gamma \subset \H)$, and the arguments  there yield that the same is true when the end-points of the contour are not on $\partial \H$; that is, for any fixed $\ep\in (0,1)$, there exist constants $\delta_0, \nu, c >0$ such that the following bounds hold uniformly over $\beta>0$ sufficiently large, over $\u,\v \in \H$ satisfying $\v\in \u +\fcone_{\delta}\setminus\{0\}$, and over $r \geq 1+\ep$:
    \begin{align*}
        \Gb(\u \to \v \given \gamma \subset \H \,,\, |\gamma| \geq r \norm{\v-\u}_1 ) \leq ce^{-\nu \beta \norm{\v-\u}_1} \Gb(\u \to \v \given \gamma \subset \H)\,.
    \end{align*}
    and 
    \begin{align*}
        \Gb(\u \to \v \given \gamma \subset \H \,,\, |\Cpts(\gamma)| < 2\delta_0 \|\v-\u\|_1 ) \leq ce^{-\nu \beta \norm{\v-\u}_1} \Gb(\u \to \v \given \gamma \subset \H)\,.
    \end{align*}
As a consequence of \cref{rk:cone-points-general}, we have many cone-points for animals $\Gamma$ as well:
    \begin{align}
        \Gb(\u \to \v \given \gamma \subset \H \,,\, |\Cpts(\Gamma)|  < \delta_0 \|\v-\u\|_1) \leq ce^{-\nu \beta \norm{\v-\u}_1} \Gb(\u \to \v \given \gamma \subset \H)\,.
        \label{eqn:cone-points-general-startend}
    \end{align}

Now, towards  \cref{eqn:smallGammaL-D}, we only show
$\Gb_{D}(\x \given |\Cpts(\Gamma)|\geq 2 \,,\, |\Gamma^{(L)}| > K_{\ep}) \leq \ep \Gb_{D}(\x)$, as the analogous bound with $\Gamma^{(R)}$ instead of $\Gamma^{(L)}$ follows via the same argument.
We begin by using \cref{eqn:q-factorization,eqn:qQ-factorization} to obtain the expansion
\begin{multline}
    \Gb_{D}(\x \given |\Cpts(\Gamma)|\geq 2\,,\, |\Gamma^{(L)}| > K_{\ep})= \sum_{\substack{\u \in \fcone \\ \x-\v \in \fcone}} 
    \sum_{\substack{\Gamma^{(L)}\in\AL \\ \Gamma^{(L)} \in \cP_D(\u)}} q_{D}(\Gamma^{(L)}) \ind{|\Gamma^{(L)}| >K_{\ep}} 
     \\
    \times \sum_{\substack{ \Gamma^{(R)}\in\AR \\ \Gamma^{(R)} \in \cP_D(\v ,\x)}} q_{D}(\Gamma^{(R)}) 
    \Big( \sum_{n\geq 1} \sum_{\substack{\Gamma^{(1)}, \dots, \Gamma^{(n)} \in \A \\ \Gamma^{(1)} \circ \cdots \circ \Gamma^{(n)} \in \cP_{\H}(\u,\v)}} \prod_{i=1}^n q_{\H}(\Gamma^{(i)}) \Big) \,. \label{eqn:G-H-smallGammaL-1}
\end{multline}
The above double sum over $n$ and irreducible animals is bounded by $\Gb_{\H}(\u \to \v)$, which by \cref{eqn:generalized-IST} is bounded by $C_2(\beta)\Gb(\u \to \v \given \gamma \subset \H)$. It then suffices to show 
\begin{align}
    \Gb(\u \to \v \given \gamma \subset \H ) \leq Ce^{\h_{\x}\cdot (\u+\x-\v) + \bar{\delta}\beta(\|\u\|_1 +\|\x-\v\|_1)} \Gb(\x \given \gamma \subset \H)\,. \label{eqn:G-H-smallGammaL-2}
\end{align}
where $C:= C(\beta)>0$ is a constant and $\bar{\delta}$ is as in \cref{prop:ist-repulsion}.
Indeed, substituting the bound in \cref{eqn:G-H-smallGammaL-2} into \cref{eqn:G-H-smallGammaL-1}, using the exponential tail from \cref{eqn:gen-oz-exp-decay}, and recalling $\bar{\delta}< \nu_g/4$~yields 
\[
\Gb_{D}(\x \given |\Cpts(\Gamma)|\geq 2 \,,\, |\Gamma^{(L)}|> K_{\ep}) \leq C' e^{-\tfrac{\nu_g}2 \beta K_{\ep}} \Gb(\x \given \gamma \subset \H) \,.
\]
By \cref{the:ist-main}, the right-hand side above is bounded by $\ep \Gb_{\H}(\x)$ for $K_{\ep}:= K_{\ep, \beta}>0$ large enough.

To show \cref{eqn:G-H-smallGammaL-2}, we use the existence of cone-points \cref{eqn:cone-points-general-startend} (for which we write a $1+o(1)$, where the $o(1)$ term vanishes as $\|\v-\u\|_1$ tends to infinity) and the weight factorization  in \cref{eqn:decomposed-weight}:
\begin{align*}
     &\Gb\Big(\u \to \v \given \gamma \subset \H\Big) \\
     \leq \
     &(1+o(1)) \!\!\!\!\!\! \sum_{\substack{\u'\in \u+\fcone \\ \x-\v-\v' \in \fcone}} \sum_{\substack{\Gamma^{(L)}: \u \to \u' \in \AL \\ \Gamma^{(R)}: \v' \to \v \in \AR}} \!\! q(\Gamma^{(L)}) q(\Gamma^{(R)}) 
     \sum_{n \geq 1}
     \sum_{\substack{\Gamma^{(1)} ,\dots,\Gamma^{(n)} \in \A }} \prod_{i=1}^n q(\Gamma^{(i)}) \ind{\Gamma^{(1)} \circ \cdots \circ \Gamma^{(n)} \in \cP(\u',\v')} \\
     = \ &(1+o(1)) \ e^{-\h_{\x}\cdot (\v-\u)}  \sum_{\substack{\u'\in \u+\fcone \\ \x-\v-\v' \in \fcone}} \sum_{\substack{\Gamma^{(L)}: \u \to \u' \in \AL \\ \Gamma^{(R)}: \v' \to \v \in \AR}} \Pa(\Gamma^{(L)}) \Pa(\Gamma^{(R)})  
     \fA(\u', \v')  \\
     \leq \ &(1+o(1)) \ e^{\h_{\x}\cdot (\u+\x-\v)} \Gb(\x \given \gamma \subset \H) \\
     &\quad\times \Big(\sum_{\substack{\u'\in \u+\fcone}}  \sum_{\substack{\Gamma^{(L)}: \u \to \u' \in \AL }}   \Pa(\Gamma^{(L)})e^{\overline{\delta}\beta 
     \|\u'\|_1} 
     \Big)
     \Big(\sum_{\substack{ \x-\v-\v' \in \fcone}} \sum_{\substack{\Gamma^{(R)}: \v' \to \v \in \AR}}  \Pa(\Gamma^{(R)})  e^{\overline{\delta}\beta
     \|\x-\v'\|_1 
     }\Big) \\
     \leq \ &(1+o(1)) \ e^{\h_{\x}\cdot (\u+\x-\v) + \bar{\delta}\beta(\|\u\|_1 +\|\x-\v\|_1)}
     \Gb(\x \given \gamma \subset \H) \\
    &\quad\times\Big(\sum_{\Gamma^{(L)} \in \AL}   \Pa(\Gamma^{(L)})e^{\overline{\delta}\beta \|\X(\Gamma^{(L)})\|_1}\Big)
     \Big(\sum_{\Gamma^{(R)} \in \AR}   \Pa(\Gamma^{(R)})e^{\overline{\delta}\beta \|\X(\Gamma^{(R)})\|_1}\Big)
\end{align*}
In the second-to-last line we used \cref{eqn:hitting-time-equation,eqn:U-U+} and then \cref{prop:ist-repulsion} to bound the $\fA(\u',\v')$ term, as well as  $\h_{\x} \cdot \x = \tau_{\beta}(\x)$. In the third-line, we used the triangle inequality.
\Cref{eqn:G-H-smallGammaL-2} then follows from the exponential tails of $\Gamma^{(L)}$ and $\Gamma^{(R)}$ in \cref{eqn:oz-exp-decay}, again recalling $\bar{\delta}<\nu_g/4$. This yields \cref{eqn:smallGammaL-D}. The proof of \cref{eqn:smallGammaL-unmodified} is identical: just replace the left-hand side of \cref{eqn:G-H-smallGammaL-1} by $\Gb(\x \given |\Cpts(\Gamma)|\geq 2 , |\Gamma^{(L)}|>K_{\ep}, \gamma \subset D)$, and then  replace all modified weights $q_D$ by $q$.
\end{proof}

\subsection{Proof of \cref{prop:surface-tension-equiv}}
\label{subsection:proof-surface-tension-equiv}
Fix a direction $\vec{u} \in \bbS^{d-1}$, and let $t_N^-(\vec{u}):= (N, y_N) \in \Z^2$ denote the lattice point whose $x$-coordinate is equal to $N$ and is on or below the line $\mathrm{span}(\vec{u})$. 
Let $\ostar:= (1/2, 1/2)$ be the origin of the dual lattice $(\Z^2)^*$, and define $\x_N(\vec{u}):= t_N^-(\vec{u})+(-1/2,1/2) \in (\Z^2)^*$.
Then we have the formula
\[
\tau_{\beta}^{\SOS}(\vec{u}) = \lim_{N\to\infty} \lim_{M \to \infty} -\frac{1}{\beta \norm{\x_N(\vec{u})}} \log \Gb_{\cI_{N,M}}(\x_N(\vec{u}))\,,
\]
where 
$\cI_{N,M}$ denotes the strip $[0,N] \times [-M,M]$.
\cite[Theorem~4.16]{DKS92}
implies that the surface tension $\tau_{\beta}$ is equal to the surface tension of the model in $\cI_{N,M}$ with the same (free) weights $q(\gamma)$, that is:\footnote{
The proof of that result in the book \cite{DKS92} contains an error, corrected in the appendix to \cite{IST15}. The mistake had to do with a part of the result that is irrelevant to us, namely, a part that claims an unchanged surface tension even after modification of the decoration functions $\Phi(\cC;\Gamma)$. In our application of \cite[Theorem~4.16]{DKS92}, we are not making any modifications to the decoration functions.} 
    \[
        \tau_{\beta}(\vec{u}) = \lim_{N\to\infty} \lim_{M\to\infty} -\frac{1}{\beta \norm{\x_N(\vec{u})}} \log \Gb(\x_N(\vec{u})\given \gamma \subset \cI_{N,M})\,. 
    \]
Thus, the proof will be finished if we can compare to exponential order the  partition function with modified weights $\Gb_{\cI_{N,M}}(\x_N(\vec{u}))$ and the partition function with free weights $\Gb(\x_N(\vec{u})\given \gamma \subset \cI_{N,M})$. The proof of this comparison is extremely similar to the proof of \cref{eqn:partition-fn-comparison} above, though much simpler, and so it is omitted for brevity's sake.\footnote{
A crucial input for  the proof of \cref{eqn:partition-fn-comparison} was  \cref{prop:many-cone-points}, which we proved using \cref{lem:ist-length-cpts}, cited from the paper of \cite{IST15}. Note \cite{IST15} was written for Ising polymers, so in particular, \cref{p:P4} of the surface tension is always assumed in that paper.
However, we use \cref{prop:surface-tension-equiv} in order to prove \cref{p:P4} of $\tau_\beta$, so this might appear to be circular reasoning. However, \cref{lem:ist-length-cpts} is proved in \cite{IST15} (Eq.~(4.5) there) without assuming \cref{p:P4}, so there is no issue in using it. No other results from \cite{IST15} are needed for the proof of \cref{prop:surface-tension-equiv}.}  \qed

\subsection{Boundedness of the irreducible pieces}\label{subsec:bounded-animals}

The main result of this subsection is \cref{prop:bounded-irreducible-pieces}, which states that each irreducible piece of $\Gamma$ has size bounded by $(\log N)^2$ with high $\bP_D^{\x}(\cdot)-$probability.

When $\Gamma$ has at least two cone-points, recall from \cref{eqn:irreducible-decomposition} that we write $\Gamma^{(1)}, \dots, \Gamma^{(n)}$ to denote the irreducible components of $\Gamma$, where  $n := |\Cpts(\Gamma)|-1$.
\begin{proposition} \label{prop:bounded-irreducible-pieces}
For $D=Q$ or $D=\H$,
\begin{align*}
    \bP_D^\x \big( \{|\Cpts(\Gamma)|\geq 2\} \,,\, \{\exists  i  \in \{L, 1, \dots,  |\Cpts(\Gamma)|-1 ,  R\} : |\Gamma^{(i)}| \geq (\log N)^2 \} \big) = o(1)\,.
\end{align*}
\end{proposition}

We will need the following two results on the asymptotic size of the partition functions of interest. 
Recall the notation for asymptotic relations set out at the start of \cref{sec:polymers-oz}.

\begin{lemma}\label{lem:full-partition-function-size}
    There exists some $C:= C(\beta)>0$ 
    \begin{align}
    \Gb(\x) \leq Ce^{-\tau_{\beta}(\x)}\,. \label{eqn:full-partition-function-size}
    \end{align}
\begin{proof}
Similar to \cref{eqn:good-decomposed-2}, we have from \cref{eqn:hitting-time-equation}
\begin{align*}
    e^{\tau_{\beta}(\x)} \Gb(x \given |\Cpts(\Gamma)| \geq 2) &= \sum_{\Gamma^{(L)} \in \AL} \sum_{\Gamma^{(R)} \in \AR} \Pa(\Gamma^{(L)}) \Pa(\Gamma^{(R)}) \P_{\X(\Gamma^{(L)})}(H_{\x- \X(\Gamma^{(R)})} < \infty) \\
    &\leq \Big(\sum_{\Gamma^{(L)} \in \AL} \Pa(\Gamma^{(L)}) \Big) \Big( \sum_{\Gamma^{(R)} \in \AR} \Pa(\Gamma^{(R)}) \Big) \,.
\end{align*}
The above summations are finite by \cref{claim:oz-1}. The Lemma then follows from~\cref{eqn:ist-many-cone-points}.
\end{proof}
\end{lemma}
 We remark that a precise first-order asymptotic for $\Gb(\x)$ is given in \cite[Eq.~(4.12.3)]{DKS92}.
The asymptotic can be proved by showing $\P_{\X(\Gamma^{(L)})}(H_{\x- \X(\Gamma^{(R)})} < \infty)$ is of order $|\x|^{-1/2}$ via random walk estimates, similar to the proof of \cref{thm:our-rw-inputs}.

\begin{lemma}\label{lem:partition-function-size}
We have 
\begin{align}
    \Gb(\x \given \gamma \subset \H) \gtrsim  e^{-\tau_{\beta}(\x)}N^{-3/2}  \,, \label{eqn:partition-function-size}
\end{align}
where the implied constant depends on $\beta$.

\begin{proof}
Item (1) of \cref{thm:our-rw-inputs} shows that for any positive, fixed $u$ and $v$ in $\N$ (independent of $N$), we have
\begin{align}
    \P_{(0,u)}(H_{(N,v)} < H_{\H_{-}}) \sim \mathbf{C} \kappa \frac{V_1(u) V_1'(v)}{N^{3/2}}\,. \label{eqn:hitting-ballot-asymp}
\end{align}
The lower-bound of \cref{eqn:partition-function-size} then follows immediately from \cref{prop:ist-repulsion} and \cref{eqn:hitting-ballot-asymp}.
\end{proof}
\end{lemma}

It is not too hard to show a matching upper-bound, so that $\Gb(\x  \given \gamma\subset \H) \asymp e^{-\tau_{\beta}(\x)} N^{-3/2}$. Again, we do not pursue this here as the lower-bound suffices.

\begin{proof}[Proof of \cref{prop:bounded-irreducible-pieces}]
Define the set of animals
\begin{align}
\mathcal{P}_{D}^{\cp, \len}(\x, K) := 
\{ \Gamma \in \cP_D(\x) : |\Cpts(\Gamma)| \geq 2 \,,\, \max(|\Gamma^{(L)}|, |\Gamma^{(R)}|) \leq K \,,\, |\Gamma| \leq 1.1 \norm{\x}_1 \}\,.
    \label{def:animals-cp-len-K}
\end{align}
In \cref{lem:G-H-smallGammaL,rk:upgrade-H-to-D,lem:generalized-length-cpts}, it was shown that this set of animals dominates $\bP_D^{\x}$, i.e.,
\begin{align*}
    \lim_{K\to\infty} \lim_{N\to\infty} \bP_D^{\x}\big(\mathcal{P}_{D}^{\cp,\len}(x,K)\big) = 1 \,.
\end{align*}
Therefore, in light of  \cref{lem:partition-function-size}, 
the proof of \cref{prop:bounded-irreducible-pieces} will be complete upon showing
\begin{align}
     \Gb_D\big( \x \given \cP_D^{\cp,\len}(\x,K) \,,\, \{\exists 1 \leq i \leq |\Cpts(\Gamma)|-1 : |\Gamma^{(i)}| \geq (\log N)^2 \} \big) = o\big(e^{-\tau_{\beta}(\x)}N^{-3/2}\big)
     \label{eqn:bounded-irr-animals-goal}
\end{align}
for any fixed $K$.

Define the sets of animals
\begin{align}
    \A_{\mathsf{L},K} := \{ \Gamma \in \AL : |\Gamma| < K \}
    \quad \text{ and } \quad 
    \A_{\mathsf{R},K} := \{\Gamma \in \AR : |\Gamma| < K \}\,.
    \label{def:ALK-ARK}
\end{align}
Using the factorization of weights 
(\cref{eqn:q-factorization}, and \cref{eqn:qQ-factorization} for $D=Q$),
we have 
\begin{align*}
    &\Gb_D 
    \big( \x \given \mathcal{P}_{D}^{\cp, \len}(\x, K) \,,\, \{\exists 1 \leq i \leq |\Cpts(\Gamma)|-1 : |\Gamma^{(i)}| \geq (\log N)^2 \} \big) \nonumber \\
    &\quad =  \sum_{\substack{\Gamma^{(L)} \in \A_{\mathsf{L},K} \\ \gamma^{(L)} \subset D}} q_D(\Gamma^{(L)})
    \ \sum_{\substack{\Gamma^{(R)} \in \A_{\mathsf{R},K} \\ \gamma^{(R)} \subset D}} 
    q_D(\Gamma^{(R)}) 
    \sum_{n \geq 1}
    \  \sum_{\Gamma^{(1)}, \dots, \Gamma^{(n)} \in \A} 
    \ \prod_{i=1}^n q_{\H}(\Gamma^{(i)}) \\
    &\qquad \times \ind{\Gamma^{(1)} \circ \dots \circ \Gamma^{(n)} \in \cP_{\H}(\X(\Gamma^{(L)}), \x- \X(\Gamma^{(R)}))} 
    \ind{\exists 1 \leq k \leq n : |\Gamma^{(k)}| \geq (\log N)^2} \nonumber \\ 
    &\quad\leq   
    \sum_{\substack{\Gamma^{(L)} \in \A_{\mathsf{L},K} \\ \gamma^{(L)} \subset D}} q_D(\Gamma^{(L)})
    \ \sum_{\substack{\Gamma^{(R)} \in \A_{\mathsf{R},K} \\ \gamma^{(R)} \subset D}} 
    q_D(\Gamma^{(R)}) \nonumber \\
    &\qquad \times \sum_{\substack{ \u \in \X(\Gamma^{(L)})+\fcone \\ \v \in \x - \X(\Gamma^{(R)}) + \bcone }} 
    \Gb_{\H}\big( \X(\Gamma^{(L)}) \to \u \big)
    \Gb_{\H}\big( \v \to \x - \X(\Gamma^{(R)}) \big) \sum_{\Gamma \in \A \cap \cP_{\H}(\u, \v)} q_{\H}(\Gamma) \ind{|\Gamma|\geq (\log N)^2}\,.
\end{align*}
For any $\mathsf{a}, \mathsf{b} \in \H$,  \cref{eqn:generalized-IST,eqn:full-partition-function-size}
along with the trivial bound
$
    \Gb(\mathsf{a} \to \mathsf{b} \given \gamma \subset \H) \leq \Gb(\mathsf{b}-\mathsf{a})\,.
$
implies
$
    \Gb_{\H}(\mathsf{a} \to \mathsf{b}) \leq C e^{\tau_{\beta}(\mathsf{b}-\mathsf{a})}  
$
for some constant $C:=C(\beta)>0$.
Define the weights $W_{\H}^{\h}(\Gamma) := e^{\h\cdot\X(\Gamma)}q_{\H}(\Gamma)$, under which $\Gamma$ has exponential tails by \cref{eqn:gen-oz-exp-decay}.
Then 
\begin{align*}
    &\Gb_D 
    \big( \x \given \cP_{D}^{\cp,\len}(\x,K) \,,\, \{\exists 1 \leq i \leq |\Cpts(\Gamma)|-1 :  
 |\Gamma^{(i)}| \geq (\log N)^2 \} \big) \nonumber \\
    &\leq  C^2 
    \sum_{\substack{\Gamma^{(L)} \in \A_{\mathsf{L},K} \\ \gamma^{(L)} \subset D}} q_D(\Gamma^{(L)}) e^{\tau_{\beta}(\X(\Gamma^{(L)}))}
    \ \sum_{\substack{\Gamma^{(R)} \in \A_{\mathsf{R},K} \\ \gamma^{(R)} \subset D}} 
     q_D(\Gamma^{(R)}) e^{\tau_{\beta}(\X(\Gamma^{(R)}))} 
     \\
    &\times 
    \sum_{\substack{ \u \in \X(\Gamma^{(L)})+\fcone \\ \v \in \x - \X(\Gamma^{(R)}) + \bcone }} 
    \Big(\sum_{\Gamma \in \A} W_{H}^{\h_{\v-\u}}(\Gamma) \ind{|\Gamma|\geq (\log N)^2} \Big)
    \\
    &\times \exp\Big(-\tau_{\beta}\big(\X(\Gamma^{(L)})\big)- \tau_{\beta}\big(\u - \X(\Gamma^{(L)})\big) - \tau_{\beta}(\v-\u) - \tau_{\beta}\big(\x - \X(\Gamma^{(R)}) - \v\big)-\tau_{\beta}\big(\X(\Gamma^{(R)})\big) \Big)\,.
\end{align*}
The exponential tails in \cref{eqn:gen-oz-exp-decay} implies both
\[
    \sum_{\substack{\Gamma^{(L)} \in \A_{\mathsf{L},K} \\ \gamma^{(L)} \subset D}} q_D(\Gamma^{(L)})e^{\tau_{\beta}(\X(\Gamma^{(L)}))}
    \quad \text{and} \quad \sum_{\substack{\Gamma^{(R)} \in \A_{\mathsf{R},K} \\ \gamma^{(R)} \subset D}} q_D(\Gamma^{(R)})e^{\tau_{\beta}(\X(\Gamma^{(R)}))}
\]
are bounded by a constant $C_K>0$, while
\[
    \sum_{\Gamma \in \A} W_{\H}^{\h_{\v-\u}}(\Gamma) \ind{|\Gamma|\geq (\log N)^2}  \leq c' e^{-\nu_g \beta (\log N)^2}\,.
\]
These bounds, the strong triangle inequality (\cref{prop:surface-tension-facts}), and the fact that there are at most $|0+\fcone \cap \x +\bcone \cap \Z^2| \leq (N+1)^2$ possible values for $\u$ and for $\v$ yield
\begin{multline*}
    \Gb_Q 
    \big( \x \given  \{\exists 1 \leq i \leq |\Cpts(\Gamma)|-1 : |\Gamma^{(i)}| \geq (\log N)^2 \} \big)
    \leq C'_K(N+1)^4 e^{-\nu_g \beta(\log N)^2} e^{-\tau_{\beta}(\x)} \,,
\end{multline*}
for some constant $C'_K>0$.
Thus, we have \cref{eqn:bounded-irr-animals-goal}, finishing the proof.
\end{proof}

\subsection{Entropic repulsion of the animal}
\label{subsec:entropic-repulsion}
Fix any $\delta \in (0,1/4)$, and define the rectangles  $\mathfrak{R}_0(\delta) := [N^{4\delta}, N- N^{4\delta}] \times [0, 2N^{\delta}]$ and 
$\mathfrak{R}(\delta) := [N^{4\delta}, N- N^{4\delta}] \times [0, N^{\delta}]$.
The main result in this subsection is the entropic repulsion result \cref{prop:animal-entropic-repulsion}, which states that, under $\bP_D^{\x}(\cdot)$, the cone-points of $\Gamma$ do not intersect $\mathfrak{R}(\delta)$ with probability  tending to $1$ as $N \to \infty$.

\begin{proposition}\label{prop:animal-entropic-repulsion}
Let $D=Q$ or $D=\H$. For any $\delta \in (0,1/4)$,
\[
    \lim_{N\to\infty} \bP_D^{\x}\big( \Cpts(\Gamma) \cap \mathfrak{R}(\delta)  \neq \emptyset \big) = 0\,.  
\]
\end{proposition}
From the shape of $\fcone$ and the $(\log N)^2$ bound on each $|\Gamma^{(i)}|$ (\cref{prop:bounded-irreducible-pieces}), \cref{prop:animal-entropic-repulsion} is enough to deduce that, with high $\bP_D^{\x}$-probability, the entire animal (contour and clusters) stays away from $\partial \H$ in a slightly shorter rectangle~\eqref{eqn:entire-animal-repulsion}. We will use this fact in the following subsection to couple $\Cpts(\Gamma)$ with the random walk.
We begin with an analogous result in the free Ising polymer case, where we can exploit the connection with the random walk and associated estimates, \`a la \cref{prop:ist-repulsion}.

\begin{lemma}\label{lem:entropic-repulsion-unmodified}
Let $D=Q$ or $D=\H$.
For any $\delta \in (0,1/8)$, 
\[
    \lim_{N\to\infty} \bP^{\x} \big( \Gamma \cap \mathfrak{R}(\delta)  \neq \emptyset \given \gamma \subset D \big) = 0\,.  
\]
\begin{proof}
In light of \cref{prop:bounded-irreducible-pieces} and the shape of $\fcone$, it suffices to show that the cone-points of $\Gamma$ avoid the larger rectangle $\mathfrak{R}_0(\delta)$ with high probability, i.e.,
\begin{align*}
    \lim_{N\to\infty} \bP^{\x} \big( E_{\delta} \given \gamma \subset D \big) = 0 \,, \quad \text{ where } E_{\delta} :=  \{ \Cpts(\Gamma) \cap \mathfrak{R}_0(\delta) \neq \emptyset \}\,.
\end{align*}
Recall $\cP_D^{\cp,\len}(\x,K)$ from~\eqref{def:animals-cp-len-K}, and observe that \cref{lem:G-H-smallGammaL,lem:generalized-length-cpts,rk:upgrade-H-to-D} imply 
\begin{align*}
    \lim_{K\to\infty} \lim_{N\to\infty} \bP^{\x}\big(\mathcal{P}_{D}^{\cp,\len}(\x,K) \given \gamma \subset D \big) = 1 \,.
\end{align*}
From the previous two displays, the proof of the Lemma will be complete upon showing
\begin{align}
    \lim_{K\to\infty} \lim_{N\to\infty} \bP^{\x}\big(\mathcal{P}_{D}^{\cp,\len}(\x,K) , E_{\delta} \given \gamma \subset D \big) = 0\,.
    \label{eqn:entropic-repulsion-goal}
\end{align}

Now, fix $K \geq 1$ and $N$ large compared to $K$. 
Recall the sets $\A_{\mathsf{L},K}$ and $\A_{\mathsf{R},K}$ from~\eqref{def:ALK-ARK}.
Using \cref{eqn:good-decomposed-2}, we have
\begin{multline*}
    \bP^{\x}\big(\cP_D^{\cp,\len}(\x, K) , E_{\delta} \given \gamma \subset D \big) \\
    \leq 
    \frac{1}{\Gb(\x \given \gamma \subset D) } e^{-\tau_{\beta}(\x)} 
    \sum_{\substack{\Gamma^{(L)}\in \A_{\mathsf{L},K} \\ \gamma^{(L)} \subset D}} \Pa(\Gamma^{(L)})
    \ \sum_{\substack{\Gamma^{(R)} \in \A_{\mathsf{R},K} \\ \gamma^{(R)} \subset D}} 
    \Pa(\Gamma^{(R)})  \ \fA\big( \X(\Gamma^{(L)}), \x- \X(\Gamma^{(R)}) ;  E_{\delta} \big)\,.
\end{multline*}
Using \cref{eqn:partition-fn-comparison-freeQH}, \cref{eqn:U-U+} and \cref{prop:ist-repulsion}, we may bound from above the right-hand side of the previous display 
\begin{align}
    &C^{-1}\!\!\! \sum_{\substack{\Gamma^{(L)} \in \A_{\mathsf{L},K} \\ \gamma^{(L)} \subset D}} \!\!\!\! \Pa(\Gamma^{(L)}) e^{\overline{\delta} \beta \|\X(\Gamma^{(L)})\|_1}
    \! \sum_{\substack{\Gamma^{(R)} \in \A_{\mathsf{R},K} \\ \gamma^{(R)} \subset D}}  \!\!\!\!
    \Pa(\Gamma^{(R)}) e^{\overline{\delta} \beta \|\X(\Gamma^{(R)})\|_1} \  \frac{\fA^+\big( \X(\Gamma^{(L)}), \x- \X(\Gamma^{(R)}) ; E_{\delta} \big)}{\P_{X(\Gamma^{(L)})}(H_{\x-\X(\Gamma^{(R)})} > H_{\H_{-}})} \nonumber \\
    =&~ C^{-1}\!\!\! \sum_{\substack{\Gamma^{(L)} \in \A_{\mathsf{L},K} \\ \gamma^{(L)} \subset D}} \Pa(\Gamma^{(L)}) e^{\overline{\delta} \beta \|\X(\Gamma^{(L)})\|_1}
    \ \sum_{\substack{\Gamma^{(R)} \in \A_{\mathsf{R},K} \\ \gamma^{(R)} \subset D}} 
    \Pa(\Gamma^{(R)}) e^{\overline{\delta} \beta \|\X(\Gamma^{(R)})\|_1} \nonumber 
    \\
    &\times \P_{\X(\Gamma^{(L)})}\big( \exists i < H_{\x -\X(\Gamma^{(R)})} : \sS_1(i) \in [N^{4\delta}, N-N^{4\delta}] \,,\, \sS_2(i) \in [0, 2N^{\delta}] \given H_{\x-\X(\Gamma^{(R)})} < H_{\H_{-}}\big) .
    \label{eqn:entropic-repulsion-1}
\end{align}
The ballot-type result \cref{thm:our-rw-inputs}(1) and standard random walk estimates yield
\begin{align*}
    \lim_{N\to\infty} \P_{\X(\Gamma^{(L)})}\big( \exists i < H_{\x -\X(\Gamma^{(R)})} : \sS_1(i) \in [N^{4\delta}, N-N^{4\delta}] \,,\, \sS_2(i) \in [0, 2N^{\delta}] \given H_{\x-\X(\Gamma^{(R)})} < H_{\H_{-}}\big) \\
    = 0 \qquad
\end{align*}
uniformly over $\Gamma^{(L)}\in \A_{\mathsf L, K}$ and $\Gamma^{(R)} \in \A_{\mathsf R, K}$. 
On the other hand, the exponential tails in \cref{eqn:oz-exp-decay} and the relation $4\overline{\delta} < \nu_g$ yield
\[
    C^{-1} \Big(\sum_{\substack{\Gamma^{(L)} \in \A_{\mathsf{L},K} \\ \gamma^{(L)} \subset Q}} \Pa(\Gamma^{(L)}) e^{\overline{\delta} \beta \norm{\X(\Gamma^{(L)})}_1} \Big) 
    \Big( \sum_{\substack{\Gamma^{(R)} \in \A_{\mathsf{R},K} \\ \gamma^{(R)} \subset Q}} 
    \Pa(\Gamma^{(R)}) e^{\overline{\delta} \beta \norm{\X(\Gamma^{(R)})}_1} \Big) \leq C_K\,,
\]
for some constant $C_K:= C_{K,\beta} > 0$. Taking the limit as  $N$, then $K\to\infty$ in ~\eqref{eqn:entropic-repulsion-1} yields~\eqref{eqn:entropic-repulsion-goal}.
\end{proof}
\end{lemma}

\begin{proof}[Proof of \cref{prop:animal-entropic-repulsion}]
\cref{lem:entropic-repulsion-unmodified} and \cref{cor:small-contour-probs} imply that $\gamma$ under $\bP_D^{\x}(\cdot)$ stays above $\fR(\delta)$. 
Of course, all cone-points of $\Gamma$ lie on $\gamma$, so we are done.
\end{proof}

\subsection{Coupling with the effective random walk} \label{subsec:random-walk-coupling}
Let $D=Q$ or $D=\H$, and
fix $\delta \in (0,1/8)$.
The entropic repulsion result, \cref{prop:animal-entropic-repulsion}, sets the stage for a coupling between the  cone-points of $\Gamma$ lying in the strip 
$\cS_{N^{4\delta}, N-N^{4\delta}}:= [N^{4\delta}, N-N^{4\delta}]\times [0,\infty)$ and the random walk $\sS(\cdot)$ defined in \cref{subsec:rw-model}. Before defining this coupling explicitly, we set up some notation.

For $\u,\v \in \Z^2$, define
\begin{align*}
    &\cP_{D,*}^{\delta}(\u, \v) \\
    &:= \Big\{ \Gamma \in \cP_D(\u,\v)  : |\Cpts(\Gamma)| \geq 2 \,,\, \max_{i\in \{L,1,\dots, |\Cpts(\Gamma)|-1,R\}} |\Gamma^{(i)}| < (\log N)^2 \,,\, \Cpts(\Gamma) \cap \fR(\delta) = \emptyset 
    \Big\} 
\end{align*}
where, as usual, we write $\cP_{D,*}^{\delta}(\x) := \cP_{D,*}^{\delta}(0,\x)$. Observe that the last two conditions in the definition, along with the shape of $\fcone$, ensure that 
\begin{align}
    \Gamma \cap ([N^{4\delta}, N-N^{4\delta}] \times [0,N^{\delta/2}]) = \emptyset \text{ for all } \Gamma \in \cP_{D,\delta}^*(\u,\v)\,,
    \label{eqn:entire-animal-repulsion}
\end{align}
where the intersection pertains to both contour and clusters of $\Gamma$. From \cref{lem:generalized-length-cpts,prop:bounded-irreducible-pieces,prop:animal-entropic-repulsion}, we know that 
\begin{align}
    \bP_D^{\x}( \cP_{D,*}^{\delta}(\x))  = 1 +o(1)\,.
    \label{eqn:probability-good-event}
\end{align}
Define the measure
\[
    \bP_{D,*}^{\x}(\cdot) := \bP_D^{\x}(\cdot \given \cP_{D,*}^{\delta}(\x))\,.
\]
For $\Gamma \in \cP_{D,*}^{\delta}$, let $\zeta^{(\Lstar)}$ and $\zeta^{(\Rstar)}$  denote the left-most and right-most cone-points of $\Gamma$ in $\cS_{N^{4\delta}, N-N^{4\delta}}$, respectively.
Note that, since $\Gamma \in \cP_{D,*}^{\delta}$, we have 
\begin{align}
    \zeta^{(\Lstar)} \in (N^{4\delta}, N^{4\delta}+(\log N)^2]\times [N^{\delta}, N^{4\delta}(\log N)^2]
    \label{eqn:zeta-Lstar-condition}
\end{align}
and
\begin{align}
    \zeta^{(\Rstar)} \in [N-N^{4\delta}-(\log N)^2, N-N^{4\delta}] \times  (N^{\delta}, N^{4\delta}(\log N)^2] \,.
    \label{eqn:zeta-Rstar-condition}
\end{align}
Let $\Gamma^{(\Lstar)} \in \cP_{D,*}^{\delta}(0,\zeta^{(\Lstar)})$  denote the portion of $\Gamma$ connecting $0$ to $\zeta^{(\Lstar)}$ (including all clusters connected to $\gamma^{(\Lstar)}$). Similarly, define $\Gamma^{(\Rstar)}\in\cP_{D,*}^{\delta}(\zeta^{(\Rstar)},\x)$ as the portion of $\Gamma$ connecting 
 $\zeta_{\Rstar}$ to $\x$, and define $\Gamma^* \in \cP_{D,*}^{\delta}(\zeta^{(\Lstar)},\zeta^{(\Rstar)})$ as the portion of $\Gamma$ connecting $\zeta_{\Lstar}$ to $\zeta_{\Rstar}$ so that 
 \[
    \Gamma = \Gamma^{(\Lstar)} \circ \Gamma^* \circ \Gamma^{(\Rstar)}\,.
 \]
Note that $\Gamma^*$ is a concatenation of irreducible animals,
each of which are completely contained in $Q$ by~\eqref{eqn:entire-animal-repulsion}. Therefore, from~\eqref{eqn:q-factorization}, we have
\[
    q_D(\Gamma) = q_D(\Gamma^{(\Lstar)}) q(\Gamma^*) q_D(\Gamma^{(\Rstar)})\,.
\]
Crucially, the $q_D$-weight of $\Gamma^*$ is equal to its (free) $q$-weight.

For any fixed $\overline{\zeta}^{(\Lstar)}$ and $\overline{\zeta}^{(\Rstar)}$ satisfying~\eqref{eqn:zeta-Lstar-condition} and~\eqref{eqn:zeta-Rstar-condition} respectively, and for any fixed
animals $\overline{\Gamma}^{(\Lstar)} \in \cP_D(0, \overline{\zeta}^{(\Lstar)})$ and $\overline{\Gamma}^{(\Rstar)} \in \cP_D(\overline{\zeta}^{(\Rstar)}, \x)$, consider the  measure 
\begin{align} \label{def:P*}
    \bP_{D,*}^{\overline{\zeta}^{(\Lstar)},\overline{\zeta}^{(\Rstar)}}(\cdot) := 
    \bP_{D,*}^{\x} \big(\cdot \given \Gamma^{(\Lstar)} = \overline{\Gamma}^{(\Lstar)}, \Gamma^{(\Rstar)} = \overline{\Gamma}^{(\Rstar)} \big) 
    \,,
\end{align}
which describes the conditional law of $\Gamma^*$ and defines  a probability measure on the following set of animals: 
\begin{align*}
\cP_{D,\delta}^{*,\A}\big(\overline{\zeta}^{(\Lstar)}, \overline{\zeta}^{(\Rstar)}\big) := \big\{ \overline\Gamma^{(1)} \!\circ \cdots \circ \overline\Gamma^{(n)} \!\in \cP_{D,\delta}^*\big(\overline{\zeta}^{(\Lstar)}, \overline{\zeta}^{(\Rstar)}\big) \, : \, n\in \N \,,\,
\overline\Gamma^{(i)} \!\!\in \A \text{ for each } i \in [1,n]
\big\}.
\end{align*}
Note that, for fixed $\overline{\zeta}^{(\Lstar)}$ and  $\overline{\zeta}^{(\Rstar)}$, the choice of $\bar{\Gamma}^{(\Lstar)}$ and $\bar{\Gamma}^{(\Rstar)}$ does not change $\bP_{D,*}^{\overline{\zeta}^{(\Lstar)},\overline{\zeta}^{(\Rstar)}}(\cdot)$. In particular, we have the formula
\begin{align}
    \bP_{D,*}^{\overline{\zeta}^{(\Lstar)},\overline{\zeta}^{(\Rstar)}}(E) = \sum_{\Gamma \in E \cap \cP_{D,\delta}^{*,A}(\overline{\zeta}^{(\Lstar)},\overline{\zeta}^{(\Rstar)}) } q(\Gamma) \ \ / \ \sum_{\Gamma \in \cP_{D,\delta}^{*,A}(\overline{\zeta}^{(\Lstar)},\overline{\zeta}^{(\Rstar)})} q(\Gamma) \,.
    \label{eqn:conditional-law}
\end{align}
Again, note that only the free weight appears in the above formula.

Lastly, for any $c>0$, define the half-space $\H_c := \{(x,y) \in \R^2 : y \geq c\}$.

\begin{proposition}\label{prop:rw-coupling}
Let $D=Q$ or $D=\H$. There exists $\nu_2>0$ such that for all $N\in \N$, for all $\beta>0$ sufficiently large, and for all $\overline{\zeta}^{(\Lstar)}$ and $\overline{\zeta}^{(\Rstar)}$ satisfying~\eqref{eqn:zeta-Lstar-condition} and~\eqref{eqn:zeta-Rstar-condition} respectively, if we let $T= H_{\overline{\zeta}^{(\Rstar)}}$ and view  $\Cpts(\Gamma^*)$  as an ordered tuple  (see \cref{eqn:cone-points-vector}), 
\[ \Big\| \bP_{D,*}^{\overline{\zeta}^{(\Lstar)},\overline{\zeta}^{(\Rstar)}}\Big(\Cpts(\Gamma^*) \in \cdot\Big) - 
\P_{\overline{\zeta}^{(\Lstar)}}\Big( 
\big(\sS(i)\big)_{i= 0}^{T}
\in \cdot \given T < H_{\H_{N^{\delta}}} \Big)
\Big\|_{\tv}\leq Ce^{-\nu_2 \beta (\log N)^2}
\]
for some constant $C:=C(\beta)>0$.

\begin{proof}
Fix $\overline{\zeta}^{(\Lstar)}$ and $\overline{\zeta}^{(\Rstar)}$ satisfying~\eqref{eqn:zeta-Lstar-condition} and~\eqref{eqn:zeta-Rstar-condition} respectively.
Recall $\V_{\u,\v}^+$ from~\eqref{def:V+,A+}, and 
define
\begin{align*}
    \V_{\overline\zeta^{(\Lstar)}, \overline\zeta^{(\Rstar)}}^{\delta} :=& \Big \{ (\vec v_1, \dots, \vec v_n) \in \V_{\overline\zeta^{(\Lstar)}, \overline\zeta^{(\Rstar)}}^+ :  n \geq 1,
    \ \norm{\vec v_k}_1 \leq (\log N)^2  \\
    &\qquad \text{ and } 
    \overline{\zeta}^{(\Lstar)} + \sum_{i=1}^k \vec v_i \geq N^{\delta} ~\forall k \in [1, n] \Big \}\,,  \\
    \A_{\overline\zeta^{(\Lstar)}, \overline\zeta^{(\Rstar)}}^{\delta} :=& \Big \{
    (\overline \Gamma^{(1)}, \dots, \overline \Gamma^{(n)}) \in \A^n : \ n \geq 1, \ \big( X(\overline \Gamma^{(1)}), \dots, X(\overline \Gamma^{(n)}) \big)  \in \V_{\overline\zeta^{(\Lstar)}, \overline\zeta^{(\Rstar)}}^{\delta} 
    \Big \}\,.
\end{align*}
Note that the sets $\A_{\overline\zeta^{(\Lstar)}, \overline\zeta^{(\Rstar)}}^{\delta}$ and $\cP_{D,\delta}^{*,\A}(\overline\zeta^{(\Lstar)}, \overline\zeta^{(\Rstar)})$ are in bijection with one another, since for any 
$(\overline \Gamma^{(1)}, \dots, \overline \Gamma^{(n)}) \in \A_{\overline\zeta^{(\Lstar)}, \overline\zeta^{(\Rstar)}}^{\delta}$, the animal $\overline\Gamma^{(1)} \circ \cdots \circ \overline\Gamma^{(n)}$ is in $\cP_{D,\delta}^{*,\A}(\overline\zeta^{(\Lstar)}, \overline\zeta^{(\Rstar)})$.
The above observations, along with~\eqref{eqn:decomposed-weight} (where we substitute $\y := \x$, and use that $\P^{\h_\x}(\Gamma^{(i)})=e^{\h_{\x} \cdot \X(\Gamma^{(i)})}q(\Gamma^{(i)})$) 
yields the following:
\begin{align}
    \sum_{\Gamma \in \cP_{D,\delta}^{*,A}(\overline{\zeta}^{(\Lstar)},\overline{\zeta}^{(\Rstar)})} 
    &q(\Gamma) 
    = e^{\h_{\x}\cdot \big(\overline{\zeta}^{(\Lstar)}+\x-\overline{\zeta}^{(\Rstar)}\big)} 
    \ \sum_{n\geq 1} \sum_{(\overline\Gamma^{(1)}, \dots, \overline \Gamma^{(n)}) \in \A_{\overline{\zeta}^{(\Lstar)}, \overline{\zeta}^{(\Rstar)}}^{\delta}} \prod_{i=1}^n \Pa(\overline \Gamma^{(i)}) \nonumber \\
    &= e^{\h_{\x}\cdot \big(\overline{\zeta}^{(\Lstar)}+\x-\overline{\zeta}^{(\Rstar)}\big)} 
    \ \sum_{n\geq 1} \sum_{(\vec v_1, \dots, \vec v_n) \in \V_{\overline{\zeta}^{(\Lstar)}, \overline{\zeta}^{(\Rstar)}}^{\delta}}
    \prod_{i=1}^n \Big(\sum_{\overline \Gamma^{(i)} \in \A} \Pa(\overline \Gamma^{(i)}) \ind{\overline \X(\overline\Gamma^{(i)}) = \vec v_i} \Big) \nonumber \\
    &= e^{\h_{\x}\cdot \big(\overline{\zeta}^{(\Lstar)}+\x-\overline{\zeta}^{(\Rstar)}\big)} 
    \ \sum_{n\geq 1} \sum_{(\vec v_1, \dots, \vec v_n) \in \V_{\overline{\zeta}^{(\Lstar)}, \overline{\zeta}^{(\Rstar)}}^{\delta}} \prod_{i=1}^n \P(\sS(1) = \vec v_i) \nonumber \\
    &= e^{\h_{\x}\cdot \big(\overline{\zeta}^{(\Lstar)}+\x-\overline{\zeta}^{(\Rstar)}\big)}  \
    \P_{\overline{\zeta}^{(\Lstar)}}\big( E_{\delta}^{*} \big)\,,
    \label{eqn:great-partition-function}
\end{align}
where we define the random walk event 
\[
E_{\delta}^{*} := \big \{ H_{\overline{\zeta}^{(\Rstar)}} < H_{\H_{N^{\delta}}} \,,\,  \|\sS(i+1) - \sS(i) \|_1 \leq (\log N)^2, \ \forall i \in [0, H_{\overline{\zeta}^{(\Rstar)}}-1]  \big \} \,.
\]
A nearly identical calculation but for the numerator of the right-hand side of~\eqref{eqn:conditional-law}  yields the following: for any $n\geq 1$ and for any $(\vv_1, \dots, \vv_n) \in (\Z^2)^n$,
\begin{align*}
    &\bP_{D,*}^{\overline{\zeta}^{(\Lstar)},\overline{\zeta}^{(\Rstar)}} \Big( \Cpts(\bar{\Gamma}) = \Big(\overline{\zeta}^{(\Lstar)}, \ \overline{\zeta}^{(\Lstar)} + v_1, \dots, \ \overline{\zeta}^{(\Lstar)} + \sum_{i=1}^{n-1} \vec{v}_i , \overline{\zeta}^{(\Rstar)} \Big) \Big) \\
    &=\P_{\overline{\zeta}^{(\Lstar)}} \Big( H_{\overline{\zeta}^{(\Rstar)}} = n+1 \,,\, \sS(k) = \overline{\zeta}^{(\Lstar)} + \sum_{i=1}^k \vv_i \,,\, \forall k \in [1, n] \given E_{\delta}^{*} \Big) 
        \ind{(\vv_1, \dots, \vv_n) \in \V_{\overline{\zeta}^{(\Lstar)}, \overline{\zeta}^{(\Rstar)}}^{\delta}}\,.
\end{align*}
Thus, the law of $\Cpts(\overline \Gamma)$ under $\bP_{D,*}^{\overline{\zeta}^{(\Lstar)},\overline{\zeta}^{(\Rstar)}}$ is \emph{equal to} the law of $(\sS(i))_{i=0}^{H_{\overline{\zeta}^{(\Rstar)}}}$ under $\P_{\overline{\zeta}^{(\Lstar)}}(\cdot \given E_{\delta}^{*})$.
Now, observe that \cref{thm:our-rw-inputs} along with shift-invariance of the random walk (to shift the walk downwards by $N^{\delta}$) give the following estimate, which holds uniformly over our ranges of $\overline{\zeta}^{(\Lstar)}$ and~$\overline{\zeta}^{(\Rstar)}$:
\[
    \P_{\overline{\zeta}^{(\Lstar)}}\big( H_{\overline{\zeta}^{(\Rstar)}} < H_{\H_{N^{\delta}}}\big) = \P_{\overline{\zeta}^{(\Lstar)} - (0,N^{\delta})}\big( H_{\overline{\zeta}^{(\Rstar)}-(0,N^{\delta})} < H_{\H_{-}}\big) \gtrsim N^{-3/2}\,.
\]
Exponential tails of the random walk increments~\eqref{eqn:rw-exp-tail} and the above estimate then imply that
\begin{align}\label{eqn:random-walk-increments-bound}
    \P_{\overline{\zeta}^{(\Lstar)}}\big(E_{\delta}^* \given H_{\overline{\zeta}^{(\Rstar)}} < H_{\H_{N^{\delta}}}\big)
    \geq 1- Ce^{-\nu \beta (\log N)^2}
\end{align}
for some constant $\nu >0$ independent of $\beta$, and a constant $C:=C(\beta)>0$, finishing the proof.
\end{proof}
\end{proposition}

\subsection{Proof of \cref{thm:BM-for-Ising-polymers-restated}}
\label{subsec:pf-BM-for-Ising-polymers}

We see from the shape  of the cone $\fcone$ and the restriction on $|\Gamma^{(i)}|$ that for all $\Gamma \in \cP_{D,\delta}^*(\x)$, 
\[
    \frac{1}{ \sqrt N}\max_{x\in [0,N]} |\overline\gamma(x) - \underline\gamma(x)| \leq \frac{2(\log N)^2}{ \sqrt N}\,.
\]
In addition to~\eqref{eqn:probability-good-event}, this tells us that it suffices to show $\mathfrak{J}_N(\Cpts(\Gamma))$ under $\bP_{D,*}^{\x}(\cdot)$ converges weakly to a Brownian excursion. Now,  we have the following on $\cP_{D,\delta}^*$:
\begin{align}
    \sup\{ y : \exists x \in \Z \text{ s.t. } (x,y) \in \gamma^{(\Lstar)} \} \,,\, 
    \sup\{ y : \exists x \in \Z \text{ s.t. } (x,y) \in \gamma^{(\Rstar)} \}
    \leq N^{4\delta} (\log N)^2 
\end{align}
as well as the bounds in \cref{eqn:zeta-Lstar-condition,eqn:zeta-Rstar-condition}. These bounds imply that $\Gamma^{(\Lstar)}$ and $\Gamma^{(\Rstar)}$ do not impact the diffusive scaling limit of $\Cpts(\Gamma)$. Thus, recalling the linear interpolation $\mathfrak{J}_N$ from~\eqref{def:J},  it suffices to show 
$
\mathfrak{J}_N(\Cpts(\Gamma^*)) 
$
under law 
$\bP_{D,*}^{\overline{\zeta}^{(\Lstar)},\overline{\zeta}^{(\Rstar)}}(\cdot)$
converges weakly to a Brownian excursion,
uniformly in $\overline{\zeta}^{(\Lstar)}$ and $\overline{\zeta}^{(\Rstar)}$ satisfying~\eqref{eqn:zeta-Lstar-condition} and~\eqref{eqn:zeta-Rstar-condition}. 
But this is an immediate consequence of \cref{prop:rw-coupling}, the shift invariance of the random walk, and \cref{thm:rw-invariance}.

This concludes the proof of \cref{thm:BM-for-Ising-polymers-restated}, and hence also \cref{mainthm:BM-for-Ising-polymers}.
\qed

\section{Random Walks in a half-space}\label{sec:rw-half-space}
This section is devoted to the proofs of  \cref{thm:our-rw-inputs,thm:rw-invariance} via the analysis of random walks in a half-space.

As mentioned after \cref{thm:rw-invariance}, we seek to extend many of the results from \cite{DenisovWachtel15} and \cite{DurajWachtel20}. These papers are written for random walks whose coordinates are uncorrelated. Though $\sS(\cdot)$ may not have this property, we can obtain such a random walk by rotating--- this is done in \cref{subsec:pf-our-rw-inputs}. 

As such, we must consider random walks in general half-spaces. Throughout the rest of this section, fix a unit vector $\n \in \bbS^1$, and consider the half-space through the origin with inward normal $\n$, and call this space $\H_{\n} \subset \R^2$.
Let $\np \in \bbS^1$ be a unit vector orthogonal to $\n$ (the choice between $\n$ and $-\n$ does not matter). We will think of our random walk in terms of coordinates with respect to $\n$ and $\np$ instead of $\e_1$ and $\e_2$.
Let $S(\cdot)$ denote a general $2$D random walk  with step distribution $X = (X_1, X_2)$ such that $\E X =\vec{0}$, $\Cov X = \mathrm{Id}$, and $\P(|X|>m) \leq c_1e^{-c_2m}$  for some constants $c_1, c_2>0$ and for all $m \geq 1$. Additionally, we impose a lattice assumption: assume that $X$ takes values on a lattice $\mathcal{L}$ that is a non-degenerate linear transformation of $\Z^2$, and that the distribution of $X$ is strongly aperiodic; that is, for each $\u \in \mathcal{L}$, $\mathcal{L}$ is the smallest subgroup of $\Z^2$ containing 
\[
\{ \v : \v = \u + \mathsf{w} \text{ for some $\mathsf{w}$ such that $\P(X = \mathsf{w}) > 0$ } \}\,.
\]

The theory of random walks confined to a half-space is intimately related to the theory of harmonic functions for processes killed upon exiting a half-space, and so we recall the relevant facts from this theory  before proceeding.

\subsection{Harmonic functions for processes in a half-space} \label{subsubsec:harmonic-functions}
 Our domain of interest is a relatively simple one, the half-space $\H_{\n}$, and we are concerned with the harmonic functions of the Brownian motion and of the random walk $S(\cdot)$ killed upon exit from $\H_{\n}$.

The harmonic function of the Brownian motion killed at $\partial \H_{\n}$ is given by the minimal (up to a constant), strictly positive harmonic function on $\H_{\n}$ with zero boundary conditions. For us, such functions take a very explicit form. 
Consider the rotation that sends $\n \mapsto \e_2$, and thus $\H_{\n}$ to $\H$. This is a conformal mapping that induces a bijection between positive harmonic functions in $\H$ that vanish on $\partial \H$ to positive harmonic functions in $\H_{\n}$ that vanish on $\partial \H_{\n}$. Since\footnote{See, e.g., \cite[Theorem~7.22]{SheldonBourdonWade01}.} every such harmonic function in $\H$ takes the form $h(\v) = c \v\cdot \e_2$, for some constant $c>0$, it follows that every harmonic $h$ function in $\H_{\n}$ takes the form $h(\v) = c \v\cdot \n$. In particular, the harmonic function only depends on the projection of $\v$ onto $\n$. In what follows, we take $c=1$ and define 
\[
u(\v) := \v \cdot \n\,.
\]
 We refer to~\cite{Dante87} for a general discussion of the harmonic function of the Brownian motion in cones. 

One of the many achievements of~\cite{DenisovWachtel15} is the construction of a positive harmonic function $V$ for a wide class of random walks $S(\cdot)$ killed upon exit from $K$, i.e., $V$ solves the equation\footnote{In~\cite{DenisovWachtel15} and~\cite{DurajWachtel20}, the authors consider random walks $S(\cdot)$ from $S_0 = 0$, and study the law of $(\v+S(\cdot))$ for $\v \in K$. In particular, they write $\tau_\v$ to the hitting time $H_{K^c}^{S(\cdot)}$ under law $\P_\v$.}
\[
\E_\v[V(S(1)), H_{K^c} >1] = V(\v) \,,\, \text{ for } \v\in K\,,
\]
where $K$ is an element of a wide class of cones.\footnote{See also~\cite{DenisovWachtel19}, where such a harmonic function was constructed for a more general class of cones; as well as~\cite{DenisovWachtel21}, which addressed a more general class of random walks.} Many of their limit theorems that we wish to extend were stated in terms of $V$. The function $V$ was constructed as
\[
V(\v) = \lim_{n \to \infty} \E_\v[ h(S(n))\,,\, H_{K^c} > n ]\,,
\]
where $h$ is a choice of the harmonic function for Brownian motion killed at $\partial K$. For $K := H_{\H_{\n}}$,  $h := u$, and so we see
\begin{align}
V(\v) = V_1(\v\cdot \n)\,, \text{ for all } \v \in \H_{\n}\,,
\label{eqn:harmonic-function-asymptotics}
\end{align}
where $V_1 := V_1^{\n}$ is the unique positive harmonic function for the one-dimensional random walk $\n \cdot S(\cdot)$ killed at leaving $(0,\infty)$ satisfying 
\[
\lim_{a \to \infty} \frac{V_1(a)}{a}  = 1 \,.
\]
The uniqueness of $V_1$, as well as an exact formula for $V_1$, was established by Doney in~\cite{Doney98}.

\subsection{Limit theorems for general random walks in a half-space}
The following two results, \cref{prop:rwcones-ballot,prop:2nd-component-genrw-limit}, are modifications of a ballot-type theorem from~\cite{DenisovWachtel15} and an invariance principle from~\cite{DurajWachtel20}, respectively. We'll write $\mathfrak{e}(\cdot)$ to denote the standard Brownian excursion on~$[0,1]$.

\begin{proposition}[Modification of {\cite[Theorem~6]{DenisovWachtel15}}]
\label{prop:rwcones-ballot}
Fix any $A > 0$ and $\delta \in (0,1/2)$.
Then there exists a constant $C_1>0$ such that, uniformly over sequences $a_k, b_k$ and $u_k$ satisfying
\begin{align}
    a_k \in [-A\sqrt{k}, A\sqrt{k}] \ , \  \frac{a_k}{\sqrt{k}} \to a \in [-A,A] \ , \ b_k, u_k \in (0, k^{1/2-\delta}]\ , \  \text{ and} \ \ \{ u_k \n \,,\, a_k \np+b_k \} \subset \mathcal{L} \label{eqn:uniform-parameters}
\end{align}
we have
    \begin{align}
        \P_{u_k \n}\Big(S(k) = a_k \np + b_k \n \,,\, H_{\H_{-\n}}^S > k \Big) \sim C_1\kappa^2 \frac{V_1(u_k) V_1'(b_k)}{k^2} e^{-a^2/2 }\,, \label{eqn:rwcones-ballot}
    \end{align}
    where $V_1'$ is the positive harmonic function for $-\n\cdot S_2(\cdot)$ killed upon leaving $(0,\infty)$.
\end{proposition}

\begin{proposition}[Modification of {\cite[Theorem~6]{DurajWachtel20}}]\label{prop:2nd-component-genrw-limit}
Fix any $\delta \in (0,1/2)$.
Uniformly over sequences 
\begin{align}
    a_k \in [-A\sqrt{k}, A\sqrt{k}] \quad \text{ and } \quad b_k, u_k \in (0, k^{1/2-\delta}] \quad \text{ and } \quad \{ u_k \n \,,\, a_k \np+b_k\} \subset \mathcal{L}\,, \label{eqn:uniform-parameters-noconverge}
\end{align}
the family of conditional laws
\begin{align*}
    \bQ_{u_k, a_k, b_k}^k (\cdot) := \P_{u_k \n}\bigg( \Big(\frac{\n \cdot S(\floor{tk})}{\sqrt{k}}\Big)_{t\in [0,1]} \in \cdot \given S(k) = a_k \np + b_k\n \,,\, H_{\H_{-\n}}^{S} > k \bigg)
\end{align*}
converge  as $k \to \infty$  to the law of $\mathfrak{e}(\cdot)$ in the Skorokhod space $(D[0,1], \|\cdot\|_{\infty})$.
\end{proposition}

For the next result, let $\bS(\cdot)$ denote a two-dimensional random walk with step distribution $\bX = (\bX_1, \bX_2)$ satisfying the same lattice, covariance, and tail decay assumptions as $X$, but with mean  $\E \bX =\mu \np$ for some $\mu >0$. 

\begin{proposition}\label{prop:genrw-ballot}
Fix any $\delta \in (0,1/2)$. There exists a constant $\mathbf{C'}:= \mathbf{C'}(\mathbf{X})>0$ such that, uniformly over $u, v \in (0,N^{1/2-\delta}]$ and $\{u\n\,,\,N\np+v \n\} \subset \mathcal{L}$, we have 
\begin{align}
     \P_{u \n}\Big(H^{\bS}_{N \np + v \n} < H_{\H_{-\n}}^{\bS} \Big)  \sim \mathbf{C'} \frac{V_1(u) V_1'(v)}{N^{3/2}}\,,
     \label{eqn:genrw-ballot-leadingorder}
\end{align}
where $V_1'$ denotes the harmonic function for $-\n\cdot\bS_2(\cdot)$ killed upon leaving $(0,\infty)$.

Furthermore, 
if
\[ \mathbf{p}_{N,A} := \P_{u \n}\Big(H^{\bS}_{N \np + v \n}  \in [N/\mu - A\sqrt{N}, N/\mu+ A\sqrt{N}] \given H^{\bS}_{N \np + v \n} < H_{\H_{-\n}}^{\bS} \Big)\]
then 
\begin{align}
    \lim_{A\to\infty}\liminf_{N\to\infty} \mathbf{p}_{N,A} = 1\,.
    \label{eqn:genrw-ballot-goodk}
\end{align}
\end{proposition}

\Cref{prop:rwcones-ballot,prop:2nd-component-genrw-limit,prop:genrw-ballot} will be proved in \cref{subsec:pfs-limit-thms}. Note that when $\n = \e_2$ (so $\H_{\n} = \H$),  \Cref{prop:genrw-ballot} reduces to \cite[Theorem~5.1]{IOVW20}.

\subsection{Proof of \cref{thm:our-rw-inputs,thm:rw-invariance}}
\label{subsec:pf-our-rw-inputs}
Recall that the random walk $\sS(\cdot)$ is a two-dimensional random walk whose step distribution $X= ( X_1,   X_2)$ satisfies $\E[ X] = (\mu,0)$, for some $\mu >0$, and exponential tail decay for $| X|$. Recall $\sigma_1^2 := \Var(X_1)$ and $\sigma_2^2 := \Var(X_2)$. 

\begin{proof}[Proof of \cref{thm:our-rw-inputs}]
\Cref{thm:our-rw-inputs} follows from \cref{prop:rwcones-ballot,prop:2nd-component-genrw-limit,prop:genrw-ballot} by linearly transforming the random walk $\sS(\cdot)$ to satisfy the hypotheses of these results.
We demonstrate this explicitly for item (2) of \cref{thm:our-rw-inputs}; the proof of item (1) follows readily.

Following from \cite[Example~2]{DenisovWachtel15}, we  derive the aforementioned linear transformation, . 
Consider the re-centered, normalized random walk 
\[
    \overline \sS(n) := 
    \begin{pmatrix}
        \sigma_1^{-1} & 0 \\ 0 & \sigma_2^{-1}
    \end{pmatrix} 
    \bigg( \sS(n) - n\begin{pmatrix} \mu \\ 0 \end{pmatrix} \bigg)\,.
\]
This is a random walk with increments $\bar{X}:= (\bar{X}_1, \bar{X}_2)$ of mean $0$ such that, for some $\rho \in (0,1)$, we have $C:= \Cov \bar{X} =  \E \bar{X}\bar{X}^t = \begin{pmatrix} 1&\rho \\ \rho &1 \end{pmatrix}$.
Let $C = ODO^t$. Since $C$ is a covariance matrix, it is positive semi-definite; however, if $C$ had $0$ as an eigenvalue, then that would imply the random walk lives on a line, which we know is not true as $X$ can be $\e_1$ or $\e_1+\e_2$ with positive probability. So, $C$ is positive-definite. Consider the transformation matrix $M := O^t\sqrt{D^{-1}} O$. Then 
\[
\Cov(MX) = \E[M\bar{X}\bar{X}^tM^t] = \mathrm{Id} \,.
\]
To be explicit, straight-forward calculations reveal that if $\theta \in [-\pi,\pi]$ solves $\sin 2\theta = \rho$, then we may take
\[
M = \frac{1}{\sqrt{1- \rho^2}} \begin{pmatrix} \cos \theta & - \sin \theta \\ -\sin \theta & \cos \theta \end{pmatrix}\,,
\]
though we will not need this explicit form.

Now, observe that $\n := M^{-1}\e_2$ has norm $1$ (explicitly, $\n = (\sin \theta, \cos \theta)$). Using the fact that $M^{-1}$ is symmetric and $\n$ is outward normal to $M\H_{-} = \H_{-\n}$, 
\begin{multline*}
    \big\{ H_{\H_-}^{\sS}> k \big\} = \big\{\e_2^{t} \sS(i) >0 \,,\, \forall i \in [0,k] \big\}
= \big\{\e_2^{t} \overline\sS(i) >0 \,,\, \forall i \in [0,k] \big\} \\
= \big\{ M \overline{\sS}(i) \cdot M^{-1}\e_2 >0 \,,\, \forall i \in [0,k]\big\} 
    = \big\{ H_{\H_{-\n}}^{M\overline\sS} >k \big\}\,,
\end{multline*}
In what follows, we write $M(a,b)$ to denote the matrix applied to the vector $a \e_1 + b\e_2$. 
Then for any measurable set $\mathrm{B}$,
\begin{multline}
    \P_{(0,u)}\Big(\sS(\cdot) \in \mathrm{B} \given  S(k) = (N,v) \,,\, H_{\H_-}^{\sS(\cdot)} > k \Big)  \\ 
    = \P_{(0,\sigma_2^{-1} u)}\Big(\overline{\sS}(\cdot) \in \mathrm{B}-k(\mu,0) \given  \overline{\sS}(k) = (\tfrac{N-k\mu}{\sigma_1},\tfrac{v}{\sigma_2}) \,,\, H_{\H_-}^{\overline\sS} > k \Big)  \\
    = \P_{M(0,u)}\Big(M\overline{\sS}(\cdot) \in M(\mathrm{B}-k(\mu,0)) \given M\overline{\sS}(k)  = M(\tfrac{N-k\mu}{\sigma_1},\tfrac{v}{\sigma_2}) \,,\, H_{\H_{-\n}}^{M\overline\sS} >k \Big)\,.
    \label{eqn:transformed-rw-law}
\end{multline}
Let us now check the hypotheses of \cref{prop:2nd-component-genrw-limit}.
Take $\np := M\e_1/\norm{M\e_1}$. Note that $\np$ spans $\partial \H_{-\n}$ and $\n$ is the inward normal of $M\H=\H_{\n}$.
$M\sS(\cdot)$ is a random walk on the lattice $\mathcal{L}$ generated by $\norm{M\e_1} \np$ and $M\e_2$, the latter of which can be expressed as
\begin{align*}
    M\e_2 = (M\e_2\cdot \np) \np + \n = \frac{ \e_1^t M^tM\e_2}{\norm{M\e_1}} \np + \n = \frac{\e_1^tC^{-1}\e_2}{\norm{M\e_1}}  \np + \n\,. 
\end{align*}
Thus, for any $x, y \in \R$, 
\[
M(x,y) = \frac{x + y\e_2^t C^{-1} \e_1}{\norm{M\e_1}} \np + y\n \,.
\]
For any $k \in [N/\mu -A\sqrt{N}, N/\mu +A\sqrt{N}]$, for $A$ arbitrarily large (but fixed with respect to $N$), and for any $u,v  \in [1, N^{1/2-\delta}]$, we see that
\[
    u_k \n:= M(0,u)  \quad  \text{ and } \quad a_k \np+ b_k\n := M(\tfrac{N-k\mu}{\sigma_1},\tfrac{v}{\sigma_2})
\]
satisfies \cref{eqn:uniform-parameters}. With the hypotheses of \cref{prop:2nd-component-genrw-limit} satisfied, it follows from \cref{eqn:transformed-rw-law} and
\[
\sigma_2^{-1}\sS_2 (\cdot)  =  \overline{\sS}(\cdot) \cdot \e_2 = M\overline{\sS}(\cdot) \cdot M^{-1}\e_2  = M\overline{\sS}(\cdot)\cdot \vec{n} 
\]
that item (2) of \cref{thm:our-rw-inputs} is an immediate consequence of \cref{prop:2nd-component-genrw-limit}.
\end{proof}

We now use \cref{thm:our-rw-inputs} to prove the Brownian excursion  limit. 

\begin{proof}[Proof of \cref{thm:rw-invariance}]
Due to \cref{eqn:our-rw-goodk}, it suffices to show that, for each fixed $k$, uniformly in the number of steps $k \in [N/\mu - A\sqrt{N}, N\mu + A\sqrt{N}]$, the  family of conditional laws
\begin{align*}
    \mathbf{Q}_{u,v}^{N,k} (\cdot) &:= \P_{(0,u)}\Big(\big(\mathfrak{e}^{\sS,v}(t)\big)_{t\in[0,1]} \in \cdot \given  \sS(k) = (N,v) \,,\, k < H_{\H_{-}}^{\sS} \Big)
\end{align*}
 converges  as $k \to \infty$  to the law of the standard Brownian excursion on $[0,1]$. We begin with \cref{claim:rw-x-size}, which localizes the $x$-coordinate $\sS_1(j)$  to an interval of size $o(N)$, for $j\in[0,k]$, and  allows us to compare the linear interpolation $\mathfrak{e}^{\sS,v}$ with the linear interpolation in item (2) of \cref{thm:our-rw-inputs}.

\begin{claim}\label{claim:rw-x-size}
For any $\eta >0$,  uniformly over $u$ and $v$ as in the theorem statement, we have
\begin{align}
    \lim_{N \to\infty}\P_{(0,u)} \Big( \max_{j \in [0,k]} \big|\sS_1(j) - \mu j\big| > N^{1/2+\eta} \given \sS(k) = (N,v) \,,\, k < H_{\H_{-}}^{\sS} \Big) = 0\,.
\end{align}
\begin{proof}
In light of the bound on $\P_{(0,u)}(\sS(k) = (N,v) \,,\, k < H_{\H_{-}}^{\sS})$ provided by \cref{eqn:our-rw-ballot}, it suffices to show
\begin{align}
    \P_0 \Big( \max_{j \in [0,N]} \big|\sS_1(j) - \mu j\big| > N^{1/2+\eta} \Big) = o(N^{-3/2})\,.
    \label{eqn:claim-rw-x-size-crux}
\end{align}
This follows from a union bound over $j$, the exponential tail bound on the steps of $\sS_1(\cdot)$ from \cref{eqn:rw-exp-tail}, and Hoeffding's inequality:
\begin{align*}
    &\P_{0} \Big( \max_{j \in [0,N]} \big|\sS_1(j) - \mu j\big| > N^{1/2+\eta} \Big) 
    \leq N \max_{j \in [0,N]} \P_0\Big(\big|\sS_1(j)- \mu j\big| > N^{1/2+\eta} \Big) \\
    &\leq N \max_{j \in [0,N]} \P_0\Big(\big|\sS_1(j)- \mu j\big| > N^{1/2+\eta} \,,\, \max_{i \in [0,j-1]} \big|\sS(i) - \sS(i+1) \big| \leq (\log N)^2 \Big) \\
    &\qquad+ c' N^2 e^{-\nu_g \beta (\log N)^2} \\
    &\leq 2N e^{-\frac{2N^{2\eta}}{(\log N)^4}}+ c' N^2 e^{-\nu_g \beta (\log N)^2} = o(N^{-3/2})\,. \qedhere
\end{align*}
\end{proof}
\end{claim}
Note that \cref{thm:our-rw-inputs}(2) and \cref{claim:rw-x-size} are the equivalents of \cite[Eqs.~(76) and~(77)]{IOVW20} for our random walk. We therefore finish the proof of our \cref{thm:rw-invariance} exactly as in the proof of  \cite[Theorem~5.3]{IOVW20}.
\end{proof}

\subsection{Uniform estimates for the random walk in a half-space} \label{subsec:proof-rw-be-inputs}
Before proving \cref{prop:rwcones-ballot,prop:2nd-component-genrw-limit,prop:genrw-ballot}, 
we need to modify several key estimates of~\cite{DenisovWachtel15} to address the range of parameters in \cref{eqn:uniform-parameters,eqn:uniform-parameters-noconverge}. In what follows, we will repeatedly refer back to~\cite{DenisovWachtel15}, explaining how their arguments can be adapted to give the desired uniformity. When stating or citing the results of~\cite{DenisovWachtel15} and~\cite{DurajWachtel20}, we will for the most part use their notation, pointing out any discrepancies explicitly.

\subsubsection{Needed inputs}
The results of~\cite{DenisovWachtel15} are often stated in terms of the dimension $d$ and a positive constant $p$, where $p$ is related to the asymptotic behavior of the relevant harmonic function in the cone. For us, $d=2$ and  $p=1$. 

Our first input is an extension of two estimates of \cite{DenisovWachtel15} uniformly over our range of initial positions. 
\begin{proposition}[Modification of {\cite[Eq.~(7)]{DenisovWachtel15}}] 
\label{prop:rwcones-tau}
Fix any $\delta \in (0,1/2)$. The following estimate holds
uniformly over all $a_k \in \R$ and $b_k \in (0 , k^{1/2-\delta}]$ such that $a_k \np +b_k\n \in \mathcal{L}$:
\begin{align}
    \P_{a_k \np + b_k\n}\big(H_{\H_{-\n}} > k \big) \sim \kappa V_1(b_k) k^{-1/2}\,.
    \label{eqn:hitting-time}
\end{align}
\begin{proof}
    \cref{eqn:hitting-time}, for $a_k := \mathfrak{a}$ and $b_k := \mathfrak{b}$ fixed, is proved in  \cite{DenisovWachtel15} as an immediate consequence of the lemmas of \cite[Sections~3 and 4]{DenisovWachtel15}. Of these, only \cite[Lemma~21]{DenisovWachtel15} is insufficient for the uniformity that we require.\footnote{It may appear that the constant $C(x)$ in \cite[Lemma~16]{DenisovWachtel15} also poses an issue for uniformity, but here $C(x)$ can be taken to be $C(\epsilon)$ for $\epsilon$ in that lemma, as a consequence of \cite[Lemma~14]{DenisovWachtel15} in our special case of $u(\v):= \v \cdot \n$.}
    That is, we must show that for all $\ep >0$ sufficiently small,
    \[
    \E_{a_k \np + b_k\n}\Big[u\big(S(\nu_k)\big); H_{\H_{-\n}}^{S} > \nu_k, \nu_k \leq k^{1-\ep} \Big] = V(a_k \np + b_k\n)(1+o(1))\,,
    \]
    where $\nu_k$ is defined as the first hitting time of $k^{1/2-\ep} \n + \H_{\n}$ and $o(1) \to 0$ as $k \to \infty$ uniformly over the ranges of interest for $a_k$ and $b_k$. Uniformity in $a_k$ is trivial. Uniformity of $b_k$ was proven in \cite[Sec.~5.6]{IOVW20} (see their Eq.~$(60)$, and recall that $u(S(\nu_k)) = S(\nu_k) \cdot \n$ and $V(a_k \np + b_k\n) = V_1(b_k)$). Thus, the proof of \cite[Eq.~(7)]{DenisovWachtel15} given at the end of \cite[Sec.~4]{DenisovWachtel15} extends to prove our proposition.
\end{proof}
\end{proposition}

\begin{proposition}[Modification of {\cite[Theorem~3]{DenisovWachtel15}}]\label{prop:thm3}
     Uniformly over sequences $a_k$ and $b_k$ satisfying \cref{eqn:uniform-parameters}, the family of measures
     \[
        \P_{a_k \np+  b_k \n} \Big( \frac{S(k)}{\sqrt k} \in \cdot \given H_{\H_{-\n}}^S >k \Big)
     \]
     converges weakly, as $k$ tends to $\infty$, to the probability measure on $\H_{\n}$ with density given by  $H_0 (\y\cdot \n) e^{-\|a\np-\y\|^2/2} \d \y$, where $H_0 > 0$ is the normalizing constant.
\begin{proof}
Let $B$ denote a measurable subset of $K$. Theorem~3 of \cite{DenisovWachtel15} gives us 
\[
    \lim_{k\to\infty} \P\Big( \frac{S(k)}{\sqrt k} \in B- a \np \given H_{\H_{-\n}}^S >k \Big) = H_0 \int_{B- a\np} \!\!\!(\y\cdot \n) e^{-|\y|^2/2} \d \y = H_0 \int_{B} (\y\cdot \n) e^{-|a\np -\y|^2/2} \d \y.
\]
A simple continuity argument then yields
\[
\lim_{k \to \infty} \P\Big( \frac{S(k)+a_k \np + b_k \n}{\sqrt k}  \in B \given H_{\H_{-\n}}^S >k \Big)= H_0 \int_{B} (\y\cdot \n) e^{-|a\np -\y|^2/2} \d \y
\]
as well. This concludes the proof.
\end{proof} 
\end{proposition}

\begin{proposition}[Modification of {\cite[Theorem~5]{DenisovWachtel15}}]
\label{prop:rwcones-thm5}
Recall $H_0$ from \cref{prop:thm3}. Then uniformly over sequences $a_k$ and $b_k$ satisfying \cref{eqn:uniform-parameters}, 
    \begin{align}
        \limsup_{k \to \infty} \sup_{\y \in \H_{\n}} \abs{\frac{k^{3/2}}{V_1(b_k)} \P_{a_k \np + b_k \n}\big(S(k) = \y \,,\, H_{\H_{-\n}}^S > k \big) - \kappa H_0  \frac{\y\cdot \n}{\sqrt{k}} e^{- \|a\sqrt{k}\np - \y\|^2/2k}} = 0\,.
    \label{eqn:rwcones-thm5}
    \end{align}
\begin{proof}[Proof of \cref{prop:rwcones-thm5}]
Below, we adapt the proof of Theorem~5 given in \cite[Section~6.2]{DenisovWachtel15}, beginning as in \cite{DenisovWachtel15} by splitting   $K := \H_{\n}$ into three parts:
\begin{align*}
    K^{(1)} &:= \{ \y \in \H_{\n} : \norm{\y} > R \sqrt k \} \, \\
    K^{(2)} &:= \{ \y \in \H_{\n} : \norm{\y} \leq R \sqrt{n}\,,\, \y \cdot \n \leq 2 \ep \sqrt{n} \} \, \\
    K^{(3)} &:= \{\y \in \H_{\n} : \norm{\y} \leq R \sqrt k \,,\, \y \cdot \n  > 2\ep\sqrt k \} \,,
\end{align*}
for some $R>0$ and $\ep >0$. Below, we let $C>0$ denote a constant independent of $k$, $\y$, $R$, and $\ep$  that may change from line to line. Since
\begin{align*}
    \lim_{R \to \infty} \lim_{\ep \to 0} \sup_{\y \in K^{(1)} \cup K^{(2)}} \frac{\y \cdot \n}{\sqrt{k}} e^{- \|a\sqrt{k}\np - \y\|^2/2k} = 0 \,,
\end{align*}
\cref{prop:rwcones-thm5} will be proved if we can show
\begin{align}
    \lim_{R \to \infty} \lim_{\ep \to 0} \limsup_{k \to\infty} \frac{k^{3/2}}{V_1(b_k)}\sup_{\y \in K^{(1)} \cup K^{(2)}} \P_{a_k \np + b_k \n}\big(S(k) = \y \,,\,H_{\H_{-\n}}^S>k\big) = 0 \label{eqn:rwcones-thm5-zeroregion}
\end{align}
and 
\begin{align}
    \lim_{\ep \to 0} \limsup_{k\to\infty} \!\! \sup_{\y \in K^{(3)}} \abs{\frac{k^{3/2}}{V_1(b_k)} \P_{a_k\np+b_k\n}\big(S(k) = \y\,,\, H_{\H_{-\n}}>k\big) - \kappa H_0 \frac{\y\cdot\n}{\sqrt{k}} e^{-\frac{\|a\np - \y\|^2}{2k}}} = 0.
    \label{eqn:rwcones-thm5-main}
\end{align}
We begin with the more complicated \cref{eqn:rwcones-thm5-main}.

Set $m: = \lfloor \ep^3k \rfloor$. Our starting point is \cite[Eq.~(82)]{DenisovWachtel15}, reproduced below:
\begin{multline}
    \P_{a_k \np + b_k \n}\big(S(k) = \y \,,\,H_{\H_{-\n}}^S > k \big) \\
    = \sum_{\z \in \H_{\n}} \!\! \P_{a_k \np + b_k \n}\big(S(n-m) = \z \,,\, H_{\H_{-\n}}^S > k-m \big) \P_z\big(S(m) = \y \,,\, H_{\H_{-\n}}^S > m\big)\,. \label{eqn:rwcones-82}
\end{multline}
Let $\H_{\n}^{(1)}(\y) := \{ \z \in \H_{\n} : \norm{\z-\y} < \ep \sqrt{k} \}$. Then we may follow the computations in \cite[Eqs.~(83),(84)]{DenisovWachtel15} exactly, yielding a constant $a> 0$ such that the following  inequalities hold uniformly interm $a_k, b_k$ and $\y$ such that $\y\cdot \n > 2\ep\sqrt{k}$:
\begin{multline}
    \frac{k^{3/2}}{V_1(b_k)}\sum_{\z \in \H_{\n}\setminus \H_{\n}^{(1)}(\y)} \!\!\!\!\!\! \P_{a_k \np + b_k \n}\big(S(k-m) = \z \,,\, H_{\H_{-\n}}^S > k-m \big) \P_z\big(S(m) = \y \,,\, H_{\H_{-\n}}^S > m \big) 
    \\ \leq C  \ep^{-3} e^{-\frac{a}{\ep}} \label{eqn:rwcones-thm5-3-1}
\end{multline}
and 
\begin{multline}
    \frac{k^{3/2}}{V_1(b_k)}\sum_{\z \in \H_{\n}^{(1)}(\y)} \P_{a_k \np + b_k \n}\big(S(k-m) = \z \,,\, H_{\H_{-\n}}^S > k-m \big) \P_z\big(S(m) = \y \,,\, H_{\H_{-\n}}^S < k \big) 
    \\
    \leq C \ep^{-3}e^{-\frac{a}{\ep}} \,. \label{eqn:rwcones-thm5-3-2}
\end{multline}
Both right-hand sides go to $0$ as $\ep \to 0$, and so we turn our attention to the following expression:
\begin{align*}
    \Sigma(\y) = \sum_{\z \in \H_{\n}^{(1)}(\y)} \P_{a_k \np + b_k \n}\big(S (k-m) = \z \,,\, H_{\H_{-\n}}^S > k-m \big) \P_z\big(S(m) = \y\big)\,.
\end{align*}
The bound in \cite[Eq.~(85)]{DenisovWachtel15} also holds uniformly over $a_k, b_k$, and $\y\cdot \n > 2 \ep \sqrt{k}$, except that there should be a $V(\x)$ factor in the $O(\cdot)$-expression of Eq.~(85) there, where for us $\x := a_k \np +b_k \n$ and so $V(\x) = V_1(b_k)$.\footnote{The $V_1(b_k)$ factor comes from an application of \cref{prop:rwcones-tau} in the second line of \cite[Eq.~(85)]{DenisovWachtel15}. This term was dropped in \cite{DenisovWachtel15} because they consider $\x$ fixed, and so $V_1(b_k)$ is order $1$.
} 
Altogether, we find
\begin{align}
    \Sigma(\y) 
    = (2\pi k \ep^3)^{-1}\P_{a_k\np+b_k\n}\big( H_{\H_{-\n}}^S > k-m \big) 
    \Sigma_1(\y)
    + O\Big(\frac{V_1(b_k)}{k^{3/2}} \ep^{-3} e^{-a/\ep}\Big)\,,
    \label{eqn:rwcones-85}
\end{align}
where
\[
    \Sigma_1(\y):= \sum_{\z \in \H_{\n}^{(1)}(\y)} \P_{a_k \np + b_k \n}\big(S(k-m) = \z \given  H_{\H_{-\n}}^S > k-m \big) e^{-\frac{\norm{\y-\z}^2}{2\ep^3k}}\,.
\]
From \cref{prop:thm3} and a compactness argument, we have
\begin{align*}
    \limsup_{k\to\infty}\!\! \sup_{\y\in K^{(3)}} 
    \abs{
        \Sigma_1(\y) - H_0 \int_{\|(1-\ep^3)^{1/2}r - \y/\sqrt{k}\|<\ep} (\mathsf{r} \cdot \n) e^{-\|a \np - \mathsf{r}\|^2/2} e^{-\| (1-\ep^3)^{1/2}\mathsf{r} - \y/\sqrt{k} \|/2\ep^3}  \d \mathsf{r}
    } = 0\,.
\end{align*}
We can follow the steps up to the display before~\cite[Eq.~(86)]{DenisovWachtel15}, appealing this time to the uniform continuity of $(\mathsf{r}\cdot \np) e^{-\|a\np-\mathsf{r}\|^2/2}$ (instead of  the function $u(\mathsf{r})e^{-\|r\|^2/2}$ as in \cite{DenisovWachtel15}), to obtain
\begin{align*}
    \limsup_{k \to \infty} \sup_{\y\in K^{(3)}} \abs{\Sigma_1(\y) - H_0 \frac{\y\cdot \n}{\sqrt{k}} e^{-\frac{\|a\np - \y\|^2}{2k}}} = o(\ep^{3})\,.
\end{align*}
Combining the above with \cref{eqn:rwcones-85} and
applying \cref{prop:rwcones-tau}, we find
\begin{align*}
    \lim_{\ep \to 0} \limsup_{k\to\infty} \sup_{\y \in K^{(3)}} \abs{\frac{k^{3/2}}{V_1(b_k)} \Sigma(\y) - \kappa H_0 \frac{\y\cdot\n}{\sqrt{k}} e^{-\frac{\|a\np - \y\|^2}{2k}}} = 0\,.
\end{align*}
Combining this with \cref{eqn:rwcones-thm5-3-1,eqn:rwcones-thm5-3-2} yields \cref{eqn:rwcones-thm5-main}.

The modifications of~\cite[Section~6.2]{DenisovWachtel15} required to obtain \cref{eqn:rwcones-thm5-zeroregion} are much simpler than those needed to obtain \cref{eqn:rwcones-thm5-main}.
\Cref{eqn:rwcones-thm5-zeroregion} follows very similarly to the proofs of~\cite[Eqs.~(77) and~(81)]{DenisovWachtel15}, and so we just highlight the main differences. The main technical modification comes from the following, which adapts their Lemmas~27 and~28:
\begin{align}
        \P_{u\n}\big(S(k) = a \np + b \n \,,\, H_{\H_{-\n}}^S > k \big) \leq C (1+V_1(u)) n^{-\frac{3}2} \wedge C (1+V_1(u))(1+V_1(b)) n^{-2}
        \label{eqn:tau-bound}
\end{align}
for all $u, a, b \geq 0$.
Indeed, note that \cref{prop:rwcones-tau} gives the bound 
    \begin{align*}
    \P_{a_k \np + b_k\n}\big(H_{\H_{-\n}}^S > k \big) \leq C\big(1+V_1(b_k)\big)k^{-1/2}\,.
    \end{align*}
Then the proofs of Lemmas~27 and~28 of \cite{DenisovWachtel15} can be repeated to yield \cref{eqn:tau-bound} (in particular, the $C(\x)$ in Lemma~27 can be expressed as  $C V_1(\x\cdot\n)$, and the $C(\x,\y)$ in Lemma~28 can be expressed as $C V_1(\x\cdot\n)V_1(\y\cdot\n)$). So all $C(\x)$ terms from the proofs of~\cite[Eqs.~(77),(81)]{DenisovWachtel15} should be replaced by $V_1(b_k)$, after which one finds that their work up to \cite[Eq.~(81)]{DenisovWachtel15} yields \cref{eqn:rwcones-thm5-zeroregion}.
\end{proof}
\end{proposition}

\subsection{Proofs of \cref{prop:rwcones-ballot,prop:2nd-component-genrw-limit,prop:genrw-ballot}}
\label{subsec:pfs-limit-thms}
\begin{proof}[Proof of \cref{prop:rwcones-ballot}]
    In what follows, all estimates will be uniform over $a_k, b_k$, and $u_k$ satisfying \cref{eqn:uniform-parameters}. 

    We begin by following the proof given in \cite[Section~6.3]{DenisovWachtel15}. In particular, we also set $m = \lfloor (1-t)k \rfloor$ for some $t \in (0,1)$ and write the decomposition
    \begin{align}
        &\P_{u_k \n}\Big(S(k) = a_k \np + b_k \n \,,\, H_{\H_{-\n}}^S > k \Big) \nonumber \\
        &= \sum_{\z \in \H_{\n}} \P_{u_k \n}\Big(S(k-m)= \z \,,\, H_{\H_{-\n}}^S > k-m \Big)\P_{a_k \np + b_k \n}\Big(S'(m) = \z \,,\, H_{\H_{-\n}}^{S'} > m \Big)\,,
    \end{align}
    where $S'$ is distributed as $-S$. 
    Letting $\x:= u_k \n$ and $\y:= a_k \np+ b_k \n$, 
    and recalling \cref{eqn:tau-bound} as the needed modification of Lemmas~27 and~28 of \cite{DenisovWachtel15},
    we can follow the proof given in \cite[Sec.~6.3]{DenisovWachtel15} up to their Eq.~(89) to find 
    \[
        \lim_{R \to \infty} \lim_{k \to \infty} \Big(\frac{V_1(u_k)V'(b_k)}{k^2}\Big)^{-1} \Sigma_1(R,k) = 0\,,
    \]
    where 
    \[
        \Sigma_1(R,k) := \!\!\! \sum_{\z \in \H_{\n} : |\z|> R \sqrt{k}} \!\!\! \P_{u_k \n}\Big(S(k-m)= \z \,,\, H_{\H_{-\n}}^S > k-m \Big)\P_{a_k \np + b_k \n}\Big(S'(m) = \z \,,\, H_{\H_{-\n}}^{S'} > m \Big) \,.
    \]
    We are now in a position to apply \cref{prop:rwcones-thm5} to the remainder term
    \begin{align*}
    &\Sigma_2 (R,k) := \\
    &\quad \sum_{\z \in \H_{\n}: \|\z\| \leq R\sqrt k} \P_{u_k \n}\Big(S(k-m)= \z \,,\, H_{\H_{-\n}}^S > k-m \Big)
    \P_{a_k \np + b_k \n}\Big(S'(m) = \z \,,\, H_{\H_{-\n}}^{S'} > m \Big) \\
    &= \frac{H_0^2\kappa^2 V_1(u_k)V_1(b_k)}{\big(t(1-t)\big)^{3/2}k^3}
    \sum_{\z \in \H_{\n}: \|\z\| \leq R\sqrt k} \Big(\frac{\z\cdot\n}{\sqrt{tk}}\Big) \Big(\frac{\z\cdot\n}{\sqrt{(1-t)k}}\Big)
    e^{-\frac{|\z|^2}{2tk} - \frac{|a\sqrt{k}\np-\z|^2}{2(1-t)k}} \nonumber \\
    &\qquad +o\big(V_1(u_k)V'(b_k)k^{-2}\big) \\
    &= \frac{H_0^2\kappa^2 V_1(u_k)V_1(b_k)}{\big(t(1-t)\big)^{3/2}k^3}e^{-\frac{a^2}{2}}
    \!\!\!\!\!\sum_{\z \in \H_{\n}: \|\z\| \leq R\sqrt k} \Big(\frac{\z\cdot\n}{\sqrt{tk}}\Big) \Big(\frac{\z\cdot\n}{\sqrt{(1-t)k}}\Big)
    e^{-\frac{(\z\cdot\n)^2}{2t(1-t)k} -  \frac{(\z\cdot\np-ta\sqrt{k})^2}{2t(1-t)k}}  \nonumber \\
    &\qquad + o\big(V_1(u_k)V'(b_k)k^{-2}\big)\,.
    \end{align*}
    Note that, compared to the first display after \cite[Eq.~(89)]{DenisovWachtel15}, the only different terms  are those involving $a$ and the little-oh terms: this is a consequence of the modifications in our \cref{prop:rwcones-thm5} compared to their Theorem~5. From here, we can follow their  proof step-by-step until the end, yielding \cref{eqn:rwcones-ballot}.
\end{proof}

\begin{proof}[Proof of \cref{prop:2nd-component-genrw-limit}]
We begin by showing convergence of the finite-dimensional distributions. For this, it is enough to consider sequences $a_k\in [-A\sqrt{k}, A\sqrt{k}]$ such that $a_k/\sqrt{k} \to a$, and show that for any $a \in [-A,A]$, the finite-dimensional distributions converge to the same limit (as then every subsequence has a further  subsequence converging to the same distribution).

We proceed by following the arguments as in \cite[Sec.~4]{DurajWachtel20}, making adaptions as necessary. We begin with \cite[Eq.~(40)]{DurajWachtel20}, which states that for any $t \in [0,1)$ and $B \in \sigma(\{S_2(i) , i \leq kt\})$, we have
\begin{align}
    &\P_{u_k \n}\big(B \given S(k) = a_k \np + b_k \n \,,\, H_{\H_{-\n}}^S> k \big) 
    = \E\Big[ h_{u_k,a_k,b_k}^{(k)}\big(t, X_{k,t}\big) \indset{B}  \given H_{\H_{-\n}}^S >kt \Big] \,,
    \label{eqn:40}
\end{align}
where $X_{k,t} := S(\floor{tk})/\sqrt{k}$ and 
\begin{align*}
    h_{u_k,a_k,b_k}^{(k)}(t,\w)
    = \frac{\P_{u_k\n}\big(H_{\H_{-\n}} > kt \big)
    \P_{\w\sqrt{n}}\Big(S((1-t)k) = a_k \np+b_k\n \,,\, H_{\H_{-\n}}^S >(1-t)k \Big)}{\P_{u_k \n}\Big(S(k) = a_k\np+b_k\n \,,\, H_{\H_{-\n}}^S > k\Big)}\,.
\end{align*}
From \cref{prop:rwcones-ballot,prop:rwcones-tau}, we have
\[
    \frac{\P_{u_k\n}\big(H_{\H_{-\n}} > kt \big)
    }{\P_{u_k \n}\Big(S(k) = a_k\np+b_k\n \,,\, H_{\H_{-\n}}^S > k\Big)} \sim 
    \frac{ t^{-1/2}  e^{a^2/2}}{C_1 \kappa V_1'(b_k)}k^{3/2} \,.
\]
Now, let $S'(\cdot)$ denote the random walk whose increments are independent copies of $-X$. Considering the walk $S(\cdot)$ in reversed time, we have 
\begin{multline*}
    \P_{\w\sqrt{n}}\Big(S((1-t)k) = a_k \np+b_k\n \,,\, H_{\H_{-\n}}^S >(1-t)k \Big ) \\
    = \P_{a_k \np+b_k\n}\Big( S'((1-t)k) = \w\sqrt{k} \,,\, H_{\H_{-\n}}^{S'} >(1-t)k \Big) \,.
\end{multline*}
Applying
\cref{prop:rwcones-thm5} yields
\begin{multline*}
    \limsup_{k\to\infty} \sup_{\w\in  \H_{\n}} \bigg| \frac{(1-t)^{3/2}k^{3/2}}{V_1'(b_k)} \P_{\w\sqrt{k}}\Big(S((1-t)k) = a_k \np+b_k\n \,,\, H_{\H_{-\n}}^S >(1-t)k \Big ) 
    \\
    - \kappa H_0  \frac{\w\cdot \n}{(1-t)^{1/2}} e^{-\frac{|a\np - \w|^2}{2(1-t)}} \bigg|
    = 0
\end{multline*}
uniformly over $a_k$ and $b_k$. Altogether, we find 
\begin{align*}
    h_{u_k,a_k,b_k}^{(k)}(t,\w) &= \big(1+o(1)\big) h(a,t,\w) \,,
\end{align*}
uniformly over $\w \in \H_{\n}$, where
\[
    h(a,t,\w) := \frac{H_0}{C_1} t^{-\frac{1}2}
    (1-t)^2 
    (\w\cdot \n)
    e^{- \frac{|\w\cdot \n|^2}{2(1-t)} 
    }
    \exp \Big(- \frac{|\w\cdot \np|^2}{2(1-t)} + \frac{a \w\cdot\np}{1-t} - \frac{ta^2}{2(1-t)} \Big) \,.
\]
Recall $D[0,t]$ the space of cadlag functions from $[0,t]$ to $\R$. For any bounded and continuous functional $g_t : D[0,t] \to \R$, \cref{eqn:40} gives us
\begin{multline}
\E_{u_k \n}\big[g_t(X_{k,.}\cdot \n) \given S(k) = a_k \np + b_k \n \,,\, H_{\H_{-\n}}^S > k \big] \\
= \big(1+o(1)\big)\E_{u_k \n}\big[g_t(X_{k,.}\cdot \n) h\big(a,t, X_{k,t}\big) \given  H_{\H_{-\n}}^S > kt \big] \,,
\label{eqn:gt-formula-1}
\end{multline}
where we have used that for fixed $a$ and $t$, $h(a,t,\w)$ is uniformly bounded in $\w$.
Applying the convergence result \cite[Theorem~2]{DurajWachtel20} (see also Remark~1 in \cite{DurajWachtel20}), we find
\begin{multline}
    \lim_{k \to \infty} \E_{u_k \n}\big[g_t(X_{k,.}\cdot \n) \given S(k) = a_k \np + b_k \n \,,\, H_{\H_{-\n}}^S > k \big] \\
    =
    \E\Big[g_t\big(t^{\frac{1}2}\fM_{\H_{\n}}\cdot \n\big) h\big(a,t,t^{1/2}\fM_{\H_{\n}}(1)\big)  \Big]\,,
    \label{eqn:gt-limit-formula-3}
\end{multline}
where $(\fM_{\H_{\n}}(s))_{s\in [0,1]}$ denotes the Brownian meander  in $\H_{\n}$ started from the origin.
Recall that $\fM_{\H_{\n}}(s) = W_s \np +  M_s \n$, where $(W_s)_{s\in[0,1]}$ is a standard one-dimensional Brownian motion, $(M_s)_{s\in[0,1]}$ is a standard one-dimensional Brownian meander, and $W_.$ and $M_.$ are independent processes. Then the expectation over $W_.$ factors out of the right-hand side of \cref{eqn:gt-limit-formula-3} as follows:
\[
\E\Big[\frac{H_0}{C_1}(1-t)^{-\frac{3}2} M_1 e^{-\frac{t|M_1|^2}{2(1-t)} }g_t(t^{1/2}M_.) \Big]
\E \Big[ \exp \Big(- \frac{t\calN^2}{2(1-t)}+ \frac{a\sqrt{t}}{1-t}\calN-\frac{ta^2}{2(1-t)}\Big) \Big]\,,
\]
where $\calN \sim N(0,1)$.
That second expectation evaluates to $\sqrt{1-t}$ (in particular, there is no dependence on $a$ in the above expression). So, letting $C_2:=H_0/C_1$, we've found 
\begin{multline*}
    \lim_{k \to \infty} \E_{u_k \n}\big[g_t(X_{k,.}\cdot \n) \given S(k) = a_k \np + b_k \n \,,\, H_{\H_{-\n}}^S > k \big] \\
    = \E\Big[C_2(1-t)^{-\frac{3}2} M_1 e^{-\frac{t|M_1|^2}{2(1-t)} }g_t(t^{\frac12}M_.) \Big]\,.
\end{multline*}
Thus, for every fixed $t <1$, we have shown convergence in distribution on $D[0,t]$ to the same limit for every $a$. This in particular implies convergence of all finite dimensional distributions to the same limit for every $a$, as well as tightness on $[0,1-\delta]$, for any $\delta \in (0,1)$. Tightness on $[1-\delta,1]$ follows by applying the exact same arguments for the random walk reversed in time, just as in the end of \cite[Section~4]{DurajWachtel20}.

Thus, we have the convergence of the family of laws $\mathbf{Q}^k_{u_k,a_k,b_k}$. The limit may be identified as that of the standard Brownian excursion by taking $S(i) := \hat{S}_1(i) \np + \hat{S}_2(i) \n$, where $\hat{S}_1(\cdot)$ and $\hat{S}_2(\cdot)$ are independent simple symmetric random walks on $\Z$.
\end{proof}

\begin{proof}[Proof of \cref{prop:genrw-ballot}]
Following the proof of \cite[Thm~5.1]{IOVW20} until their Eq.~(59) yields equation \cref{eqn:genrw-ballot-goodk}. \Cref{eqn:genrw-ballot-leadingorder} is then an immediate consequence of \cref{prop:rwcones-ballot}.
\end{proof}

\section{Proof of \cref{mainthm:FS}}
\label{sec:FS}
This section is dedicated to the proof of \cref{mainthm:FS}. 
We prove a slightly more detailed result below. Recall the generator $\mathsf{L}$ of the relevant Ferrari--Spohn diffusion defined in \cref{def:fs-generator}. 

\begin{theorem} \label{thm:fs-detailed}
    Fix an integer $n \geq 0$. Suppose that $a_L$, the fractional part of $\tfrac{1}{4\beta}\log L$, converges to some limit $a$, and let $\lambda>0$ denote the $L\to\infty$ limit of $\lambda^{(n)}(L):= c_{\infty}e^{4\beta a_L}(1-e^{-4\beta})e^{4\beta n}$. 
    Define $\y:= (\lfloor KL^{2/3} \rfloor,0) $, for some $K >0$, and the box $Q:= \llbracket -KL^{2/3}, KL^{2/3}\rrbracket^2$.
    Consider the modified Ising polymer $\gamma$ in $D =Q$ or $\H$ with start point $-\y$ and end point $\y$ with area tilt $\exp(-\frac{\lambda^{(n)}(L)A(\gamma)}{L})$.
    Let $\overline{\gamma}(x)$ denote the maximum vertical distance of $\gamma$ at $x \in \R$. Let $\sigma>0$ be the constant from \cref{mainthm:BM-for-Ising-polymers}, Part (b). For any fixed $T>0$, the diffusively-rescaled interface $\sigma^{-\frac{1}2}L^{-\frac{1}3}\overline{\gamma}(L^{\frac{2}3} x)$ converges weakly in $(D[-T,T], \|\cdot\|_{\infty})$  as  first $L\to\infty$ then $K\to\infty$  to the stationary Ferrari--Spohn diffusion on $(0,\infty)$ with generator 
    $\mathsf{L}$
    and Dirichlet boundary condition at $0$. The same holds for $\underline{\gamma}(x)$, the minimum vertical distance at $x$.
\end{theorem}

\begin{remark}
    Fix $\beta$ sufficiently large, and consider the SOS model with a floor on the box $Q$ with boundary conditions $H-n$ everywhere, except on the bottom where they are $H-n-1$. \cref{eqn:contourLawAreaTerm} and \cref{obs:sos-is-ist} imply that the law of the $H-n$-level line (connecting the bottom corners of $Q$) is given by a modified  Ising polymer  with area tilt as in \cref{thm:fs-detailed}. Thus, \cref{thm:fs-detailed} implies \cref{mainthm:FS}.
\end{remark}

\begin{remark}\label{rem:K-L-1/20}
The argument used to prove the above theorem holds mutatis mutandis for $K\in (0,L^\epsilon)$ for, say, $0<\epsilon<\frac1{20}$, where the restriction on $\epsilon$ is due to the fact that we are able to control the effect of the area tilt term $\exp[-\frac\lambda L A(\gamma)]$ on boxes of side-length $L^{2/3+\epsilon}$ (via~\cite[Prop.~A.1]{CLMST16}). 
\end{remark}

\subsection{Proof of \cref{thm:fs-detailed}}
Let $\bP_D^{\u,\v}$ be the  modified Ising polymer in $D = Q$ or $\H$  with start-point $\u\in \Z^2$ and end-point $\v\in \Z^2$, i.e.,
\[
    \bP_D^{\u,\v}(\cdot) := \frac{\Gb_D(\u \to \v \given \gamma \in \cdot)}{\Gb_D(\u \to \v)}\,,
\]
where we recall the partition functions $\Gb_D(\u \to \v)$ from \cref{def:general-partition-fn}.\footnote{In \cref{def:mod-ising-polymer}, we defined modified Ising polymers with start-point zero for ease of notation, and because until now, essentially all of our Ising polymers started from $0 \in \Z^2$.}
Let $\bE_{D}^{\u,\v}$ denote expectation with respect to $\bP_D^{\u,\v}$.
Next, define the modified Ising polymer with area-tilt
\begin{align}\label{def:area-tilted-Ising-polymer}
    \tilde{\bP}^{-\y,\y}_{D,\lambda}(\cdot) := \frac{\bE^{-\y,\y}_D \big[\ind{\gamma \in \cdot} e^{-\frac{\lambda}L A(\gamma)} \big]}{\bE^{-\y,\y}_D[e^{-\frac{\lambda}L A(\gamma)}]} \,.
\end{align}
As  explained above \cref{eqn:ising-polymer-animals}, we will also view $\bP_{D,\lambda}^{-\y,\y}$ as a measure on animals $\Gamma$ (in which $A(\Gamma) := A(\gamma)$).
Note that we have replaced $\lambda^{(n)}(L)$ with its $L\to\infty$ limit $\lambda$, and as the next claim (controlling the influence of the area tilt) will show, the difference between these will play no role in our estimates.
\begin{claim} \label{claim:area-tilt}
For each fixed $K>0$, with $\y=(\lfloor K L^{2/3}\rfloor,0)$ as in \cref{thm:fs-detailed},
there exists $a_K >0$ such that $\bE_D^{-\y,\y}\big[e^{-\frac{\lambda}{L} A(\gamma)}\big] \to a_K$ as $L\to \infty$.
\begin{proof}
From \cref{mainthm:BM-for-Ising-polymers}, we know that both $L^{-1/3}\underline{\gamma}(\lfloor L^{2/3}x\rfloor)$ and $L^{-1/3}\overline{\gamma}(\lfloor L^{2/3}x\rfloor)$ converge weakly in $D([-K,K], \|\cdot\|_{\infty})$ to the Brownian excursion from $0$ to $0$. Since the function
\begin{align*}
    f: (D[-K,K],\|\cdot\|_{\infty})  &\to \R \\
    g &\mapsto \exp\Big(-\int_{-K}^K |g(x)| \d x\Big)
\end{align*}
is a bounded, continuous function,  it follows after a change of variables $x \mapsto L^{2/3}x$ that 
\[
    \bE_D^{-\y,\y}[e^{-\frac{\lambda}{L}A(\gamma)}] \to \bE[e^{-\lambda \int_{-K}^K \xi(x) \d x}] \,,
\]
where $\xi$ denotes a Brownian excursion from $0$ to $0$ on $[-K,K]$. This concludes the proof.
\end{proof}
\end{claim}
The above claim (applied to the denominator of \cref{def:area-tilted-Ising-polymer} to show the effect of the area tilt there is uniformly bounded, whereas its effect on the numerator can be isolated via Cauchy--Schwarz) now implies that the total-variation distance between $\tilde\bP^{-\y,\y}_{D,\lambda}$ and $\tilde\bP^{-\y,\y}_{D,\lambda^{(n)}}$ 
vanishes as $L\to\infty$.

We prove \cref{thm:fs-detailed} by first coupling the cone-points of $\Gamma \sim \bP_{D,\lambda}^{-\y,\y}$ that lie inside some strip (those cone-points for which we will have entropic repulsion) with the trajectory of a random walk with area-tilt. We then fit this random walk into the framework of \cite[Section~6]{IOSV21}, where a Ferrari--Spohn limit was proved for a broad class of directed, 2D random walks with area tilt.

\subsubsection{Coupling with an area-tilted random walk}
This subsection will closely follow the notation and  work in \cref{subsec:random-walk-coupling}, where  the existence of a coupling between the cone-points of a modified Ising polymer (\emph{without} area tilt) and the corresponding  random walk was proved. The inputs to construct such a coupling were:
\begin{enumerate}[(a)]
\item existence of many cone-points and boundedness of the  polymer length (\cref{lem:generalized-length-cpts}); \label{property:many-cone-points}
\item boundedness of the irreducible pieces (\cref{prop:bounded-irreducible-pieces}); and
\label{property:boundedness}
\item entropic repulsion (\cref{prop:animal-entropic-repulsion}). \label{property:entropic-repulsion}
\end{enumerate}
These results all held with probability tending to $1$ as $N$ (the side length of the box) tends to infinity. In the current situation, we take  $N= 2\|\y\|_1 = 2\lfloor KL^{2/3} \rfloor$ and we shift $\x$ (as in the statement of the above results) to $\y$ so that these results hold under $\bP_D^{-\y,\y}$ with probability tending to $1$ as $L$ tends to $\infty$. \cref{claim:area-tilt} states that $\bP_D^{-\y,\y}[e^{-\frac{\lambda}{L}A(\gamma)}]$ is bounded away from $0$ uniformly in $L$, so that \cref{property:many-cone-points,property:entropic-repulsion} above hold with probability tending to $1$ as $L$ tends to $\infty$ under $\bP_{D,\lambda}^{-\y,\y}$ as well. Thus, we are in good position to establish a coupling between the area-tilted Ising polymer and an area-tilted random walk that we define below.

Fix $\delta \in (0,1/8)$. Analogous to \cref{def:P*}, define the measure
\[
    \bP_{D,\lambda,*}^{\overline{\zeta}^{(\Lstar)}, \overline{\zeta}^{(\Rstar)}}\big(\cdot\big) := \bP_{D,\lambda}^{-\y,\y}\big(\cdot \given \cP_{D,\delta}^*(-\y,\y), \Gamma^{(\Lstar)} = \overline{\Gamma}^{(\Lstar)}, \Gamma^{(\Rstar)} = \overline{\Gamma}^{(\Rstar)}\big)\,,
\]
where
$\overline{\zeta}^{(\Lstar)}$ and $\overline{\zeta}^{(\Rstar)}$ are points in $\Z^2$ satisfying
\begin{align}
    \overline{\zeta}^{(\Lstar)} &\in [-\tfrac{N}2 + N^{4\delta}, -\tfrac{N}2+ N^{4\delta}+(\log N)^2]\times (N^{\delta}, N^{4\delta}(\log N)^2]\,,
    \label{eqn:zeta-Lstar-condition-areatilt}\\    \overline{\zeta}^{(\Rstar)} &\in [\tfrac{N}2-N^{4\delta}-(\log N)^2, -\tfrac{N}2-N^{4\delta}] \times  (N^{\delta}, N^{4\delta}(\log N)^2] 
    \label{eqn:zeta-Rstar-condition-areatilt}
\end{align}
(these are the analogues of \cref{eqn:zeta-Lstar-condition,eqn:zeta-Rstar-condition}, respectively). Recall that, thanks to \cref{property:boundedness} and \cref{property:entropic-repulsion}, the first and last cone-points of $\Gamma$ in the strip $\cS_{-\frac{N}2+N^{4\delta}, \frac{N}2- N^{4\delta}}$ satisfy \cref{eqn:zeta-Lstar-condition-areatilt,eqn:zeta-Rstar-condition-areatilt} with $\bP_{D,\lambda}^{-\y,\y}$-probability tending to $1$ as $L$ tends to $\infty$. It is between these cone-points that we couple with an area-tilted random walk.

Recall the random walk $\sS$ and its law $\P$ defined in \cref{subsec:rw-model}. Write $\E$ and $\E_{\u}$ for expectation under  $\P$ and $\P_{\u}$, respectively. Define the area under $\sS$ as follows: for an $l$-steps walk $\sS=\{(\sS_1(0),\sS_2(0)),\dots,(\sS_1(l),\sS_2(l))\}$, we write
\begin{equation}
A(\sS) := \sum_{i=1}^l(\sS_1(i)-\sS_1(i-1))\sS_2(i)\,.
\end{equation}
To indicate the law of the random walk $\sS$ with area tilt, started from $\sS(0) = \u$, we write
\begin{equation}\label{def:rw-area-tilt}
\P_{\lambda}^{\u}(\cdot) :=  \frac{\E_\u\big[\ind{\sS\in \cdot}e^{-\frac{\lambda}{L}A(\sS)} \big]}{\E_\u\big[e^{-\frac{\lambda}{L}A(\sS)} \big]}\,.
\end{equation}

\begin{proposition}
There exists $\nu>0$ such that for all $K>0$ and $L$ large enough with respect to $K$, for all $\beta>0$ sufficiently large, and for all $\overline{\zeta}^{(\Lstar)}$ and $\overline{\zeta}^{(\Rstar)}$ satisfying \cref{eqn:zeta-Lstar-condition,eqn:zeta-Rstar-condition} respectively, if we let $T= H_{\overline{\zeta}^{(\Rstar)}}$ and view  $\Cpts(\Gamma^*)$  as an ordered tuple  (see \cref{eqn:cone-points-vector}), then
\[ \Big\|\bP_{D,\lambda,*}^{\overline{\zeta}^{(\Lstar)}, \overline{\zeta}^{(\Rstar)}}\Big(\Cpts(\Gamma^*) \in \cdot\Big) - 
\P_{\lambda}^{\overline{\zeta}^{(\Lstar)}}\Big( 
\big(\sS(i)\big)_{i= 0}^{T}
\in \cdot \given T < H_{\H_{N^{\delta}}} \Big)
\Big\|_{\tv}\leq Ce^{-\nu \beta (\log N)^2}
\]
for some constant $C:=C(\beta)>0$. 
\end{proposition}
\begin{proof}
This follows from the coupling of the untilted measures given by \cref{prop:rw-coupling}. One just needs to check that the difference in the definitions of area for $\Gamma^*$ and $\sS$ results in a negligible difference in the tilts. This is easy to see: by \cref{property:boundedness}, 
the distance between two consecutive cone-points is at most $(\log L)^2$, and hence one can enclose the diamond between them in a square of area $(\log L)^4$. Since the total area of such squares is an upper bound on $|A(\Gamma^*)-A(S)|$, we find
\[
\left|\frac{A(\Gamma^*)}{L} - \frac{A(\sS)}{L}\right| \leq 2K\frac{(\log L)^4}{L^{1/3}} = o(1)\,.
\]
This shows that the tilts are equivalent up to a $o(1)$ factor.
Since \cref{claim:area-tilt} implies that the expectations of $e^{-\frac{\lambda}LA(\Gamma^*)}$  and $e^{-\frac{\lambda}{L}A(\sS)}$ are bounded away from $0$, the result follows.
\end{proof}

\subsubsection{Convergence to Ferrari--Spohn}
We have reduced to a two-dimensional random walk bridge with area-tilt conditioned to stay in $\H_{N^{\delta}}$, with start-point $\bar \u:= \overline\zeta^{(\Lstar)}$ and end-point $\bar \v:= \overline\zeta^{(\Rstar)}$ satisfying \cref{eqn:zeta-Lstar-condition-areatilt,eqn:zeta-Rstar-condition-areatilt}, respectively.
In \cite[Sections~6.6 and 6.7]{IOSV21}, the Ferrari--Spohn diffusion limit is proved for a wide-class of such random walks, with the stronger condition (\cite[Eq.~(6.10)]{IOSV21}) on the start-point $\u$ and end-point $\v$ of the random walk:
\begin{align}
    \u &\in [-\bar{K}L^{2/3}-L^{1/3+\ep}, \bar{K}L^{2/3}+L^{1/3+\ep}] \times [cL^{1/3}, CL^{1/3}] \label{eqn:iosv-endpoint-condition-u} \\
    \v &\in [\bar{K}L^{2/3}- L^{1/3+\ep}, \bar{K}L^{2/3}+L^{1/3+\ep}] \times [cL^{1/3}, CL^{1/3}]\,, \label{eqn:iosv-endpoint-condition-v}
\end{align}
where $\ep >0$ is any small constant, $C>c>0$ are fixed constants, and $\bar{K}\leq K$ is a parameter tending to $\infty$ after $L$ (these conditions are stronger than \cref{eqn:zeta-Lstar-condition-areatilt,eqn:zeta-Rstar-condition-areatilt} in the $y$-coordinates only).
\Cref{lem:rw-goes-high} shows that our random  walk indeed passes through points satisfying \eqref{eqn:iosv-endpoint-condition-u} and~\eqref{eqn:iosv-endpoint-condition-v} with high probability, thereby putting us in the same setting as \cite[Section~6.6]{IOSV21}. We prove it using the Brownian excursion limit \cref{thm:rw-invariance}.

For brevity, we will write the law of the area-tilted random walk bridge as
\[
    \P_{\lambda,+}^{\bar\u,\bar\v} :=  \P_{\lambda}^{\bar\u}(\cdot \given H_{\bar\v} < H_{\H_{N^{\delta}}})\,.
\] 
Similarly, we write the law of the un-tilted random walk bridge via
\[
    \P_{+}^{ \bar\u,\bar\v} :=  \P_{\bar\u}(\cdot \given H_{\bar\v} < H_{\H_{N^{\delta}}})\,.
\]
We will denote expectation under these measures by replacing $\P$ with $\E$.
\begin{lemma}\label{lem:rw-goes-high}
 Let $E_{c,C}$ denote the event that the random walk $\sS$ passes through $\u$ and $\v$ satisfying \cref{eqn:iosv-endpoint-condition-u,eqn:iosv-endpoint-condition-v}, for some $C>c>0$. Then for some constant $K_0>0$, the following limit holds uniformly in $\bar\u$ satisfying \cref{eqn:zeta-Lstar-condition-areatilt} and $\bar\v$ satisfying \cref{eqn:zeta-Rstar-condition-areatilt}:
 \begin{align*}
     \lim_{c\to 0} \inf_{K \geq K_0} \liminf_{L\to\infty} \P_{\lambda, +}^{\bar\u, \bar\v} (E_{c,C}) = 1\,.
 \end{align*}
\begin{proof}
We will show that $\sS$ passes through a point $\u$ satisfying \cref{eqn:zeta-Lstar-condition-areatilt} with high probability;  the same argument will show the same is true for $\v$. Further, one only needs to check that $\u_2 \geq cL^{1/3}$ (the upper bound by $CL^{1/3}$ is not needed), since $\sS_2(0) = \bar\u_2< cL^{1/3}$, and so the first time $\sS_2$ rises above $cL^{1/3}$, it will also be below $CL^{1/3}$ for $L$ large enough (recall from \cref{eqn:random-walk-increments-bound} that under the non-area-tilted measure, $\sS$ has  increments bounded by $(\log L)^2$ with probability tending to $1$ as $L\to\infty$, and so by \cref{claim:area-tilt}, the same is true under the area-tilted measure).

Now, using the localization of $\sS_1$ (\cref{claim:rw-x-size}) and the triviality-in-$L$ of the area tilt (\cref{claim:area-tilt}), we have that for all $\ep >0$ small, the following hold with $\P_{\lambda,+}^{\bar\u,\bar\v}$-probability tending to $1$ as $L\to\infty$:
\begin{align}
    \mathcal{E}_1(L^{2/3}) &:= \big\{\sS_1(L^{2/3}) \in [-(K-\mu) L^{2/3} - L^{1/3+\ep}, -(K-\mu)L^{2/3} +L^{1/3+\ep}] \big\}\,, \text{ and} 
    \nonumber\\
    \mathcal{E}_1(2L^{2/3}) &:= \big\{\sS_1(2L^{2/3}) \in [-(K-2\mu) L^{2/3} - L^{1/3+\ep}, -(K-2\mu)L^{2/3} +L^{1/3+\ep}] \big\} \,.
    \label{eqn:S1-localization-areatilt}
\end{align}
Thus, the lemma will be proved upon showing
\begin{align}
    \lim_{c\to0} \sup_{K\geq K_0} \limsup_{L\to\infty} \P_{\lambda,+}^{\bar\u,\bar\v}(\sS_2(L^{2/3})< cL^{1/3} \given \sS_2(2L^{2/3})<cL^{1/3} ~,~\cE_1(2L^{2/3})) = 0\,. 
    \label{eqn:rw-goes-high-suff1}
\end{align}
Write the above probability as 
\begin{align}
    \frac{\E_+^{\bar\u,\bar\v}\Big[ \E_{+}^{\bar\u, \sS(2L^{2/3})}\Big[\ind{\sS_2(L^{2/3})< cL^{1/3}}e^{-\frac{\lambda}{L}A(\sS)}\Big]
    F(\sS(i), i \geq 2L^{2/3}) \Big]
    }
    {\E_+^{\bar{\u},\bar\v}\Big[\E_{+}^{\bar\u, \sS(2L^{2/3})}\Big[e^{-\frac{\lambda}{L}A(\sS)}\Big]F(\sS(i), i \geq 2L^{2/3})\Big]}\,,
\end{align}
where 
\[
F(\sS(i), i \geq 2L^{2/3}) := e^{-\frac{\lambda}{L}A(\sS(i), i \geq 2L^{2/3})} \ind{\sS_2(2L^{2/3})<cL^{1/3} ,\ \cE_1(2L^{2/3})}\,.
\]
Due to the restriction on $\sS(2L^{2/3})$ imposed by the indicator defining $F$, \cref{eqn:rw-goes-high-suff1} will be shown if we can prove 
\begin{align}\label{eqn:rw-goes-high-suff2}
     \lim_{c\to0}\sup_{K\geq K_0}\limsup_{L\to\infty}\frac{\E_{+}^{\bar\u, \bar \w}\Big[\ind{\sS_2(L^{2/3})< cL^{1/3}}e^{-\frac{\lambda}{L}A(\sS)}\Big]}{\E_{+}^{\bar\u, \bar \w }\Big[e^{-\frac{\lambda}{L}A(\sS)}\Big]} = 0\,,
\end{align}
uniformly over
\[
\bar\w \in [-(K-2\mu) L^{2/3} - L^{1/3+\ep}, -(K-2\mu)L^{2/3} +L^{1/3+\ep}] \times [0,cL^{1/3})\,.
\]
Note that the $\sup$ over $K$ is trivial: indeed, $K$ plays no role in the limit in \cref{eqn:rw-goes-high-suff2}, since by horizontal shift-invariance, we may  shift to the right by $KL^{2/3}$. The denominator may be bounded from below as
a constant, by shifting both $\bar\u_2$ and $\bar\w_2$ up to $cL^{1/3}$, and then applying the Brownian excursion limit as in the proof of \cref{claim:area-tilt}. The area tilt in the numerator may be bounded from above by $1$. Thus, we have reduced to showing the following:
\[
    \lim_{c\to0}\limsup_{L\to\infty} \P_{+}^{\bar\u, \bar \w}(\sS_2(L^{2/3})< cL^{1/3}) = 0\,,
\]
uniformly over $\bar u$ and $\bar w$.
By monotonicity, we may assume $\bar{w}_2 = 0$. The result now follows from the Brownian excursion limit of $\sS_2$ under $\P_{+}^{\bar\u, \bar \w}$ (\cref{thm:rw-invariance}).
\end{proof}
\end{lemma}

Following \cite{IOSV21}, we are now ready to prove \cref{thm:fs-detailed}.

\begin{proof}[Proof of \cref{thm:fs-detailed}]
We have reduced to an area-tilted, directed 2D random walk bridge between $\u$ and $\v$ satisfying \cref{eqn:zeta-Lstar-condition-areatilt} and~\eqref{eqn:zeta-Rstar-condition-areatilt} respectively,  conditioned to stay positive. That is, we have a random walk $\sS$ under law 
\[
    \sS \sim \P_{\lambda,+}^{\u,\v}\,.
\]
This puts us in the framework of 
the proof of the Ferrari--Spohn limit for such random walks given in \cite[Section~6.6 and 6.7]{IOSV21}. Tightness follows exactly as in \cite[Section~6.6]{IOSV21}. The proof of finite-dimensional distributions in \cite[Section~6.7]{IOSV21} had just two additional inputs: their Proposition~6.2 and Lemma~6.3, and so we will be done as soon as we establish our analogues of these results.

Their Proposition~6.2 holds exactly the same for us. Their Lemma~6.3 does as well, with the exception that the constant in front of our area tilt is slightly different, and thus the resulting generator is slightly different as well. We re-state and prove that result  in our setting. 

For any $n\in\N$, we will write $\sS[0,n]:= (\sS(i))_{i\in [0,n]}$. 
Following \cite[Eq.(6.6)]{IOSV21}, for any function $f: \N \to \R$ and $u \in \H$, define the $n$-step partition function 
\[
    \Gb_{\lambda,L,+}^n[f](\u) := \E_{\u}\Big[ e^{ - \frac{\lambda}{N}A(\sS[0,n])} f(\sS_2(n))\ind{\sS[0,n] \subset \H}\Big]\,.
\]
Recall from \cref{subsec:rw-model} that $\sS$ under law $\P$ has step-distribution $X = (X_1, X_2)$.
Recall also $\sigma^2 := \Var(X_2)/\mu$ from the discussion above \cref{thm:rw-invariance}.
Following \cite[Section~6.5]{IOSV21}, define the following operator on smooth test functions $f$ with compact support in $(0,\infty)$:
\[
    \mathbf{T}_L f(r) := \E\Big[e^{-\frac{\lambda}{L}X_1 r L^{1/3} \sigma} f\Big(r+ \frac{X_2}{L^{1/3} \sigma} \Big) \ind{r+ \frac{X_2}{L^{1/3} \sigma} \geq 0} \Big]\,.
\]
Note that the indicator may be dropped, as $f$ is supported above $0$.
By Taylor expanding $f$ to second order and $e^x-1$ to first order (the errors are $o(L^{-2/3})$), we have
\begin{multline*}
    \lim_{L\to\infty} \frac{\mathbf{T}_L-\mathrm{Id}}{L^{-2/3}}f(r) \\
    = 
    \lim_{L\to\infty}
    L^{2/3} \E\Big[\big(e^{-\frac{\lambda}{L}X_1 r L^{1/3} \sigma} -1\big)f(r) 
    + e^{-\frac{\lambda}{L}X_1 r L^{1/3} \sigma} 
    \Big(f'(r)\frac{X_2}{L^{1/3}\sigma} + \tfrac{1}2 f''(r)\frac{X_2^2}{L^{2/3}\sigma^2}\Big)
    \Big] \\
    =\lim_{L\to\infty}
    L^{2/3} \E\Big[-\frac{\lambda}{L^{2/3}}X_1 r  \sigma f(r) 
    + e^{-\frac{\lambda}{L}X_1 r L^{1/3} \sigma} 
    \Big(f'(r)\frac{X_2}{L^{1/3}\sigma} + \tfrac{1}2 f''(r)\frac{X_2^2}{L^{2/3}\sigma^2} \Big) 
    \Big]\,.
\end{multline*}
Now use $\mu:= \E[X_1]$, $\E[X_2] = 0$, and $\Var(X_2) = \E[X_2^2]$,  so that the $f'(r)$ term vanishes and we find
\begin{align*}
    \lim_{L\to\infty} \frac{\mathbf{T}_L-\mathrm{Id}}{\mu L^{-2/3}}f(r)  =  \tfrac{1}2 f''(r) - \lambda \sigma r f(r)   =: \mathcal{L} f(r)\,.
\end{align*}
Kurtz's semigroup convergence theorem (\cite[Theorem~1.6.5]{EthierKurtz}, see also \cite[Eq.(2.34)]{ISV15}) then gives
\[
    \lim_{L\to\infty} \mathbf{T}_L^{\lfloor tL^{2/3}/\mu \rfloor} f = e^{t \mathcal{L}} f\,,
\]
uniformly over $t$ in bounded intervals of $\R_+$.
Define the rescaling operator
\[
    \mathrm{Sc}_L f(x) = f(xL^{-1/3} \sigma^{-1})\,.
\]
Observing that $\Gb_{\lambda,L,+}^{k}[ \mathrm{Sc}_L f] =\mathbf{T}_L^k f$, the second-to-last-display immediately yields
\[
    \lim_{L\to\infty} \Gb_{\lambda,L,+}^{\lfloor tL^{2/3}/\mu \rfloor}[ \mathrm{Sc}_L f] = e^{t\mathcal{L}}f\,,
\]
which is our version of \cite[Lemma~6.3]{IOSV21}. 
\end{proof}

\appendix

\section{Cluster expansion and open contours in the SOS model}
\label{appendix:cluster-expansion}
\subsection{Cluster expansion}
\label{subsec:cluster-expansion}
    Consider the  SOS model in a region $\Lambda \Subset \Z^2$ without floor, with boundary condition $j\in\Z$. Fix any subset $U$ of the inner boundary of $\Lambda$, and condition on $\{\varphi|_U\geq j\}$. Let $\widehat Z_{\Lambda, U} := \widehat{Z}_{\Lambda,U}^{j,+}$ denote the partition function of this model. Note that the partition function has no dependence on $j$ and is also unchanged by a replacement of the conditioning with  $\{\varphi|_U \leq j\}$. \cite{CLMST14} proves that there exists a constant $\beta_0>0$ such that for all $\beta>\beta_0$, one has 
    \begin{equation}\label{expansion}
    \log \widehat Z_{\Lambda, U} = \sum_{V \subset \Lambda} f_U(V)
    \end{equation}
    for some function $f_U$. \Cref{expansion} was proven using the main theorem from \cite{KoteckyPreiss86}, which can actually be used to show (with the same proof as in \cite{CLMST14}) that the formula holds simultaneously for all subsets $\Lambda'\subset\Lambda$ (with the same $f_U$):    \begin{equation}\label{expansionPlus}
    \log \widehat Z_{\Lambda', U} = \sum_{V \subset \Lambda'} f_U(V)
    \end{equation}
    (here, $Z_{\Lambda',U}$ is defined as before, with the understanding that the condition on $U$ only needs to be satisfied on $\Lambda'\cap U$, or in other words $Z_{\Lambda',U}:=Z_{\Lambda',U \cap \Lambda'}$). The latter observation leads, using M\"obius inversion, to the following formula for $f_U$:
    \begin{equation}\label{inversion}
    f_U(V) = \sum_{W \subset V}(-1)^{|V|-|W|}\log \widehat Z_{W, U}\,.
    \end{equation}
    From \cref{inversion}, we can read off some properties of the function $f_U$:
    \begin{enumerate}
    \item \label{decorationsPropertyA} $f_U(V)$ only depends on $U$ only through $U\cap V$, and in particular, if $U\cap V = \emptyset$, $f_U(V) = f_0(V)$ for some universal function $f_0$. Moreover, $f_0(V)$ only depends on the ``shape'' and size of the volume $V$, i.e., $f_0(\cdot)$ is invariant with respect to lattice symmetries.
    \item If $V$ is not connected, $f_U(V)=0$. Indeed, if $V=V_1 \sqcup V_2$, each log-partition function on the right-hand side of \cref{inversion} splits into the sum of two terms, both of which appear in the global sum with the same frequency of plus and minus signs.
    \item There exists a constant $\beta_0>0$ such that
    \begin{equation}\label{decorationsDecay}
    \sup_U |f_U(\cC)| \leq \exp\big((\beta-\beta_0)d(\cC)\big)\,,
    \end{equation}
    where $d(\cC)$ denotes the cardinality of the smallest connected set of bonds of $\Z^2$ containing all boundary bonds of $\cC$ (i.e., bonds connecting $\cC$ to $\cC^c$), see \cite[Proposition~3.9]{DKS92} for details.
    \end{enumerate}
    
\subsection{Proof of \cref{prop:soscontour-law}}
\label{subsec:proof-soscontour-law}
Recall $\widehat{Z}_{\Lambda,U}^{j,+}$ from above \cref{expansion}, and define $\widehat{Z}_{\Lambda,U}^{j,-}$ similarly but with the conditioning $\{\varphi|_U \leq j\}$ (as noted above \cref{expansion}, $\widehat{Z}_{\Lambda,U}^{j,+} = \widehat{Z}_{\Lambda,U}^{j,-}$. The different  notation is mostly for clarity in the calculations that follow).
Recall also $\Lambda_{\gamma}^+, \Lambda_{\gamma}^-, \Delta_{\gamma}^+,$ and $\Delta_{\gamma}^{-}$ from above \cref{prop:soscontour-law}.

Observe that 
\begin{align}
    \widehat Z_{\Lambda}^{\xi} = \sum_{\gamma} e^{-\beta|\gamma|} \widehat Z_{\Lambda_{\gamma}^+, \Delta_{\gamma}^+}^{1,+} \widehat Z_{\Lambda_{\gamma}^-,\Delta_{\gamma}^-}^{0,-}\,,
    \label{eqn:claim-pf-2}
\end{align}
This is a simple consequence of the following identity: for any $\u, \v \in \Z^2$ such that  $\varphi_\u \geq 1$ and $\varphi_{\v} \leq 0$, we have
\begin{align*}
         |\varphi_\u - \varphi_{\v}| = \abs{(\varphi_\u -1) - \big(\varphi_{\v} - 0\big) + (1-0)} = (\varphi_\u -1) + (0 - \varphi_{\v}) + 1\,.
\end{align*}
Now, cluster-expanding each of the partition functions on the right-hand side of \cref{eqn:claim-pf-2} yields
\begin{align}
    \widehat Z_{\Lambda_{\gamma}^+, \Delta_{\gamma}^+}^{1} \widehat Z_{\Lambda_{\gamma}^-,\Delta_{\gamma}^-}^{0}
    &= \exp \Big(\sum_{\substack{\cC \subset \Lambda \\ \cC \cap \Delta_{\gamma} = \emptyset}} f_0(\cC) + \sum_{\substack{\cC \subset \Lambda_{\gamma}^+ \\ \cC \cap \Delta_{\gamma} \neq \emptyset}} f_{\Delta_{\gamma}^+}(\cC) + \sum_{\substack{\cC \subset \Lambda_{\gamma}^- \\ \cC \cap \Delta_{\gamma} \neq \emptyset}} f_{\Delta_{\gamma}^-} (\cC) \Big) 
    \nonumber \\
    &= e^{\sum_{\cC\subset \Lambda} f_0(\cC)} \exp\Big( \Psi_{\Lambda}(\gamma)\Big)  =\widehat Z_{\Lambda}^{0} \exp\Big( \Psi_{\Lambda}(\gamma)\Big)  \,, \label{eqn:claim-pf-3}
\end{align} 
where 
\[
\Psi_{\Lambda}(\gamma) := - \sum_{\substack{\cC \subset \Lambda \\ \cC \cap \Delta_{\gamma} \neq \emptyset}} f_0(\cC) + \sum_{\substack{\cC \subset \Lambda_{\gamma}^+ \\ \cC \cap \Delta_{\gamma} \neq \emptyset}} f_{\Delta_{\gamma}^+}(\cC) + \sum_{\substack{\cC \subset \Lambda_{\gamma}^- \\ \cC \cap \Delta_{\gamma} \neq \emptyset}} f_{\Delta_{\gamma}^-}(\cC)\,.
\]
Defining
\begin{equation}
\phi(\cC;\gamma):= -f_0(\cC) + f_{\Delta_{\gamma}^+}(\cC)\ind{\cC\cap \Delta_{\gamma}^- =\emptyset} + f_{\Delta_{\gamma}^-}(\cC) \ind{\cC\cap \Delta_{\gamma}^+ =\emptyset},
\label{def:Phi}
\end{equation}
it follows from \cref{eqn:claim-pf-2,eqn:claim-pf-3} that
\[
\widehat{Z}_{\Lambda}^{\xi} =\widehat{Z}_{\Lambda}^0\sum_{\gamma} \exp\Big(-\beta|\gamma| + \sum_{\substack{\cC \subset \Lambda \\ \cC \cap \Delta_{\gamma} \neq \emptyset}} \phi(\cC;\gamma)\Big)\,.
\]
\Cref{eqn:sos-contour-law,eqn:sos-contour-partitionfunction} follow immediately. From the definition of $\phi$ in \cref{def:Phi} and the properties of $f_U$ from \cref{subsec:cluster-expansion}, we obtain \cref{it:ce-conn,it:ce-local,it:ce-translation,it:ce-decay} of the proposition.
\qed

\subsection{Law of an open contour in the SOS model with floor}
\label{subsec:open-contour-floor}
Consider the SOS model $\pi_{R}^{\xi}$ in a domain $R$ with floor at $0$ and  boundary condition $\xi \in \{h-1, h\}^{\partial \Lambda}$ inducing a unique open 
$h$-contour $\gamma$.
Consider  $h=\lfloor \frac{1}{4\beta}\log L \rfloor - n$, for $n \in \N$ fixed relative to $L$, and assume that the domain $R$ satisfies $|R| \leq L^{\frac{4}{3}+2\epsilon}$ and $|\partial R| \leq L^{\frac{2}{3}+2\epsilon}$, for some $\ep \in (0,1/20)$. In this subsection, we derive an asymptotic expression for the law of the random contour $\gamma$ in terms of the no-floor law $\widehat{\pi}_R^{\xi}$. 

For a subset $U$ of the inner boundary of $R$, recall $\widehat{Z}_{R,U}^{j, \pm}$ as in \cref{subsec:proof-soscontour-law}, and define $Z_{R,U}^{j,\pm}$ similarly but with a floor at $0$. 
As in \cref{eqn:claim-pf-2}, we can write the law of $\gamma$ in terms of the partition functions above and below it:
\begin{equation*}
\pi_{R}^{\xi}(\gamma) = e^{- \beta |\gamma|} Z_{R_{\gamma}^+, \Delta_{\gamma}^+}^{h, +} Z_{R_{\gamma}^-,\Delta_{\gamma}^-}^{h-1, -}\,.
\end{equation*} 
The partition functions $Z_{R, U}^{j,\pm}$ can be related to the corresponding partition function $\widehat Z_{R, U}^{j,\pm}$ of the model without floor via the following (taken from~\cite[Prop.~A.1]{CLMST16}):
\[
Z_{R, U}^{j,\pm} = \widehat Z_{R, U}^{j,\pm} \exp(-\widehat\pi(\varphi_o > j)|R| + o(1))\,,
\]
where $\widehat\pi$ is the infinite volume measure obtained as the thermodynamic limit of $\widehat{\pi}_{\Lambda}^0$, see \cite{BW82}.
It follows that
\[
\pi_{R}^{\xi}(\gamma) \propto e^{- \beta |\gamma|} \widehat Z_{R_{\gamma}^+, \Delta_{\gamma}^+}^{h,+} \widehat Z_{R_{\gamma}^-,\Delta_{\gamma}^-}^{h-1,-}\exp\Big(- \big(\widehat\pi(\varphi_o > h-1)-\widehat\pi(\varphi_o > h)\big)|R_{\gamma}^{-}|+o(1)\Big)\,,
\]
where we used $|R_{\gamma}^{+}| = |R| - |R_{\gamma}^{-}|$ and dropped the term $|R|\widehat\pi(\varphi_o > h)$, which is independent of $\gamma$.
It was proved in~\cite[Lemma~2.4]{CLMST16} that there exists a constant $c_{\infty}:= c_\infty(\beta)$ such that, for $j \geq 1$,
\[ 
\widehat\pi(\varphi_o \geq j) = c_{\infty}e^{-4\beta j} + O(e^{-6\beta j})\,.
\]
Define $a_L := \frac1{4\beta}\log L - \lfloor \frac1{4\beta}\log L\rfloor$, $\lambda := \lambda(L) = c_{\infty}e^{4\beta a_L}(1-e^{-4\beta})$ and $\lambda^{(n)} := \lambda e^{4\beta n}$. We have
\begin{multline*}
\widehat\pi(\varphi_o > h-1) - \widehat\pi(\varphi_o > h) = c_{\infty} e^{-4\beta h}(1-e^{-4\beta})+O(e^{-6\beta h}) \\
= c_{\infty}e^{4\beta a_L- \log L + 4\beta n}(1-e^{-4\beta})+O(e^{-6\beta h})
= \lambda^{(n)}/L + O(e^{-6\beta h})\,.
\end{multline*}
Putting everything together, we find
\begin{equation}\label{eqn:contourLawAreaTerm}
\pi_{R}^{\xi}(\gamma) \propto \widehat\pi_{R}^{\xi}(\gamma) \exp\Big(- \frac{\lambda^{(n)}}{L}A(\gamma)+o(1)\Big) \,,
\end{equation}
where $A(\gamma)=|R_{\gamma}^{-}|$ is the area under $\gamma$ in $R$.

\section{Proof of \cref{prop:oz-weights} and \cref{eqn:gen-oz-exp-decay}} \label{appendix:ornstein--Zernike}

\begin{proof}[Proof of \cref{prop:oz-weights}]
Fix $\delta \in (0,1)$.
We begin  by showing that for any $\y \in \fcone_{\delta}\setminus\{0\}$, we have $f(\h_{\y}) = 1$, where 
\[
f(\h):= \sum_{\Gamma \in \A} W^{\h}(\Gamma)\,.
\]
Let
\[
\wulffcone := \{s \h_{\y}: \y \in \fcone_{\delta} \setminus\{0\}, \ s \in [0, 1]\}
\]
and
\[
\extwulffcone := \{s \h_{\y}: \y \in \fcone_{\delta} \setminus\{0\}, \ s \in [1, 1+\tfrac{\beta\nu}{2}\|\h_{\y}\|_2^{-1}]\}\,,
\]
where $\nu$ is the same as in  \cref{prop:many-cone-points}. In words, $\wulffcone$ is  the sector of the Wulff shape $\Kb$ where the sector boundary (part of $\partial\Kb$) is $\{\h_\y : \y\in \fcone_{\delta}\}$, while $\extwulffcone$ lies right outside of $\Kb$ along the continuation of such sector (note that $\wulffcone \subset \Kb$ comes from $0 \in \Kb$ and the convexity of $\Kb$).

In what follows, we let $C:= C(\beta)>0$ denote a  constant that may depend on $\beta$ and may change from line-to-line.
We begin with two claims.
\begin{claim}\label{claim:oz-1}
For all $\h \in \wulffcone \cup \extwulffcone$, we have
\begin{align*}
\sum_{\Gamma \in \AL \cup \AR} W^{\h}(\Gamma) < \infty\,, 
\end{align*}
and the series converges uniformly over $\h$. In particular, $f(\h)$ is continuous on $\wulffcone \cup \extwulffcone$.
\begin{proof}[Proof of Claim~\ref{claim:oz-1}]
    We will use the following bound many times: from \cref{eqn:partition-fn-tau}, we have that for any $\ep \in (0,1)$ and for all $\beta>0$ sufficiently large, there exists a constant $C:= C(\ep, \beta)>0$ such that for all $\y \in\Z^2$
    \begin{align}
        \Gb(\y) \leq C e^{-\tau_{\beta}(\y) + \ep\|\y\|_1}\,. \label{eqn:partition-fn-tau-ub}
    \end{align}

    For $\h\in \wulffcone \cup \extwulffcone$, we have $\h\cdot \y - \tau_{\beta}(\y) \leq \frac{\beta\nu}{2}\|\y\|_2 \leq \frac{\beta\nu}{2}\|\y\|_1$ for all $\y \in \Z^2 \setminus \{0\}$. Then \cref{prop:many-cone-points} and \cref{eqn:partition-fn-tau-ub}  yield
    \begin{align}
        \sum_{\substack{\Gamma \in \AL \cup \AR \\ \Gamma: 0 \to\y}} W^{\h}(\Gamma) 
        \leq  e^{\h\cdot\y}  \Gb(\y \given |\Cpts(\Gamma)| = 0) \leq  C\exp\Big(-(\tfrac{\nu \beta}2 - \ep)\norm{\y}_1\Big) \,.
        \label{eqn:bound-no-cone-points}
    \end{align}    
    Take $\ep < \nu \beta /2$.
    Then since $\A \subset \AL$, the sequence of continuous functions 
    \[
        f_N(\h) := \sum_{\substack{\y \in \Z^2 \setminus \{0\} \\ \norm{\y}_1 \leq N}}  \sum_{\substack{\Gamma \in \A \\ \Gamma: 0 \to\y}} W^{\h}(\Gamma) 
    \]
    converges  uniformly to $f(\h)$, and so the claim follows.
\end{proof}
    \end{claim}
\begin{claim}\label{claim:oz-2}
    For $\h \in \wulffcone$, $f(\h)\leq1$. For $\h \in \extwulffcone$, $f(\h)\geq1$.
\end{claim}
\begin{proof}[Proof of Claim~\ref{claim:oz-2}]
Similar to \cref{eqn:bound-no-cone-points}, we have the following for any $\h \in \wulffcone \cup \extwulffcone$:
\begin{align}
    \sum_{\y \in \fcone_{\delta}} \sum_{\substack{\Gamma = [\gamma, \underline{\mathcal{C}}] \\  \gamma: 0 \to\y}} W^{\h}(\Gamma) \ind{|\Cpts(\Gamma)|<2} 
    \leq C+  C \sum_{\substack{\y\in \fcone_{\delta} \\ \norm{\y}_1 > 2\delta_0^{-1}}} e^{\h\cdot \y - \beta\nu \norm{\y}} \Gb(\y) <\infty \,.
    \label{eqn:oz-many-cone-points-1}
\end{align}
Using the factorization of $q(\Gamma)$~\eqref{eqn:q-factorization} and~\eqref{eqn:oz-many-cone-points-1}, we have
\begin{align*}
\sum_{\y\in \fcone_{\delta}} 
&e^{{\h}\cdot\y} \Gb(\y) \\
&\leq C + \sum_{\y\in \fcone}
\sum_{m \geq 1}
\sum_{\Gamma^{(L)}\in\AL}\sum_{\Gamma^{(1)}\,\dots,\Gamma^{(m)}\in\A}\sum_{\Gamma^{(R)}\in\AR}W^{\h}(\Gamma^{(L)})\bigg(\prod_{i=1}^m W^{\h}(\Gamma^{(i)})\bigg)W^{\h}(\Gamma^{(R)})\one_{\X(\Gamma) =\y} \\
&= C+ \bigg(\sum_{\Gamma^{(L)}\in\AL}W^{\h}(\Gamma^{(L)})\bigg)\bigg(\sum_{\Gamma^{(R)}\in\AR}W^{\h}(\Gamma^{(R)})\bigg)
\sum_{m\geq1} f(\h)^m \\
&=: C + B_L(\h)B_R(\h)\sum_{m\geq1}
f(\h)^m\,,
\end{align*}
where in the inequality we replaced $\fcone_{\delta}$ by the full cone $\fcone$ and in the next line we exchanged the sums and used the fact that an animal with at least one cone-point must necessarily satisfy $\X(\Gamma) \in \fcone$. 
Hence, we have the chain of inequalities
\begin{equation}\label{ineq:doubleineq}
\sum_{\y\in \fcone_{\delta}}e^{\h\cdot\y} \Gb(\y) \leq C + B_L(\h)B_R(\h)\sum_{m\geq1}
f(\h)^m \leq C + \sum_{\y\in \Z^2}e^{\h\cdot\y} \Gb(\y)\,.
\end{equation}
Since $B_L(\h),B_R(\h)<\infty$ by \cref{claim:oz-1}, finiteness of the middle term in \eqref{ineq:doubleineq} depends only on the convergence of $\sum_{m\geq1}f(\h)^m$. 

Now, if $\h$ is in the interior of $\wulffcone$ (and thus in the interior of $\Kb$), the rightmost term in \eqref{ineq:doubleineq} converges due to the second definition of $\Kb$, and thus $\sum_{m\geq1}f(\h)^m<\infty$, and thus $f(\h)<1$. By continuity of $f$, we get $f(\h)\leq 1$ on $\wulffcone$. On the other hand, if $\h$ is in the interior of $\extwulffcone$, the leftmost term in \eqref{ineq:doubleineq} diverges (by the same argument that showed the equivalence of the two definitions of $\Kb$ in subsection \ref{subsec:surface-tension-wulff}), and thus $\sum_{m\geq1}f(\h)^m=\infty$, and thus $f(\h)\geq1$.
\end{proof}

By continuity of $f$, it follows that $f(\h)=1$ on $\wulffcone\cap\extwulffcone$, i.e., $f(\h_{\y})=1$ for $\y \in \fcone_{\delta}$.
This proves that $\P^{\h_{\y}}$ defines a probability distribution on $\A$ for $\y \in \fcone_{\delta}$. 

Next, we show~\eqref{eqn:colinear-expectation}. For any $\y \in \fcone_{\delta}\setminus \{0\}$, define
\[
F(\y) := f(\h_{\y}) = \sum_{\Gamma\in\A}\P^{\h_{\y}}(\Gamma)\,.
\]
From above, we know that $F(\y) \equiv 1$. Since $\tau_\beta$ is analytic outside of the origin (\cref{prop:surface-tension-facts}), it is in particular twice-differentiable, and thus so is $\P^{\h_{\y}}(\Gamma)$ with $\nabla \P^{\h_{\y}}(\Gamma) = J_{\h_{\y}} \X(\Gamma) \P^{\h_{\y}}$, where $J_{\h_{\y}}$ is the Jacobian matrix of $\y \mapsto \h_{\y}$. Using the same argument as in \cref{claim:oz-1} and boundedness of $J_{\h_{\y}}$, we find that $\sum_{\Gamma\in\A}J_{\h_{\y}} \X(\Gamma)\Pa(\Gamma)$ converges uniformly for $\y\in \Z^2\setminus \{0\}$, and thus we can differentiate $F(\y)$ under the sum to get
\[
0 = \nabla F(\y) = \sum_{\Gamma\in\A}J_{\h_{\y}} \X(\Gamma) \P^{\h_{\y}}(\Gamma) = J_{\h_{\y}}  \sum_{\Gamma\in\A}\X(\Gamma) \P^{\h_{\y}}(\Gamma) = J_{\h_{\y}} \E^{\h_{\y}}[\X(\Gamma)]\,,
\]
i.e., $\E^{\h_{\y}}[\X(\Gamma)] \in \text{Ker}(J_{\h_{\y}})$. 
Differentiating the relation $\h_{\y} \cdot \y = \tau_\beta(y)$, we see that $\y \in \text{Ker}(J_{\h_{\y}})$, 
and so if $J_{\h_{\y}}$ is non-degenerate we get that $\E^{\h_{\y}}[\X(\Gamma)]$ is collinear to $\y$, i.e., $\E^{\h_{\y}}[\X(\Gamma)] = \alpha \y$ for some scalar $\alpha=\alpha(\y)$. Finally, the non-degeneracy condition can be removed: since $\tau_\beta$ is analytic (outside of the origin), the points at which $J_{\h_{\y}}$ is fully degenerate form a discrete set, and since collinearity must hold outside of such set and $\E^{\h_{\y}}[\X(\Gamma)]$ is continuous, we get the result for any $\y \in \fcone_{\delta}\setminus \{0\}$.

It remains to show the exponential tail decay~\eqref{eqn:oz-exp-decay}, which actually holds for the weights $W^{\h}$ for any $\h \in \Kb$. Below, $\nu_1>0$ and $\nu_2>0$  denote  constants independent of $\h$ and $\beta$, while $C_1>0$ and $C_2>0$ may depend on $\beta$ but not $\h$.
We express the left-hand side of~\eqref{eqn:oz-exp-decay} as follows:
\begin{align}
    \sum_{\substack{\Gamma \in \AL \cup \AR \\|\Gamma| \geq k}} W^\h(\Gamma)
    \leq 
    \sum_{\norm{\y}_1 \geq k/2}
    \sum_{\substack{ \Gamma \in \AL \cup \AR \\|\Gamma| \geq k\\ \Gamma: 0 \to \y}} W^\h(\Gamma) 
    + \sum_{\norm{\y}_1 < k/2} \sum_{\substack{|\Gamma| \geq k \\ \Gamma: 0 \to \y}} W^\h(\Gamma)\,,  \label{eqn:oz-exp-decay-boundthis}
\end{align}
The first term in the above sum can be bounded using~\eqref{eqn:partition-fn-tau-ub}, \cref{prop:many-cone-points}, and the bound $\h \cdot \y \leq \tau_{\beta}(\y)$ for any $\h \in \Kb$:
\[
    \sum_{\norm{\y}_1 \geq k/2}
    \sum_{\substack{ \Gamma \in \AL \cup \AR \\|\Gamma| \geq k\\ \Gamma: 0 \to \y}} W^\h(\Gamma) 
    \leq \sum_{\norm{\y}_1 \geq k/2} e^{\h \cdot \y} \Gb(\y \given |\Cpts(\Gamma)| = 0)
    \leq \sum_{\norm{\y}_1 \geq k/2}  e^{-(\nu \beta -\ep)\norm{\y}_1}  \leq C_1 e^{-\nu_1 \beta k}\,.
\]
The second term in~\eqref{eqn:oz-exp-decay-boundthis} is bounded by
\begin{multline*}
 \sum_{\norm{\y}_1 <k/2}e^{\h\cdot \y} \Gb(\y \given |\gamma| \geq k) 
\leq \sum_{\norm{\y}_1 < k/2}  e^{\tau_{\beta}(\y)} \Big(c e^{-\nu_0 \beta (k/\norm{\y}_1) \norm{\y}_1} \Big)  \Gb(\y) \\
\leq   C_2 e^{-\nu_0 \beta k}\sum_{\norm{\y}_1 < k/2} e^{\ep \|\y\|_1} 
\leq  C_2 e^{-\nu_2 \beta k}   
\end{multline*}
where the first inequality uses~\eqref{eqn:ist-length-bound} and  $\h  \cdot \y \leq \tau_{\beta}(\y)$, and  the second inequality uses~\eqref{eqn:partition-fn-tau-ub}.
\end{proof}

\begin{proof}[Proof of \cref{eqn:gen-oz-exp-decay}]
Define the partition function over contours $\gamma \subset \Z^2$ such that $\gamma:0 \to \y$:
\[
    \Gb_D^{\full}(\y) := \sum_{ \gamma : 0 \to \y } q_D(\gamma)\,.
\]
\cref{eqn:modified-weight-comparison} gives us
    $q_D(\gamma) \leq \exp(6 e^{-\chi \beta}|\gamma|)q(\gamma)$, from which we find
\begin{align*}
    \Gb_D^{\full}(\y) \leq \sum_{\gamma:0\to\y} e^{6 e^{-\chi \beta}|\gamma|}q(\gamma) \leq e^{6.6e^{-\chi\beta}\|\y\|_1}\Gb(\y) + \sum_{\substack{\gamma: 0 \to \y \\ |\gamma| \geq 1.1 \|\y\|_1}} e^{6 e^{-\chi \beta}|\gamma|}q(\gamma) 
    \leq e^{7e^{-\chi\beta}\|\y\|_1} \Gb(\y)\,,
\end{align*}
where in the final inequality we used~\cref{eqn:ist-length-bound}. We then obtain the analogue of \cref{eqn:partition-fn-tau-ub} by applying \cref{eqn:partition-fn-tau-ub} to the above: for any $\ep \in (0,1)$, there exists $\beta_0:= \beta_0(\ep)>0$ such that for all $\beta \geq \beta_0$ and for all $\y \in \Z^2$,
\begin{align}
    \Gb_D^{\full}(\y) \leq Ce^{-\tau_{\beta}(\y) + \ep\|\y\|_1}
\end{align}
for some $C:= C(\ep,\beta)>0$.
Similarly, we obtain the length and contour cone-points bound \cref{lem:ist-length-cpts} and the animal cone-points bound \cref{prop:many-cone-points} with $\Gb_D^{\full}(\y)$ replacing $\Gb(\y)$ there. With these bounds, we may re-run the  proof of \cref{eqn:oz-exp-decay}, replacing $W^{\h}(\Gamma)$ with $e^{\h\cdot \X(\Gamma)} q_D(\Gamma)$.
\end{proof}
\begin{Backmatter}

\paragraph{Acknowledgments}
We thank Milind Hegde for a useful discussion.

\paragraph{Funding statement}
Y.H.K.\ acknowledges the support of the NSF Graduate Research Fellowship 1839302.
E.L.\ acknowledges the support of NSF DMS-2054833. 

\paragraph{Competing interests}
None.

\end{Backmatter}

\end{document}